\documentclass{amsart}

\usepackage[T1]{fontenc}
\usepackage[lf]{Baskervaldx} % lining figures
\usepackage[bigdelims,vvarbb]{newtxmath} % math italic letters from Nimbus Roman
\usepackage[cal=boondoxo]{mathalfa} % mathcal from STIX, unslanted a bit

\usepackage[all]{xy}
\usepackage{latexsym}
\usepackage{amssymb}
\usepackage{amsmath}
\usepackage{amsthm}
\usepackage{amscd}
\usepackage{fancyhdr}
\usepackage{a4wide}
\usepackage{mathrsfs}
\usepackage{graphicx}
\usepackage{color}
\usepackage[colorlinks,final,hyperindex]{hyperref}
\usepackage[noabbrev,capitalize]{cleveref}
\usepackage{tikz}
\usepackage{tikz-cd}
\usetikzlibrary{decorations.pathmorphing,arrows,arrows.meta,calc,shapes,shapes.geometric,decorations.pathreplacing}
\tikzcdset{arrow style=tikz, diagrams={>=stealth}}
\usepackage{comment}
\usepackage{tensor}

\definecolor{Chocolat}{rgb}{0.36, 0.2, 0.09}
\definecolor{BleuTresFonce}{rgb}{0.215, 0.215, 0.36}
\hypersetup{citecolor=BleuTresFonce, linkcolor=Chocolat}

\theoremstyle{plain}
\newtheorem{thm}{Theorem}[section]
\newtheorem{proposition}[thm]{Proposition}
\newtheorem{theorem}[thm]{Theorem}
\newtheorem{corollary}[thm]{Corollary}

\newtheorem{lemma}[thm]{Lemma}

\theoremstyle{definition}
\newtheorem{definition}[thm]{Definition}

\newtheorem{remark}[thm]{\sc Remark}
\newtheorem{exam}[thm]{\sc Example}

\newtheorem{notation}{\sc Notation}

\def\Sy{\mathbb{S}}
\def\k{\mathbb{k}}
\def\dgVect{\mathsf{dgVect}}
\def\SCcomod{\mathsf{mon}\textsf{-}\S\calC\textsf{-}\mathsf{comod}}
\def\dgSCcomod{\mathsf{dg\ mon}\textsf{-}\S\calC\textsf{-}\mathsf{comod}}
\def\M{\mathcal{M}}
\def\N{\mathcal{N}}
\newcommand{\Hom}{\ensuremath{\mathrm{Hom}}}
\renewcommand{\hom}{\ensuremath{\mathrm{hom}}}
\newcommand{\tens}[1]{%
 \mathbin{\mathop{\otimes}\limits_{#1}}%
 }
\def\T{\mathrm{T}}
\def\S{\mathrm{S}}
\def\V{\mathcal{V}}
\def\I{\mathcal{I}}

\def\oP{\overline{\calP}}
\def\oC{\overline{\calC}}
\def\eps{\varepsilon}
\def\sq{\, \square\,}
\def\IL{\mathrm{IL}}
\def\oH{\underset{\mathrm{H}}{\otimes}}
\def\oG{\underset{\mathrm{G}}{\otimes}}
\def\ibt{\underset{\scriptscriptstyle (1,1)}{\boxtimes}}
\def\libt{\lhd_{(*)}}
\def\ribt{\tensor[_{(*)}]{\rhd}{}}
\def\Bider{\ensuremath{\mathrm{Bider}}}
\def\Codiff{\ensuremath{\mathrm{Codiff}}}
\def\Bidiff{\ensuremath{\mathrm{Bidiff}}}
\def\sgn{\ensuremath{\mathrm{sgn}}}
\def\H{\mathrm{H}}
\def\PP{\mathrm{P}}
\def\II{\mathrm{I}}

\newcommand{\ac}{{\scriptstyle \text{\rm !`}}}

\def\ot{\otimes}
\def\L{\mathrm{L}}

\def\P{\mathcal{P}}
\def\O{\mathcal{O}}
\def\calS{\mathcal{S}}

\def\id{\mathrm{id}}

\def\gg{\mathrm{g}}
\def\B{\mathrm{B}}
\def\d{\mathrm{d}}

\def\Id{\mathrm{Id}}

\newcommand{\ZZ}{\mathbb{Z}}
\newcommand{\NN}{\mathbb{N}}

\newcommand{\G}{\mathcal{G}}
\newcommand{\GLev}{\mathcal{G}_{\mathrm{lev}}}
\newcommand{\Gs}{\mathrm{G}}
\newcommand{\GsLev}{\mathrm{G}_{\mathrm{lev}}}
\newcommand{\bGs}{\overline{\mathrm{G}}}

\renewcommand{\Bar}{\mathbf{B}}

\newcommand{\Cobar}{\mathbf{\Omega}}
\newcommand{\calP}{\mathcal{P}}
\newcommand{\calC}{\mathcal{C}}
\newcommand{\calD}{\mathcal{D}}
\newcommand{\bbS}{\mathbb{S}}

\newcommand{\ncG}{\G^{\mathrm{nc}}}
\newcommand{\ncGLev}{\GLev^{\mathrm{nc}}}

\newcommand{\Sbimod}{\Sy\mbox{-}\mathsf{bimod}}

\newcommand{\Bij}{\mathsf{{Bij}}}

\newcommand{\dgprop}{\mathsf{dg \ properads}}
\newcommand{\dgcoprop}{\mathsf{dg\ coproperads}}

\newcommand{\End}{\mathrm{End}}
\newcommand{\lev}{\mathrm{lev}}

\newcommand{\PHI}{\mathrm{PHI}}
\newcommand{\HHI}{\mathrm{HHI}}
\newcommand{\PHH}{\mathrm{PHH}}

\newcommand{\Tw}{\mathrm{Tw}}

\newcommand{\op}{^\mathsf{op}}

\newcommand{\frakf}{\mathfrak{f}}
\newcommand{\frakp}{\mathfrak{p}}

\newcommand{\Frob}{\mathrm{Frob}}
\newcommand{\IBL}{\mathrm{IBL}}

\def\Cop#1#2{\tensor*[_{#1}]{\Delta}{_{#2}}}

\title{Properadic homotopical calculus}

\date{\today}

\author{Eric Hoffbeck}
\address{LAGA, CNRS, UMR 7539, Universit\'e Paris 13, Sorbonne Paris Cit\'e, Universit\'e Paris 8, 99 Avenue JB Cl\'ement, 93430 Villetaneuse, France.}

\email{hoffbeck@math.univ-paris13.fr}
\author{Johan Leray}
\email{leray@math.univ-paris13.fr}
\author{Bruno Vallette}
\email{vallette@math.univ-paris13.fr}



\thanks{The authors were supported by the Institut Universitaire de France (IUF) and ANR ChroK (ANR-16-CE40-0003). The second author is financed by a postdoctoral allocation given by DIM Math Innov -- R\'egion \^Ile de France.}

\begin{document}

\maketitle

\begin{abstract}
In this paper, we initiate the generalisation of the operadic calculus which governs the properties of homotopy algebras to a properadic calculus which governs the properties of homotopy gebras over a properad. 
In this first article of a series, we generalise the seminal notion of $\infty$-morphisms and the ubiquitous homotopy transfer theorem. 
As an application, we recover the homotopy properties of involutive Lie bialgebras developed by Cieliebak--Fukaya--Latschev and we produce new explicit formulas. 
\end{abstract}

\setcounter{tocdepth}{1}
\tableofcontents

\section*{Introduction}
There are basically two ways to do algebraic homotopy theory: one can work on a conceptual level using model categories and higher categories or one can use the more explicit operadic calculus. 

\bigskip

Let us see how this  works. Suppose that one is interested in understanding the homotopical properties of a category of algebras of type $\calP$. This means that one would like to describe their behaviour under quasi-isomorphisms, i.e. the morphisms which induce isomorphisms on the level of homology. The main issue is that being quasi-isomorphic is not an equivalence relation: quasi-isomorphisms are not invertible in general. This problem is similar to the invertibility of $1-x$: it is not invertible in the space of degree $1$ polynomials but it is invertible if one can consider series where $(1-x)^{-1}=1+x+x^2+\cdots$. Indeed, there is a higher notion of morphism, called \emph{$\infty$-morphism}, made up of a collection of maps and such that any $\infty$-quasi-isomorphism of $\calP$-algebras, like a quasi-isomorphism, admits an $\infty$-morphism in the opposite direction which realizes the inverse homology isomorphism. 

\bigskip

One of the first seminal occurence of such a notion of $\infty$-morphism can be found in the groundbreaking proof of M. Kontsevich of the deformation quantization of Poisson manifolds \cite{Kontsevich03}. In order to prove an equivalence between two deformation theories, he proved the formality of the differential graded Lie algebra of polydifferential operators of a Poisson manifold. But  he did not directly prove the existence of a zig-zag of quasi-isomorphisms from it to its homology; instead, he constructed an $\infty$-morphism from the latter to the former which extends the Hochschild--Kostant--Rosenberg map. 

\bigskip

Another instance of the use of operadic calculus lies in the description of the homotopy categories of differential graded $\calP$-algebras with respect to their quasi-isomorphisms. When $\calP$ is an operad, one can transfer the cofibrantly generated projective model category structure on differential graded vector spaces to differential graded $\calP$-algebras. This application of Quillen's seminal result does not help much: we get that the homotopy category of differential graded $\calP$-algebras is equivalent to the category of retracts of quasi-free $\calP$-algebras equipped with a suitable filtration on its space of generators (up to some homotopy equivalence on morphisms). Instead, one can use the \emph{bar-cobar adjunction} to work with a Quillen equivalent category of differential graded $\calC$-coalgebras, where $\calC$ is either the Koszul dual cooperad of $\calP$, when it is Koszul, or its bar construction in general. In each case, the homotopy category of 
differential graded $\calP$-algebras is equivalent to the category of quasi-free dg $\calC$-coalgebras (up to some homotopy equivalence on morphisms): this category is nothing but the category of \emph{homotopy $\calP$-algebras}, also known as \emph{$\calP$-algebras up to homotopy}, equipped with their $\infty$-morphisms \cite{Hinich01, LefevreHasegawa03, Vallette14}.

\bigskip

Last but not least, let us consider a contraction of a chain complex onto another one, for instance its homology, and a $\calP$-algebra structure on the first one. One can transfer it to the second space in the form of a homotopy $\calP$-algebra structure and one can extend the contraction quasi-isomorphisms into $\infty$-quasi-isomorphisms which make the initial structure and the transferred structure being homotopy equivalent. This result, called the \emph{homotopy transfer theorem} is ubiquitous in mathematics, let us cite just a few examples. Applied to modules over the algebra of dual numbers, one gets the notion of spectral sequences and their convergence theorem \cite[Chapter~10]{LodayVallette12}; this can be applied to get the definition of cyclic homology \cite{Kassel90}. Realisations of higher Massey products in algebraic topology are produced in this way by considering the transferred $A_\infty$-algebra structure from the associative cup product on the singular cochain complex of a topological space to this cohomology. Applying the homotopy transfer theorem to unimodular Lie bialgebras allows one to recover the Batalin--Vilkovisky formalism and its celebrated Feynman diagrams (see \cite{Merkulov10}).

\bigskip 

The above situation is now well established for algebraic structures equipped with products made up of several inputs but one output; we refer the reader to \cite{LodayVallette12} and references therein. 
For algebraic structures 
made up of products and coproducts,
that is with several inputs and several outputs, the situation requires further work. First of all, there is a suitable object which encodes them: properads, for which the Koszul duality theory was developed in \cite{Vallette07} and the deformation theory developed in \cite{MerkulovVallette09I, MerkulovVallette09II}. In \emph{loc. cit.}, a representation of a (possibly colored) properad $\calP$ is called a \emph{$\calP$-gebra} following J.-P. Serre \cite{Serre93}, since it includes all notions such as algebras, coalgebras, bialgebras, modules, comodules, bimodules, etc. 
(Dioperads only encode the genus 0 part of the composition of operations and props do not receive a  Koszul duality so far.)  This allows one to get the notion of \emph{homotopy $\calP$-gebra}. In the operadic case, there are four equivalent definitions of $\calP$-algebras up to homotopy, which all together form a so called \emph{Rosetta Stone} \cite[Section~10.1.9]{LodayVallette12}. Before the present paper, only three of them have been shown to hold on the properadic level. 

\bigskip

Recall that the fourth definition is based on the abovementioned bar-cobar adjunction between $\calP$-algebras and $\calC$-coalgebras. Such a construction \emph{cannot hold} as such on the level of $\calP$-gebras since there does not exist a free $\calP$-gebra in general. The first goal of the present paper is to overcome this difficulty: with the help of some additional monoidal structures for $\bbS$-bimodules, the underlying objects of properads, we introduce a suitable new algebraic notion extending comodules over a cooperad (\cref{def:dgMonSCComod}), to which belong $\calC$-coalgebras. This allows us to provide the literature with the required fourth equivalent definition of homotopy $\calP$-gebra, thus completing the properadic Rosetta stone (\cref{thm:CompleteRosetta}). 

\bigskip

Since the notion of $\infty$-morphism is defined via the fourth definition of homotopy $\calP$-algebras, we use the above new theory to introduce a meaningful notion of \emph{$\infty$-morphism} for homotopy $\calP$-gebras (\cref{def:InftyMor}). This new notion is shown to share the same nice properties than its operadic ancestor: description of the invertible $\infty$-morphisms (\cref{prop:Inverse}), homology inverse of $\infty$-quasi-isomorphisms (\cref{thm:InvInfQI}) and obstruction theory (\cref{thm:Obstr}). 

\bigskip 

Finally, we prove a homotopy transfer theorem for gebras over a properad (\cref{thm:HTT}). 
Recall that the existence of transferred structures encoded by cofibrant prop(erad)s was established by B. Fresse in \cite{Fresse10ter} using abstract model category arguments and without $\infty$-morphisms. 
Recall also that a genus $0$ homotopy transfer theorem was proved by S. Merkulov in \cite{Merkulov10}, but the present treatment includes all genera. The way we establish this general homotopy transfer theorem is also new: it is  based first on an effective  homotopy transfer theorem universally associated to any contraction and then on a functorial property with respect to coproperad maps. 

\bigskip 

The present general theory includes as a particular case of application all the algebraic properties of homotopy involutive Lie bialgebras as developed by K. Cieliebak, K. Fukaya, and J. Latschev in \cite{CFL11}. Moreover,  we provide explicit new formulae, like the one for the homotopy transfer theorem (\cref{thm:HTTIBL}), whereas in \emph{loc. cit.} this kind of results is obtained via obstruction theory. 

\bigskip

The present work will be followed by a series of papers in which we will study extensively 
 the algebraic properties of $\infty$-isomorphisms and the homotopical properties of $\infty$-quasi-isomorphisms. 
For instance, we will integrate the properadic convolution Lie  algebra with $\infty$-isomorphisms using a new kind of exponential map based of graphs. 
In order to provide further homotopical properties of $\infty$-quasi-isomorphisms, we will introduce new model structures and simplicial enrichements for homotopy $\calP$-gebras. This latter work will allow one to apply the new methods of Ginot--Yalin \cite{GinotYalin19} coming from derived algebraic geometry  to study of the moduli spaces of $\calP$-gebras up to $\infty$-quasi-isomorphisms.

\subsection*{Layout} 
The paper is organised as follows. 
The first section describes various monoidal structures on the category of $\Sy$-bimodules. 
In the second section, we recall the notions of properads and coproperads together with their bar and cobar constructions. 
In the third section, we introduce a suitable notion of higher morphisms, called $\infty$-morphisms, for homotopy gebras. 
The forth section deals with the homotopy transfer theorem for homotopy gebras. 
In the fifth section, we establish the obstruction theory for $\infty$-morphisms. 
In the last section, we detail some examples of gebras to which the present theory applies. 

\subsection*{Conventions} 
We use the conventions of  \cite{LodayVallette12}  for operads and of \cite{Vallette07, MerkulovVallette09I, MerkulovVallette09II} for properads. Let $\k$ be a ground field  of characteristic $0$.  
We work over the underlying category of differential $\ZZ$-graded vector spaces, denoted by $\dgVect$. 
We use the homological degree convention, for which differentials have degree $-1$.  
We denote the symmetric groups by $\bbS_n$. 

\subsection*{Acknowledgements} We would like to express our deep appreciation to Kai Cieliebak, Kenji Fukaya, Janko Latschev, Geoffrey Powell, Salim Rivi\`ere, and Dennis Sullivan for enlightening discussions. B.V. would like to thank the Laboratoire J.A. Dieudonn\'e of the University Nice Sophia Antipolis for the particularly generous hospitality.

\section{Monoidal structures on symmetric bimodules}

\subsection{$\Sy$-bimodules}The notion of an $\Sy$-bimodule is meant to encode operations with $n$ inputs and $m$ outputs, for $n, m\in \NN$. 

\begin{definition}[$\Sy$-bimodule]
A collection 
$\{\M(m,n)\}_{m,n\in \mathbb{N}}$ of (graded, differential graded) $\Sy_m\times \Sy_n^{\mathrm{op}}$-modules is called an (graded, differential graded) \emph{ $\Sy$-bimodule}.
\end{definition}

From now on, we will only work in the sub-category of \emph{left reduced} $\Sy$-bimodules, i.e. $\M(m,n)=0$, for $m=0$. We denote this category by $\Sbimod$. For simplicity, this section is written in this category, but all the constructions and results hold \emph{mutatis mutandis} on the differential graded level. 

\begin{remark}
One can consider the category $\Bij$ whose objects are finite sets and whose morphisms are bijections. 
The notion of an $\Sy$-bimodule is equivalent to a 
presheaf $\M : \Bij\times\Bij\op \to \mathsf{Vect}$ with value in the category of vector spaces. For more details about this point of view, we refer to reader to \cite[Section 3]{Leray19protoI}. 
\end{remark}

\subsection{The composition product}
\begin{definition}[Composition product] \label{def::prod_compo_bimod}
	The \emph{composition product} of  $\Sy$-bimodules is the $\Sy$-bimodule  
	defined by 
	\[\left( \M \sq \N\right)(m,n)\coloneqq 
	\left\{
	\begin{array}{ll}
	\bigoplus_{k\in \NN^*} \M(m,k) 	\tens{\Sy_k} \N(k, n)\ , & \text{for}\  n>0\ ,\\
	\rule{0pt}{10pt}\bigoplus_{k\in \NN^*} \M(m,k) 	\tens{\Sy_k} \N(k, 0) \, \oplus\, \M(m,0)    \ , & \text{for}\ n=0\ .
	\end{array}
	\right.\]
\end{definition}

	The product $\square$ can be thought pictorially as of grafting all outputs of an element of $\N$, drawn on top, on all inputs of an element of $\M$, drawn on the bottom,  as in \cref{fig::produit_vertical}.

\begin{figure*}[h]
	\begin{tikzpicture}[scale=0.5,baseline=1]
	\draw[thick]	(8,0) -- (8,1);
	\draw[thick]	(0,0) -- (0,1);
	\draw[thick]	(1.5,0) -- (1.5,1);
	\draw[thick]	(3,0) -- (3,1);	
	\draw[thick]	(0,4) -- (0,4.5);		
	\draw[thick]	(2,4) -- (2,4.5);			
	\draw[thick]	(4,4) -- (4,4.5);				
	\draw[thick]	(8,4) -- (8,4.5);				
	\draw (0,0) node[below] {\tiny{$1$}} ; 
	\draw (1.5,0) node[below] {\tiny{$2$}} ; 	
	\draw (3,0) node[below] {\tiny{$3$}} ; 			
	\draw (0,4.5) node[above] {\tiny{$1$}};
	\draw (2,4.5) node[above] {\tiny{$2$}};
	\draw (4,4.5) node[above] {\tiny{$3$}};
	\draw (8,4.5) node[above] {\tiny{$n$}};
	\draw[thick] (6,4.5) node[above] {\tiny{$\cdots$}};
	\draw[thick] (5.5,0) node[below] {\tiny{$\cdots$}};
	\draw[thick] (8,0) node[below] {\tiny{$m$}};			
	\draw[thick] (1,1)--(1,4);
	\draw[thick] (2.5,3.5) to[out=270,in=90] (5.5,1) ;
	\draw[draw=white,double=black,double distance=2*\pgflinewidth,thick] (4,3.5) to [out=270,in=90]  (2.5,1);
	\draw[thick] (5.5,3.5) to [out=270,in=90]  (7,1);
	\draw[draw=white,double=black,double distance=2*\pgflinewidth,thick] (7,3.5) to [out=270,in=90]  (4,1);
	\draw[thick] (0.6,2.7) node {\tiny{$1$}};
	\draw[thick] (2.2,2.7) node {\tiny{$2$}};
	\draw[thick] (7.3,2.7) node {\tiny{$k$}};
	\draw[fill=white] (-0.3,0.4) rectangle (8.3,1.1);
	\draw[thick] (4,0.7) node {\small{$\mu$}};
	\draw[fill=white] (-0.3,3.4) rectangle (8.3,4.1);
	\draw[thick] (4,3.7) node {\small{$\nu$}};
	\end{tikzpicture}
	\caption{An element of $\M \sq \N$.}
		\label{fig::produit_vertical}
\end{figure*}

We consider the  $\Sy$-bimodule $\k[\Sy]$ defined by $\k[\Sy](m,n)\coloneqq \k[\Sy_n]$, for $m=n$ in $\NN^*$, and $0$ otherwise. 

\begin{definition}[Inner hom]
The \emph{inner hom} is the $\Sy$-bimodule defined by 
\[\hom\left( \M, \N\right)(m,n)
\coloneqq 
	\left\{
	\begin{array}{ll}
			\displaystyle{	\prod_{k\in \NN} \Hom_{\Sy_k}\left( \M(n, k), \N(m,k) \right)} \ , & \text{for}\  n>0\ ,\\
			\rule{0pt}{10pt}\N(m,0)\ , & \text{for}\ n=0\ .
	\end{array}\right.
 \]
\end{definition}

\begin{proposition}\label{prop:closedmono}
	The category $\big(\Sbimod,\square, \k[\Sy], \hom\big)$ forms  a bicomplete right closed monoidal category.
\end{proposition}

\begin{proof}
The various axioms are straightforward to check. Notice that both terms of the associativity relation 
$\left( \M \sq \N\right)\sq \O\cong \M \sq \left(\N\sq \O\right)$ are made up of 
vertical graphs with 3 levels labeled from top to bottom by one element of $\O$, $\N$, and $\M$ respectively, 
vertical graphs with 2 levels labeled from top to bottom by one element of $\N$ with input arity $0$ and $\M$ respectively, 
vertical graphs with 1 level labeled by one element of $\M$ with $0$ input. For every $\Sy$-bimodule $\M$, the functor $-\sq \M$ is left adjoint to the functor $\hom(\M, -)$. 
\end{proof}

\subsection{The tensor product}
We extend the tensor product of $\Sy$-modules \cite[Section~5.1.3]{LodayVallette12} to $\Sy$-bimodules as follows. 

\begin{definition}[Tensor product] \label{def::prod_conca_bimod}
	The \textit{tensor product} of $\Sy$-bimodules is the $\Sy$-bimodule 
	defined by
	\[\left( \M \otimes \N\right)(m,n)\coloneqq \bigoplus_{m_1+m_2=m\atop n_1+n_2=n} 
	\k[\Sy_m] \tens{\Sy_{m_1}\times \Sy_{m_2}} \M(m_1,n_1) \otimes 
	\N(m_2, n_2)\tens{\Sy_{n_1}\times \Sy_{n_2}} \k[\Sy_n] \ .\]
\end{definition}

This monoidal product can be thought of as an horizontal composition, as depicted in \cref{fig::produit_horizontal}.

\begin{figure*}[h]
\begin{tikzpicture}[scale=0.5,baseline=1]
\draw[thick] (0,-2) to[out=90,in=270] (2,0) -- (2,1.5) to[out=90,in=270] (0,3.5);
\draw[thick] (5,0) node[below] {\tiny{$\cdots$}};
\draw[thick] (5,1.5) node[above] {\tiny{$\cdots$}};
\draw[thick] (6,-2)  to[out=90,in=270]  (9.3,0) -- (9.3,1.5)  to[out=90,in=270] (4,3.5);
\draw[thick] (12.1,0) node[below] {\tiny{$\cdots$}};
\draw[thick] (12.1,1.5) node[above] {\tiny{$\cdots$}};
\draw[thick] (12.3,-2) to[out=90,in=270] (14.3,0) -- (14.3,1.5) to[out=90,in=270] (11,3.5);
\draw[draw=white,double=black,double distance=2*\pgflinewidth,thick]	(10,-2)  to[out=90,in=270]  (8,0) -- (8,1.5)  to[out=90,in=270] (13,3.5);
\draw[draw=white,double=black,double distance=2*\pgflinewidth,thick] (4,-2) to[out=90,in=270] (0,0) -- (0,1.5) to[out=90,in=270] (6,3.5);
\draw[fill=white] (-0.3,0.4) rectangle (8.3,1.1);
\draw[fill=white] (9.1,0.4) rectangle (14.5,1.1);
\draw[thick] (12,0.7) node {\small{$\nu$}};
\draw[thick] (4,0.7) node {\small{$\mu$}};
\end{tikzpicture}~.
\caption{An element of $\M \otimes \N$.}
\label{fig::produit_horizontal}
\end{figure*}

\begin{remark}
	In the category of  $\Sy$-bimodules, the tensor product does not admit a unit, because our $\bbS$-bimodules are left reduced.  
\end{remark}

\begin{proposition}
	The category $(\Sbimod,\otimes)$ forms  a symmetric monoidal category without a unit.
\end{proposition}

\begin{proof}
The various axioms are straightforward to check.
\end{proof}

\begin{notation}\label{eq:interchange_law}
The compatibility between the two monoidal structures $\square$ and $\otimes$ is encoded into the  following natural transformation 
\begin{equation*}
	\IL_{\M,\M',\N,\N'}\  :\  (\M\:\square \:\M')\otimes (\N\:\square\:\N')\hookrightarrow
	(\M\otimes \N)\:\square\:(\M'\otimes \N')\ ,
\end{equation*}
	 called \emph{the interchange law}.
\end{notation}
	 
\begin{proposition}[{\cite[Corollary~3.9]{Leray19protoI}}]\label{prop:nuMonMonoiCat}
The category of non-unital $\otimes$-monoids $\big(\mathsf{nuMon}_\otimes, \square,\k[\Sy] \big)$ and  the category of non-unital commutative $\otimes$-monoids $\big(\mathsf{nuComMon}_\otimes, \square,\k[\Sy] \big)$ are monoidal categories. 
\end{proposition}

\begin{proof}
	The proof is similar to the proof of \cite[Corollary~3.9]{Leray19protoI}, except that our underlying category is the one of left reduced $\bbS$-bimodules instead of the category of reduced $\bbS$-bimodules (on both sides). The key arguments in the proof are the properties of the interchange law and its compatibility with the braiding.
\end{proof}

\begin{remark}
The arguments of this proof belong to the methods of $\otimes$-braided duoidal categories of \cite{AguiarMahajan06} without the structure maps involving the unit for $\otimes$. 
The interchange law satisfies the associativity axiom \cite[Proposition~2]{Vallette08} 
of the definition of a \emph{lax 2-monoidal category}. This property ensures that the $\square$-product of two non-unital $\otimes$-monoids is again a $\otimes$-monoid. 
This notion of lax 2-monoidal category was refined in \cite[Chapter~6]{AguiarMahajan06} under the name \emph{2-monoidal category} (which is different from the notion of 2-monoidal category introduced in \cite[Section~1.3]{Vallette08}). It is now also dubbed \emph{duoidal category}, see \cite{BookerStreet13}.
\end{remark}

\begin{definition}[Free non-unital monoid]
The free non-unital $\otimes$-monoid on an $\Sy$-module $\M$ is given by 
\[\T \M\coloneqq \bigoplus_{k\in \NN^*} \M^{\otimes k}\]
and the free non-unital commutative  $\otimes$-monoid is given by 
\[\S\M\coloneqq \bigoplus_{k\in \NN^*} \left(\M^{\otimes k}\right)_{\Sy_k}\ ,\]
with product given by the concatenation. 
\end{definition}

\cref{prop:nuMonMonoiCat} prompts the following question: what kind of non-unital $\otimes$-monoid is 
$\S\M \sq \S\N$? The next section shows that it is actually free. 

\subsection{Connected composition product}\label{sect::produit_connexe_Sbimod}

We denote $a$-tuples of integers by $\overline{i}\coloneqq(i_1, \ldots, i_a)$. We consider 
$\Sy_{\overline{i}}\coloneqq \Sy_{i_1}\times \cdots \times \Sy_{i_a}$ and 
$\M\big(\overline{j}, \overline{i}\big)\coloneqq \M(j_1, i_1)\otimes \cdots \otimes \M(j_a, i_a)$\ .

\begin{definition}[Connected permutation]
Let $\overline{k}$ and $\overline{l}$ be respectively a $b$-tuple and an $a$-tuple of positive integers satisfying $N\coloneqq\big| \overline{k}\big|=\big|\overline{l}\big|$. A \emph{$\big(\overline{k},\overline{l}\big)$-connected permutation} is a  permutation $\sigma$ of $\Sy_N$ such that the graph of a geometric representation of $\sigma$ is connected if one merges the inputs labeled by 
$l_1+\cdots+l_i+1,\ldots, l_1+\cdots+l_i+l_{i+1}$, for $0\leqslant i \leqslant a-1$, and the outputs labeled by 
$k_1+\cdots+k_i+1,\ldots,k_1+\cdots+k_i+k_{i+1}$ for $0\leqslant i \leqslant b-1$. 
By convention, the only element of $\Sy_0$ is a connected permutation. 
We denote the associated set of permutations by $\Sy^c_{\overline{k}, \overline{l}}$. 
\end{definition}

\begin{definition}[Connected composition product] \label{def::prod_compo_conn_bimod}
	The \emph{connected composition product} of  $\Sy$-bimodules is the $\Sy$-bimodule 
	defined by 
\[
		(\M\boxtimes \N)(m,\, n) \coloneqq
	\left\{
	\begin{array}{ll}
			 \displaystyle \bigoplus_{N\in \mathbb{N}^*} \left( \bigoplus_{\overline{l},\, \overline{k},\, \overline{j},\, \overline{i}}\k[\Sy_m]
	\tens{\Sy_{\overline{l}}}
	\M\big(\overline{l},\, \overline{k}\big)
	\tens{\Sy_{\overline{k}}} \k[\Sy^c_{\overline{k}, \overline{l}}]
	\tens{\Sy_{\overline{j}}} 
	\N\big(\overline{j},\, \overline{i}\big)
	\tens{\Sy_{\overline{i}}} \k[\Sy_n]
	\right)_{\Sy_b^{\textrm{op}}\times \Sy_a}\ , & \text{for}\  n>0\ ,\\
	\rule{0pt}{10pt}	 \displaystyle \bigoplus_{N\in \mathbb{N}^*} \left( \bigoplus_{\overline{l},\, \overline{k},\, 		 	\overline{j}}\k[\Sy_m]
	\tens{\Sy_{\overline{l}}}
	\M\big(\overline{l},\, \overline{k}\big)
	\tens{\Sy_{\overline{k}}} \k[\Sy^c_{\overline{k}, \overline{l}}]
	\tens{\Sy_{\overline{j}}} 
	\N\big(\overline{j},\, \overline{0}\big)
	\right)_{\Sy_b^{\textrm{op}}\times \Sy_a} \oplus \M(m,0)\ , & \text{for}\ n=0\ .
	\end{array}\right.
\]
where the second direct sum runs over the $b$-tuples $\overline{l}$, $\overline{k}$ and the
$a$-tuples $\overline{j}$, $\overline{i}$ such that  $|\overline{l}|=m$, $|\overline{k}|=|\overline{j}|=N$, $|\overline{i}|=n$ and where the coinvariants correspond to the following action of $\Sy_b^{\textrm{op}}\times \Sy_a$ :
\begin{eqnarray*}
&&\theta \otimes \mu_1\otimes \cdots \otimes \mu_b \otimes \sigma
\otimes \nu_1 \otimes \cdots \otimes \nu_a \otimes\omega
  \sim \\
&&\theta \,\tau^{-1}_{\overline{l}} \otimes \mu_{\tau^{-1}(1)}\otimes \cdots
\otimes \mu_{\tau^{-1}(b)} \otimes
 \tau_{\overline{k}}\,
 \sigma \, \upsilon_{\overline{j}} \otimes \nu_{\upsilon(1)} \otimes \cdots \otimes \nu_{\upsilon(a)}
 \otimes \upsilon^{-1}_{\overline{i}} \,
 \omega,
\end{eqnarray*}
for $\theta \in \Sy_m$, $\omega \in \Sy_n$, $\sigma \in \Sy_N$ and for $\tau \in \mathbb{S}_b$ with $\tau_{\overline{l}}$ the corresponding block permutation, $\upsilon \in \mathbb{S}_a$ and $\upsilon_{\overline{i}}$ the corresponding block permutation.
\end{definition}

\begin{figure*}[h]
	\begin{tikzpicture}[scale=0.7]
	\draw[thick] (2,1) to[out=270,in=90] (6,-1) to[out=270,in=90] (7,-2);
	\draw[thick] (1,1) to[out=270,in=90] (3,-1);
	\draw[thick] (0,-1) to[out=270,in=90] (1,-2);
	\draw[thick] (1,2) to[out=270,in=90] (0,1) to[out=270,in=90] (1,-1) -- (3,-1) to[out=270,in=90] (2,-2);
	\draw[thick] (8,2) to[out=270,in=90] (6,1) to[out=270,in=90] (8,-1) to[out=270,in=90] (10,-2);
	\draw[thick] (3,2) to[out=270,in=90] (5,1) ;	
	\draw[thick] (10,1) to[out=270,in=90] (9,-1);
	\draw[draw=white,double=black,double distance=2*\pgflinewidth,thick] (1,-1) to[out=270,in=90] (0,-2);
	\draw[draw=white,double=black,double distance=2*\pgflinewidth,thick] (9,-1) to[out=270,in=90] (8,-2);
	\draw[draw=white,double=black,double distance=2*\pgflinewidth,thick] (7,-1) to[out=270,in=90] (4,-2);
	\draw[draw=white,double=black,double distance=2*\pgflinewidth,thick] (7,2) to[out=270,in=90]  (9,1) to[out=270,in=90] (10,-1);
	\draw[draw=white,double=black,double distance=2*\pgflinewidth,thick] (0,2) to[out=270,in=90] (2,1);
	\draw[draw=white,double=black,double distance=2*\pgflinewidth,thick] (8,1) to[out=270,in=90] (4,-1);
	\draw[draw=white,double=black,double distance=2*\pgflinewidth,thick] (5,2) to[out=270,in=90] (4,1) to[out=270,in=90] (0,-1);
	\draw[fill=white] (-0.3,0.8) rectangle (2.3,1.2);
	\draw[fill=white] (3.7,0.8) rectangle (6.3,1.2); 
	\draw[fill=white] (7.7,0.8) rectangle (10.3,1.2); 
	\draw[fill=white] (-0.3,-1.2) rectangle (4.3,-0.8);
	\draw[fill=white] (5.7,-1.2) rectangle (10.3,-0.8);
	\draw (1,1) node {\small{$\nu_1$}};
	\draw (5,1) node {\small{$\nu_2$}};
	\draw (9,1) node {\small{$\nu_3$}};
	\draw (2,-1) node {\small{$\mu_1$}};
	\draw (8,-1) node {\small{$\mu_2$}};
	\end{tikzpicture}
	\caption{An element of $\M\boxtimes\N$.}
\end{figure*}

\begin{remark}
We refer the reader to \cite[Section~1.3]{Vallette07} for more details. In \emph{loc. cit.}  the connected composition product is only defined for left and right reduced $\Sy$-modules and is denoted by $\boxtimes $. 
Notice that if in the above formula there appears a term coming from $\N$, then there cannot be any element coming from $\M(m,0)$ since then the permutation between the two levels would not be connected. 
\end{remark}

We consider the $\Sy$-bimodule $\I$ define by $\I(1,\, 1)\coloneqq\k$ and by 
$\I(m,\,n)\coloneqq0$ otherwise. 

\begin{proposition}[{\cite[Proposition~1.6]{Vallette07}}]
		The category $\big(\Sbimod,\boxtimes, \I\big)$ forms  a monoidal category.
\end{proposition}

\begin{proposition}[{\cite[Proposition~2.11]{Leray19protoI}}]\label{prop::S_permute_prod_connexe_bi}
	Let $\M, \N$ be  two  $\Sy$-bimodules. The non-unital commutative $\otimes$-monoid $\S\M \sq \S\N$ is free on $\M \boxtimes \N$: 
	\[
\S\M \sq \S\N\cong \S(\M \boxtimes \N)\ .
	\]
\end{proposition}

\begin{proof}
The left-hand side is made up of 2-level labeled graphs with upper elements coming from $\N$ and lower elements coming from $\M$ and 1-level graphs with elements of input arity $0$ coming from $\M$; the right-hand side is made up of concatenations of connected 2-level labeled graphs with upper elements coming from $\N$ and lower elements coming from $\M$ and concatenation of elements of input arity $0$ coming from $\M$. 
They are thus isomorphic.
\end{proof}

\begin{remark}
Notice that $\S\I\cong \k[\Sy]$.
\end{remark}

There is an embedding of differential graded vector spaces into $\Sy$-bimodules given by 
$A(1,0)\coloneqq A$ and by $A(m,n)\coloneqq 0$ otherwise, for  $A\in \dgVect$. The \emph{endomorphism $\Sy$-bimodule} is given by 
$\End_A(m,n)\coloneqq \Hom\big(A^{\otimes n}, A^{\otimes m}\big)\ $. More generally, for two dg vector spaces $A$ and $B$, we consider the $\Sy$-bimodule $\End^A_B(m,n)\coloneqq \Hom\big(A^{\otimes n}, B^{\otimes m}\big)\ $.

\begin{lemma} \label{lem:LEMMA}
Every differential graded vector space $A$ and every  $\Sy$-bimodules $\M, \N$ satisfy the following relations:
	\begin{enumerate}
	\item $(\S A) (m,0)\cong A^{\otimes m}$ for every $m \in \mathbb{N}^*$\ ,
	\item $\M\boxtimes A \cong \M \sq \S A$\ ,
	\item $(\M\otimes \N)\sq \S A \cong (\M\sq \S A )\otimes (\N\sq\S A)$\ ,
	\item $\End^A_B \cong \hom(\S A, \S B)\ $.
	\end{enumerate}
\end{lemma}

\begin{proof}
The computations are straightforward.
\end{proof}

\begin{remark}
Any $\Sy$-module $\M$ can be seen as an $\Sy$-bimodule concentrated in arity $(1,n)$, for $n\in \NN$. Under this identification, we have $\M\circ \N \cong \M\boxtimes \N \cong \M \sq \S\N$, where $\circ$ stands for the composite  product of operads.
\end{remark}

\section{Properadic homological algebra}

\subsection{Properads}

\begin{definition}[Properad]
A  \emph{properad} is a monoid in the monoidal category 
$\big(\Sbimod,\allowbreak \boxtimes, \allowbreak \I\big)$\ .
\end{definition}

\begin{exam}
Equipped with the composition of functions, the endomorphism $\Sy$-bimodule $\End_A$ becomes a properad, called the \emph{endomorphism properad}.
\end{exam}

Unfolding the definition, a properad amounts to a triple $(\calP ,\gamma, \eta)$ where $\gamma : \calP\boxtimes \calP \to \calP$ composes operations along connected directed graphs with 2 levels and where $\eta : \I \to \calP$ is a unit for the composition map $\gamma$. 

\begin{definition}[Infinitesimal composition product]
The \emph{infinitesimal composition product}  \[\M \ibt \N\] of two $\Sy$-bimodules $\M$ and $\N$ is the sub-$\Sy$-bimodule  of $(\I \oplus \M) \boxtimes (\I\oplus \N)$ which is made up of the parts linear in $\M$ and in $\N$. 
\end{definition}

\begin{figure}[h!]
	\begin{tikzpicture}[scale=0.7]
		\draw[thick] (2,2.5)--(2,2);
		\draw[thick] (1,0)--(1,-0.5);
		\draw[thick] (1.5,0.5) to[in=270,out=90] (2.5,1.5);
		\draw[thick] (1,0.5) to[in=270,out=90] (2,1.5);		
		\draw[thick] (0.5,0) to[in=270,out=90] (0.5,2.5);
		\draw[thick] (3,-0.5) to[in=270,out=90] (3,2.5);
		\draw[fill=white] (0,0) rectangle (2,0.5);
		\draw[fill=white] (1.5,1.5) rectangle (3.5,2);
		\draw (1,0.2) node {\small{$\mu$}};
		\draw (2.5,1.7) node {\small{$\nu$}};
	\end{tikzpicture}
	\caption{An element of $\M \ibt \N$.}
\end{figure} 

\begin{definition}[Infinitesimal composition map]
The \emph{infinitesimal composition map} of a properad is defined by 
\[
	\begin{tikzcd}[column sep=huge]
		\gamma_{(1,1)} \ : \ \calP\ibt \calP \ar[r, hook]  &
		(\I \oplus \calP) \boxtimes (\I \oplus \calP) \ar[r,"(\eta + \id)\boxtimes (\eta + \id)"]  & 
		\calP \boxtimes \calP \ar[r,"\gamma"]  &
	  	\calP\ .
	\end{tikzcd}
	\]
It amounts to composing only two operations at a time. 
\end{definition}

\begin{definition}[Graph module endofunctor]
We consider the set $\Gs$ of \emph{directed connected graphs} and the endofunctor $\G$ of the category of $\Sy$-bimodules given the \emph{graph module} of $\M$: 
$$\G(\M) \coloneq \bigoplus_{\gg \in \Gs} \gg(\M) \ , $$
where $\gg(\M)$ is the module obtained by labeling each vertex of $\gg$ by elements of $\M$ of corresponding arity. 
The summand corresponding to the trivial graph $\gg=|$ is equal to $\I$.
\end{definition}

We denote by $\G(\M)^{(k)}$ the sub-$\Sy$-bimodule made up of graphs with $k$ vertices. Notice 
that 
\[
\G(\M)^{(0)}\cong \I\ , \quad 
\G(\M)^{(1)}\cong \M\ , \quad   \text{and} \quad 
\G(\M)^{(2)}\cong \M \ibt \M\ . \]
We refer the reader to \cite[Section~2.7]{Vallette07} for more details.

\medskip

We endow the graph endofunctor $\G$ with a monad structure where the natural transformation $\G\circ \G \to \G$ amounts to forgetting the partition of graphs on the left hand side and where the unit $\id \to \G$ is given by the embeddings $\M \hookrightarrow \G(\M)$ of graphs with 1-vertex.  

\begin{proposition}\label{prop:ProperadDefEqui}
The category of properads is equivalent to the category of algebras over the graph monad $\G$. 
\end{proposition}

\begin{proof}
This is a direct corollary of \cite[Theorem~2.3]{Vallette07}.
\end{proof}

\begin{corollary}
The free properad on an $\Sy$-module $\M$ is given by $\G(\M)$ equipped with the grafting of 2 levels of 
directed connected graphs.
\end{corollary}

\begin{definition}[Augmented properad]
An \emph{augmented properad} is a properad equipped with a morphism  $\varepsilon : \calP \to \I$ of properads,  called the \emph{augmentation morphism}, whose composite with the unit is equal to the identity. 
\end{definition}

The \emph{augmentation ideal} is the kernel of the augmentation morphism; it is denoted by $\oP\coloneqq \ker \varepsilon$. Any augmented properad is naturally isomorphic to 
$\calP\cong \I \oplus \oP$; this induces an isomorphism of categories between augmented properads and properads without units. 

\begin{remark}
Notice that the endormorphism properad $\End_A$ cannot be augmented in general. 
\end{remark}

\subsection{Coproperads} The dual situation is slightly more subtle due to the infinite sums that can appear because of counits. As a consequence, we will only consider the comonad of ``non-counital'' coproperads and then add for free the counit, which will thus be coaugmented. 

\begin{definition}[Reduced graph endofunctor]
We consider the set $\bGs\coloneqq \Gs \backslash \{| \}$ of \emph{reduced directed connected graphs} and the endofunctor $\G^c$ of the category of $\Sy$-bimodules given the \emph{reduced graph module} of $\M$: 
$$\G^c(\M) \coloneqq \bigoplus_{\gg \in \bGs} \gg(\M) \ . $$
\end{definition}

We denote again by $\G^c(\M)^{(k)}$ the sub-$\Sy$-bimodule made up of graphs with $k$ vertices, which gives now
\[
\G^c(\M)^{(0)}\cong 0\ , \quad 
\G^c(\M)^{(1)}\cong \M\ , \quad   \text{and} \quad 
\G^c(\M)^{(2)}\cong \M \ibt \M\ . \]

We endow the reduced graph endofunctor $\G^c$ with a comonad structure where the natural transformation $\Delta_{\G^c} : \G^c\to \G^c \circ \G^c $ 
sends an element of $\gg(\M)$ to the sum of all the ways to partition the underlying graph $\gg$ and  the counit $\G^c \to \id$ is given by the projections 
$\G^c(\M) \twoheadrightarrow \M$ on graphs with 1-vertex. 

\begin{definition}[Comonadic coproperad]
A \emph{comonadic coproperad} is a coalgebra over the comonad $\G^c$ of reduced graphs. 
\end{definition}

Such a structure amounts to a \emph{decomposition map} $\Delta_{\oC} : \oC \to \G^c\big(\oC\big)$, which heuristically speaking splits any operation of $\oC$ into all possible ways. 

\begin{remark}
To be fully accurate, such a notion should be called comonadic  coproperad ``without counit'' since it does not include any counit, see below. 
\end{remark}

\begin{definition}[Coproperad]
A \emph{coproperad} is a comonoid in the monoidal category 
$\big(\Sbimod,\allowbreak \boxtimes, \allowbreak \I\big)$.
\end{definition}

\begin{remark} 
Notice that the arity-wise linear dual $\calC^*$ of a coproperad admits a canonical properad structure. 
The reverse is also partial true: the arity-wise linear dual $\calP^*$ of a properad, with finite dimensional components,  admits a canonical coproperad structure. 
\end{remark}

\begin{definition}[Coaugmented coproperad]
A \emph{coaugmented coproperad} is a coproperad $\calC$ equipped with a morphism  $\eta : \I \to \calC$ of coproperads,  called the \emph{coaugmentation morphism}, whose composite with the counit is equal to the identity. 
\end{definition}

The \emph{coaugmentation coideal} is the cokernel of the coaugmentation morphism; it is denoted by $\oC\coloneqq \mathrm{coker}\,  \eta$. Any coaugmented coproperad is naturally isomorphic to 
$\calC\cong \I \oplus \oC$; this induces an isomorphism of categories between coaugmented coproperads and coproperads without counits. 

\begin{definition}[2-level graphs]
A \emph{2-level graph} is an element of $\G(\M)$ or $\G^c(\M)$ such that all its vertices can be placed along two levels.
\end{definition}

\begin{remark}
Notice that a 2-level graph necessarily contains at least two vertices, one on the bottom level and one on the top level. The sub-$\Sy$-bimodule of $\G(\M)$ or $\G^c(\M)$ made up of 2-level graphs is a sub-$\Sy$-bimodule of $(\I\oplus \M)\boxtimes (\I\oplus \M)$.
\end{remark}

Given a comonadic coproperad $\left(\oC, \Delta_{\oC}\right)$, we consider the following data $\left(\calC\coloneqq\I\oplus \oC, \Delta, \eps, \eta\right)$:
\begin{itemize}
\item[$\diamond$] $\Delta|_\I$ amounts to the isomorphism $\I \cong \I \boxtimes \I$; 
\item[$\diamond$] $\Delta|_{\oC}$ is the sum of the projection of $\Delta_{\oC}$ onto the space of 2-level graphs with the two copies, called the \emph{primitive part}, coming from $\oC \cong \I \boxtimes \oC$ and $\oC \cong \oC \boxtimes \I$; 
\item[$\diamond$] $\eps : \I\oplus \oC \twoheadrightarrow \I$ is the canonical projection; 
\item[$\diamond$] $\eta : \I \hookrightarrow\I\oplus \oC$ is the canonical inclusion.
\end{itemize}

\begin{proposition}\label{prop:ComCopropCoprop}
The above  assignment defines a functor from comonadic coproperads to coaugmented coproperads. 
\end{proposition}

\begin{proof}
The axioms for the counit and the coaugmentation are straightforward to check. The part of the coassociativity $(\Delta \boxtimes \id)\Delta \cong (\id \boxtimes \Delta)\Delta$ of the coproduct $\Delta$ whose image contains at least one level of $\I$ is clear. 
In order to treat the part where the upshot contains at least one element of $\oC$ on each of the three levels, we consider  the composite of the commutative diagram 
\begin{equation}\label{eq:CoalComonad}
	\begin{tikzcd}[column sep=large]
		\oC \ar[r, "\Delta_{\oC}"]  \ar[d, "\Delta_{\oC}"]  & \G^c\Big(\oC\Big) \ar[d, "\Delta_{\G^c}(\id)"]  \\
		\G^c\Big(\oC\Big) \ar[r, "\G^c\Big(\Delta_{\oC}\Big)"]  & \G^c\Big(\G^c\Big(\oC\Big)\Big)\ ,
	\end{tikzcd}
\end{equation}
defining a coalgebra over a comonad, with the projection onto $(\calC \boxtimes \calC) \boxtimes \calC$ viewed as a sub-$\Sy$-bimodule of partitioned graphs. The bottom-left composite produces $(\Delta \boxtimes \id)\Delta$\ . The corresponding top-right composite amounts first to taking the image of $\Delta_{\oC}$ on directed connected graphs with 3 levels followed by all the ways to partition the bottom two levels. Under the isomorphism $(\calC \boxtimes \calC) \boxtimes \calC\cong \calC \boxtimes (\calC \boxtimes \calC)$, this composite is isomorphic to the image of $\Delta_{\oC}$ on directed connected graphs with 3 levels followed by all the ways to partition the top two levels. This is equal to the bottom-left composite 
$(\id \boxtimes \Delta)\Delta$. This argument can be summarized into the following commutative diagram 
\[
\begin{tikzcd}[row sep = large, column sep=large]
	& \oC 
	\ar[d,"\Delta_{\oC}"]	\ar[dr,"\overline{\Delta}"]	\ar[dl,"\overline{\Delta}"']
	&  \\
	\calC\boxtimes \calC 
	\ar[d,"\big(\Delta_\I+\overline{\Delta}\big)\boxtimes\id"']
	& \G^c\big(\oC\big) \ar[r,two heads] \ar[l,two heads] \ar[d,"\Delta_{\G^c}(\oC)" description]
	\ar[d, "\G^c\big(\Delta_{\oC}\big)"', bend right=40]
	\ar[d,"\G^c\big(\Delta_{\oC}\big)", bend left=40]
	& \calC \boxtimes \calC  \ar[d,"\id \boxtimes\big(\Delta_\I+\overline{\Delta}\big)"] \\
	(\calC\boxtimes \calC) \boxtimes \calC \ar[rr, bend right=20, "\cong" description, leftrightarrow]
	& \G^c\left(\G^c(\oC)\right)  \ar[r,two heads] \ar[l,two heads] 
	& \calC \boxtimes (\calC\boxtimes \calC) \ , 
\end{tikzcd} 
\]
where $\Delta_\I : \I\cong\I\boxtimes \I \to \calC \boxtimes \calC$ and $\overline{\Delta}$ is equal to $\Delta$ without the primitive part. 
\end{proof}

From now on, we will only consider coaugmented coproperads.

\begin{definition}[Conilpotent coproperad]
A coaugmented coproperad which comes from a co\-mo\-na\-dic coproperad is called \emph{conilpotent}. 
\end{definition}

\begin{definition}[Infinitesimal decomposition map]
Given a coaugmented coproperad $(\calC, \Delta, \varepsilon, \eta)$, we consider the \emph{infinitesimal decomposition map} defined by 
\[
	\begin{tikzcd}[column sep=huge, row sep=tiny]
		\Delta_{(1,1)} \ar[r,phantom,":" description] &  \oC 
		\ar[r, "\Delta"]  &
		\calC \boxtimes \calC \ar[r,"(\eps;\id)\boxtimes (\eps;\id)"] &
		 \oC \ibt \oC\ ,\\
		 & \I \ar[r, "\cong"] & \I\boxtimes \I \ .
	\end{tikzcd}
\]
where the  map $(\eps;\id)\boxtimes (\eps;\id)$ amounts to applying the counit $\varepsilon$ everywhere except for one place on the left-hand and on the right-hand sides. 
\end{definition}

\begin{remark}
In the case of conilpotent coproperads,  the induced infinitesimal decomposition map, produced in two successive steps, can be given directly by 
\[
\begin{tikzcd}[column sep=normal]
		\oC 
		\ar[r, "\Delta_{\oC}"]  &
		\G^c\big(\oC\big)\ar[r, two heads] &
		 \G^c\big(\oC\big)^{(2)} \cong \oC \ibt \oC\ \ .
	\end{tikzcd}
\]
One can summarise this process by 
$\begin{tikzcd}[column sep=small]
		\Delta_{\oC}  \ar[r, Rightarrow] & \Delta \ar[r, Rightarrow] & \Delta_{(1,1)} \ .
	\end{tikzcd}$
Notice that it is not possible in general to go the other way round: being able to split into 2 vertices does allow one to get all the splittings along graph with 2 levels and being able to split into 2 levels does not allow one to get all the splittings along any graph. The following cofree comonadic coproperad provides us with counter-examples. 
\end{remark}

By definition, the \emph{cofree comonadic coproperad} on an $\Sy$-bimodule $\M$ is given by the reduced graph module $\G^c(\M)$. 
\cref{prop:ComCopropCoprop} endows 
$\I \oplus \G^c(\M)$ with a coaugmented coproperad structure which is cofree among conilpotent coproperads. 
The coproperad structure $\Delta :\I\oplus \G^c(\M) \to \left(\I\oplus \G^c(\M)\right) \boxtimes \left(\I\oplus \G^c(\M)\right)$ amounts to splitting any graph along an horizontal line, see \cite[Section~2.8]{Vallette07} for more details.

\subsection{Bar-cobar adjunction}

\begin{definition}[Convolution product] 
	Let $(\calP, \gamma, \eta)$ be a properad and let $(\calC, \Delta, \varepsilon, \eta)$ be a coaugmented coproperad. The  \emph{convolution product} of two elements of $\Hom_{\Sy}(\calC, \calP)$ is defined by the following composite
	\[
	\begin{tikzcd}[row sep=tiny]
		f\star g \ar[r,phantom,":" description] & \oC \ar[r,"\Delta_{(1,1)}"] & \calC \ibt  \calC \ar[r,"f\ibt  g"] & \calP \ibt  \calP \ar[r,"\gamma_{(1,1)}"] & \calP\ , \\
		& \I \cong \I\boxtimes\I  \ar[r, "\eta\boxtimes \eta"]& \calC \boxtimes \calC \ar[r,"f\boxtimes  g"] & \calP\boxtimes \calP \ar[r,"\gamma"] & \calP\ .
	\end{tikzcd}
	\]
\end{definition}

\begin{proposition}[{\cite[Proposition~11]{MerkulovVallette09I}}]
For any dg properad $\calP$ and any dg coproperad $\calC$, the convolution product defines a dg Lie-admissible algebra 
\[\big(\Hom_{\Sy}(\calC, \calP), \partial, \star \big)\ , \] 
called the \emph{convolution algebra}.
\end{proposition}

\begin{definition}[Twisting morphism] 
A \emph{twisting morphism} is a degree $-1$ solution to the Maurer--Cartan equation
\[\partial(\alpha) + \alpha \star \alpha =0\ .\]
\end{definition}

We denote their associated set by $\Tw(\calC, \calP)$. When the properad $\calP$ is augmented (respectively the coproperad $\calC$ is coaugmented), we require that the composite of a twisting morphism with the augmentation morphism (respectively the coaugmentation morphism) vanishes. 
The following two constructions represent the twisting morphism bifunctor. 

\begin{definition}[Bar construction]
The \emph{bar construction} of an augmented dg properad $(\calP, d_\calP, \gamma, \eta,  \varepsilon)$ is the  following quasi-cofree conilpotent dg coproperad: 
	\[
	\Bar\calP \coloneqq \left(\I\oplus \G^c\big(s\oP\big), d_1+d_2\right)\ ,
	\]
where $d_1$ is the unique coderivation extending the internal differential $d_\calP$ and where 
$d_2$ is the unique coderivation extending the infinitesimal composition map $\gamma_{(1,1)}$\ . 
\end{definition}

We refer the reader to \cite[Section~3.5]{MerkulovVallette09I} for more details. 

\begin{remark}\label{rem:BarEnd}
By convention, the bar construction of the endomorphism operad $\End_A$ is defined by the bar construction of its augmentation $\I\oplus \End_A$, i.e. $\Bar\End_A \coloneqq \left(\I\oplus \G^c\big(s\End_A\big), d_1+d_2\right)$ \ .
\end{remark}

\begin{definition}[Cobar construction]
The \emph{cobar construction} of a coaugmented dg coproperad $(\calC, d_\calC, \Delta, \varepsilon, \eta)$ is the  following (augmented) quasi-free dg properad: 
	\[
	\Cobar \calC \coloneqq \left(\G\big(s^{-1}\oC\big), d_1+d_2\right)\ ,	\]
where $d_1$ is the unique derivation extending the internal differential $d_\calC$ and where 
$d_2$ is the unique derivation extending the infinitesimal decomposition map $\Delta_{(1,1)}$\ . 
\end{definition} 

We refer the reader to \cite[Section~3.6]{MerkulovVallette09I} for more details. 

\medskip

From now on, we denote by $\dgprop$ the category of augmented dg properads and by $\mathsf{dg\ co}\allowbreak\mathsf{pro}
\allowbreak\mathsf{pe}
\allowbreak\mathsf{rads}
$ the category of 
 conilpotent dg coproperads. 

\begin{proposition}[Partial Rosetta stone {\cite[Proposition~17]{MerkulovVallette09I}}]\label{Rosetta}
There exist natural bijections 
	\[
	\Hom_{\dgprop}(\Cobar\calC,\calP) \cong \Tw(\calC,\calP) \cong \Hom_{\dgcoprop}(\calC,\Bar\calP)\ ,
	\]
	for conilpotent dg coproperads $\calC$.
\end{proposition}

Among others, this proves that the functors $\Bar$ and $\Cobar$ form a pair of adjoint functors:
\[\begin{tikzcd}
	\Cobar : \dgcoprop
\arrow[r, harpoon, shift left=0.9ex, "\perp"']
&
\arrow[l, harpoon,  shift left=0.9ex]
	\dgprop : \Bar \  .
\end{tikzcd}\]

In the sequel, we will mainly be interested by dg properads of the form $\Cobar \calC$. 
The counit of the bar-cobar adjunction provides us with a functorial cofibration resolution 
$\Cobar\Bar \calP\xrightarrow{\sim} \calP$ for dg properads $\calP$. Such a huge resolution can sometimes be simplified 
by considering a coaugmented dg sub-coproperad $\calC \hookrightarrow \Bar \calP$ equipped with twisting morphism $\calC \to \calP$ satisfying the Koszul property: $\Cobar \calC \xrightarrow{\sim} \calP$\ . We refer the reader to  \cite[Section~7]{Vallette07} for more details. 
In both cases, the category of $\Cobar \calC$-gebras deserves the name of \emph{homotopy $\calP$-gebras} and the purpose of the sequel is to show that it carries homotopy properties ($\infty$-morphisms, homotopy transfer theorem) that simply fail on the level of $\P$-gebras. 

\section{Infinity-morphisms of homotopy gebras}

The purpose of this section is to extend the  notion of $\infty$-morphism of homotopy algebras over an operad to  homotopy gebras over a properad. In the operadic case, one uses in a crucial way the notion of cofree coalgebras over the Koszul dual cooperad, as summarised on the left-hand column of the following table. Unfortunately, such a notion does not exist anymore on the properadic level. To bypass this difficulty, we introduce new notions summarised on the right-hand column of the table.

\medskip

\begin{center}
\begin{tabular}{|c|c|}
\hline
\rule{0pt}{10pt}{\sc Operads} &  {\sc Properads}\\
\hline
\rule{0pt}{10pt}$\calC$-comodules&
monoid $\S \calC$-comodules \\
\hline
\rule{0pt}{10pt}cofree $\calC$-coalgebras $\calC(A)$ &
bifree monoid $\S \calC$-comodules
$\S \calC \sq \S A \cong \S(\calC \boxtimes A)$ \\ 
\hline
\rule{0pt}{10pt}$\Hom_{\Sy}(\calC, \End_A)\cong \mathrm{Coder}(\calC (A))$ &
$\Hom_{\Sy}(\calC, \End_A)\cong \Bider(\S\calC \sq \S A)$\\
\hline
\rule{0pt}{10pt}
$\Omega \calC$-algebra structure $\cong \mathrm{Codiff}(\calC (A))$ & 
$\Omega \calC$-gebra structure $\cong \Bidiff(\S\calC \sq \S A)$ \\
\hline 
\rule{0pt}{10pt} 
$\infty$-morphism of $\Omega\calC$-algebras $\coloneqq$ 
&
$\infty$-morphism of $\Omega\calC$-gebras $\coloneqq$ 
\\
\rule{0pt}{10pt}  morphism of quasi-cofree dg $\calC$-coalgebras &
morphism of quasi-bifree dg monoid $\S \calC$-comodules
\\
\hline
\end{tabular}
\end{center}

\subsection{Monoid $\S \calC$-comodule}
Let $(\calC, \Delta, \eps)$ be a coproperad. By \cref{prop::S_permute_prod_connexe_bi},  $\S \calC$ carries a comonoid structure in the category 
$\big(\mathsf{nuComMon}_\otimes, \square,\k[\Sy] \big)$
of non-unital commutative $\otimes$-monoids, given by 
\[\begin{tikzcd}
	\S \calC  \ar[r,"\S(\Delta)"] & \S (\calC\boxtimes \calC)  \ar[r,"\cong"] &\S \calC \sq \S \calC \ . 
\end{tikzcd}\]
We denote its free product by $\nu : \S\calC\otimes \S\calC \to \S\calC$. These structure maps make $\S \calC$ like a  commutative bialgebra with respect to two different monoidal structures. 

\begin{definition}[Monoid $\S \calC$-comodule]
A \emph{monoid $\S \calC$-comodule} is a left comodule over the comonoid $\S \calC$ in the monoidal category $\big(\mathsf{nuComMon}_\otimes, \allowbreak  \square, \allowbreak \k[\Sy] \big)$ of non-unital commutative $\otimes$-monoids. 
\end{definition}

Such a structure amounts to a triple $(\M, \mu, \delta)$ where  $\mu : \M\otimes \M \to \M$ is a commutative associative product and where $\delta : \M \to \S\calC \sq \M$ is a coaction satisfying the following commutative diagram 

\[\begin{tikzcd}
\M \otimes \M \ar[rr,"\mu"] \ar[d, "\delta\otimes\delta"]&& \M \ar[d, "\delta"] \\
(\S\calC \sq \M) \otimes (\S\calC \sq \M) \ar[r,"\IL"] & 
(\S \calC \otimes \S\calC )  \sq (\M \otimes \M)
\ar[r,"\nu \otimes \mu"] &\  \S\calC \sq \M\ ,\\
\end{tikzcd}\]
where $\IL$ stands for the interchange law defined in \cref{eq:interchange_law}. 
A morphism of monoid $\S \calC$-comodules is a map of $\Sy$-bimodules $\M \to \N$ which is a  morphism of non-unital commutative $\otimes$-monoids commuting with the respective comodule structure maps. This forms a category denoted by $\SCcomod$. The assignment $\calC \mapsto \SCcomod$ is a functor, where any morphism $G : \calC \to \calD$ of coproperads induces a functor $\SCcomod \to \mathsf{mon}\textsf{-}\S\calD\textsf{-}\mathsf{comod}$ under 
$(\M, \mu, \delta)\mapsto (\M, \mu, (\S G\sq \id)\circ \delta)$\ . 

\medskip

The main example of monoid $\S \calC$-comodules is given by $\S\calC \sq \S \V \cong \S(\calC \boxtimes \V)$, 
for any $\Sy$-bimodule $\V$. It is obtained as the image of $\V$ under the composite of the free non-unital commutative $\otimes$-monoid functor followed by 
the cofree left $\S \calC$-$\square$-comodule functor. 
We call such an object a \emph{bifree monoid $\S\calC$-comodule}. 

\medskip

In the sequel, we will mainly be considering bifree monoid $\S\calC$-comodules $\S\calC \sq \S A$ on 
graded vector spaces $A$, i.e a graded $\Sy$-bimodule concentrated in arity $(1,0)$. 
This choice of  terminology  is motivated by the following property. 

\begin{lemma}\label{lem:bifree}
There is a natural bijection
\[
\Hom_{\SCcomod}(\S \calC \sq \S A, \S \calC \sq \S B)\cong \Hom_{\Sy}\big(\calC , \End^A_B\big)\ ,
\]
for coproperads $\calC$ and graded vector spaces $A$ and $B$.
\end{lemma}

\begin{proof}
Let us recall that Point~(2) of \cref{lem:LEMMA} asserts $\calC \boxtimes A \cong \calC \sq \S A$ and that 
Point~(4) of \cref{lem:LEMMA} asserts $\End^A_B\cong \hom(\S A, \S B)$. This implies by \cref{prop:closedmono}
\[\Hom_{\Sy}\big(\calC , \End^A_B\big)\cong \Hom_{\Sy}(\calC ,\hom(\S A, \S B)) \cong \Hom_{\Sy}(\calC\sq \S A, \S B)\cong \Hom_{\Sy}(\calC\boxtimes A, \S B)\ .\]
The universal property of free $\otimes$-monoids amounts to a natural bijection 
\[
\Hom_{\otimes\textsf{-}\mathsf{mon}}(\S (\calC \boxtimes A), \S B)\cong \Hom_{\Sy}(\calC\boxtimes A, \S B)\ 
\]
and the universal property of cofree $\S \calC$-comodules amounts to a natural bijection 
\[
\Hom_{\S \calC\textsf{-}\mathsf{comod}}(\S \calC \sq \S A, \S \calC \sq \S B)\cong \Hom_{\Sy}(\S \calC \sq \S A, \S B)\ .
\]
It remains to use the isomorphism $\S(\calC\boxtimes A)\cong \S \calC \sq \S A$ and to notice that
 the latter natural bijection restricts to morphisms of monoid $\S\calC$-comodules  on the left-hand side and to morphisms of monoids on the right-hand side. 
\end{proof}

\subsection{Properadic Rosetta stone}
Starting now from a dg coproperad $(\calC, d_\calC, \Delta, \eps)$, let us see how to extend the above notion to the  differential graded level. 

\begin{definition}[Biderivation]
A \emph{biderivation} of a monoid $\S\calC$-comodule $(\M, \mu, \delta)$ is an map $d : \M \to \M$ which is derivation with respect to the product $\mu$ and a coderivation with respect to 
the comodule structure $\delta$: 
\[
d\circ \mu = \mu \circ (d \otimes \id + \id \otimes d) 
\quad \text{and} \quad 
\delta \circ d =  (\id \sq d)\circ \lambda\ .
\]
\end{definition}

We denote the graded space of biderivations of a monoid $\S\calC$-comodule $\M$ by $\Bider(\M)$. 

\medskip

Given a dg $\Sy$-bimodule $(\V, d_\V)$, the differential, denoted slightly abusively by $d_{\S\V}$, induced by $\d_\V$ only on the bifree monoid $\S\calC$-comodule  
$\S\calC \sq \S \V \cong \S(\calC \boxtimes \V)$,  which is equal to the sum of   all the ways to apply $d_\V$ to all the elements coming from $\V$ but one each time, is a bidifferential. Notice however that the differential, denoted slightly abusively by $d_{\S\calC}$, induced by $\d_\calC$ fails to be a coderivation with respect to the coaction, so it cannot be a biderivation. 

\begin{lemma}\label{lem:BiDerHom}
There is a natural isomorphism of graded vector spaces 
\[
\Bider(\S \calC \sq \S A)\cong \Hom_{\Sy}(\calC , \End_A)\ ,
\]
where $\calC$ is a dg coproperad and $A$ a dg vector space.
\end{lemma}

\begin{proof}
This proof is similar to the proof of \cref{lem:bifree}. 
We consider the following action of the non-unital commutative $\otimes$-monoid $\S (\calC \boxtimes A)$ on $\S A$:
\[\begin{tikzcd}[column sep=huge]
	\S (\calC \boxtimes A)\otimes \S A  \ar[r,"\S(\varepsilon \boxtimes \id)\otimes \id"] & \S A \otimes \S A   \ar[r,"\nu"] &\S A 
\end{tikzcd}\]
and we consider the set of derivations $\mathrm{Der}(\S (\calC \boxtimes A), \S A)$ with respect to this action. 
Since any derivation from a free commutative is characterized by the image of its generators, there exists a natural isomorphism 
\[
\mathrm{Der}(\S (\calC \boxtimes A), \S A)\cong \Hom_{\Sy}(\calC , \End_A)\ .
\]
Since any coderivation of a cofree comodule is characterized by its projection onto its cogenerators, there exists a natural isomorphism 
\[
\mathrm{Coder}(\S \calC \sq \S A)\cong \Hom_{\Sy}(\S \calC \sq \S A, \S A)\ .
\]
It remains to notice that this isomorphism restricts to derivations on both sides in order to get the required natural isomorphism. 
\end{proof}

\begin{remark}
Tracing through the above mentioned isomorphisms, the unique biderivation $d_\alpha$ associated to a map  $\alpha : \calC \to \End_A$, which is equivalent to a map still denoted by $\alpha : \calC \boxtimes A \to \S A$, is explicitly given by 
\begin{equation}\label{eq:dvarphi}
\begin{tikzcd}[column sep=large]
	\S \calC \sq \S A  \ar[r,"\S(\Delta)\sq \id"] &
	\S \calC \sq \S \calC \sq \S A \cong \S\calC \sq \S(\calC \boxtimes A)  \ar[r,"\id \sq \S(\varepsilon; \alpha)"] &
	\S\calC \sq \S(A; \S A) \ar[r,"\id\sq \widetilde{\nu}"] & \S\calC \sq \S A \ ,
\end{tikzcd}
\end{equation}
where  
\begin{itemize} 
	\item[$\diamond$]  $\S(A; \S A)$ is made up of sums of monomials with some elements from $A$ but one element from $\S A$;
	\item[$\diamond$] the map $\S(\varepsilon; \alpha)$ is the sum of all the ways to apply $\varepsilon$ to all the terms except one where we apply $\alpha$;
	\item[$\diamond$] $\widetilde{\nu}$ is the natural map coming from the concatenation product $\nu$ on $\S A$.
\end{itemize}
\end{remark}

We equipped the graded space of biderivations of a monoid $\S \calC$-comodule with the usual Lie bracket 
\[[d, d']\coloneqq d \circ d' - (-1)^{|d||d'|} d'\circ d
\ .\]
In the case of bifree $\S\calC$-comonoids $\S\calC \sq \S \V$, we consider the underlying differential 
$d_{\S\calC\sq \S \V}=d_{\S \calC}+d_{\S \V}$. Even if this latter one fails to be a biderivation, its adjoint operator 
$[d_{\S\calC\sq \S \V}, -]$ preserves biderivations. So it defines a square-zero derivation of the Lie algebra of biderivations. 

\begin{proposition}\label{prop:isodgLie}
For  coaugmented dg coproperads $\calC$ and dg vector spaces $A$, the natural isomorphism of \cref{lem:BiDerHom} induces a natural isomorphism of dg Lie algebras 
\[
\big(\Bider(\S \calC \sq \S A), [d_{\S\calC\sq \S A}, -Ê], [- , -]\big) \cong \big(\Hom_{\Sy}(\calC , \End_A), \partial,  [- , -]\big)\ ,
\]
where the Lie bracket on the right-hand side is obtained by skew-symmetrizing the Lie-admissible bracket $\star$\ . 
\end{proposition}

\begin{proof}
Tracing through the composition of isomorphisms in the other way round in the proof of \cref{lem:BiDerHom}, one can see that the isomorphism $\Bider(\S \calC \sq \S A)\cong \Hom_{\Sy}(\calC , \End_A)$ is given by 
$d\mapsto \overline{d}\coloneqq (\S(\eps)\sq \id)\circ d|_{\calC\boxtimes A}$\ ,
\[\begin{tikzcd}[column sep=large]
	\overline{d}\ : \ \calC \boxtimes A  \subset \S (\calC \boxtimes A)\cong \S \calC \sq \S A \ar[r,"d"] &
	\S \calC \sq \S A  \ar[r,"\S(\eps) \sq \id"] &
	\S\I \sq \S A\cong \S A\ .
\end{tikzcd}\]

Let us first prove the compatibility with respect to the differentials, that is 
$\partial\Big(\overline{d}\Big) =\overline{[d_{\S\calC\sq \S A}, d]}$\ .
It is straightforward to see that $\overline{d\circ d_{\S\calC\sq \S A}}=\overline{d}\circ d_{\S\calC\sq \S A}$ and since $\eps\circ d_\calC=0$, we have $\overline{d_{\S\calC\sq \S A}\circ d}=d_{\S A}\circ \overline{d}$\ . 

The compatibility with the respective Lie brackets follows from the relation $\overline{d\circ d'}=\overline{d}\star\overline{d'}$\ . Up to isomorphisms, the map $\overline{d\circ d'}$ is equal to 
\[\begin{tikzcd}[column sep=large]
 \calC \boxtimes A  \subset \S (\calC \boxtimes A) \ar[r,"{d'}"] &
\S (\calC \boxtimes A) \ar[r,"d"] & \S (\calC \boxtimes A) \ar[r,"\S(\eps\boxtimes\id)"] &
 \S A\ .
\end{tikzcd}\]
Since $d=d_{\overline{d}}$ and since the coproperad $\calC$ is coaugmented, one can see that the composite 
$\S(\eps\boxtimes\id) \circ d$ vanishes outside $\S(A; \calC\boxtimes A)$, the summand made up of sums of monomials with some elements from $A$ but one element from $\calC\boxtimes A$. Indeed, by the formula of $d_{\overline{d}}$, applying $\S(\eps\boxtimes\id) \circ d$ to at least two $\otimes$-concatenated elements from $\overline{\calC}\boxtimes A$ amounts to applying to at least one of them $\eps\boxtimes \id$, which is trivial. Let us now denote by $\mathrm{proj} : \S(\calC \boxtimes A) \twoheadrightarrow \S(A; \calC\boxtimes A)$ the canonical projection. 
Using again $d'=d_{\overline{d'}}$, one can see that the composite $\mathrm{proj}\circ d'|_{\calC\boxtimes A}$ is equal to 
\[\begin{tikzcd}[column sep=huge, row sep=tiny]
 \oC \boxtimes A   \ar[r,"\Delta_{(1,1)}\boxtimes \id"] &
  \big(\oC\ibt  \oC\big)\boxtimes A \cong \big(\oC\ibt \big(\oC\boxtimes A\big)\big)\boxtimes A  
  \ar[r,"\big(\id \ibt\overline{d'} \big)\boxtimes \id"] &
 \S\big(A; \oC\boxtimes A\big)\ ,\\
 \I \boxtimes A \cong A \ar[r,"d'"] & A \ar[r,"\cong"]&  \S\big(A; \I\boxtimes A\big)\ .
  \end{tikzcd}\]
Under the isomorphism $\Hom_{\Sy}(\calC\boxtimes A, \S A)\cong \Hom_{\Sy}(\calC , \End_A)$, this means that
\[
\overline{d\circ d'}= \S(\eps\boxtimes\id) \circ d\circ \mathrm{proj}\circ d'|_{\calC\boxtimes A}= \overline{d}\star\overline{d'}\ ,
\]
which concludes the proof. 
\end{proof}

\begin{definition}[Bidifferential]
A \emph{bidifferential} $d$ of a monoid $\S\calC$-comodule structure $(\M, d_\M, \mu, \delta)$ on a dg $\Sy$-bimodule is a degree $-1$ biderivation such that 
\[
\left(d_\M +d\right)^2=0
\ .\]
\end{definition}
We denote the graded space of bidifferentials of a monoid $\S\calC$-comodule $(\M, d_\M, \mu, \delta)$ by $\Bidiff(\M)$. 

\begin{definition}[Differential graded monoid $\S \calC$-comodule]\label{def:dgMonSCComod}
A \emph{differential graded monoid $\S \calC$-comodule} 
$(\M, d_\M+d, \mu, \delta)$ is a monoid $\S\calC$-comodule equipped with a bidifferential.
\end{definition}

We call $d_\M+d$ the total differential of a dg monoid $\S \calC$-comodule. 
A morphism of differential graded monoid $\S \calC$-comodules is a morphism of monoid $\S\calC$-comodules which commutes with the respective total differentials. This forms a category denoted by $\dgSCcomod$. 

\medskip

When the coproperad $\calC$ is coaugmented, we will mainly consider the case of bifree monoid $\S\calC$-modules $\S\calC \sq \S A \cong \S(\calC \boxtimes A)$ on dg vector spaces $A$, with underlying differential $d_{\S\calC\sq \S A}$. In this case, \emph{we require that bidifferentials vanish on $A$}, i.e. the following composite is trivial
\[\begin{tikzcd}[column sep=large]
	A \cong \I\sq A \ar[r,"\eta \sq \id"] & 
	\calC \sq A\subset \S\calC \sq \S A \ar[r,"d"] &
	\S \calC \sq \S A \ar[r,"\S(\eps)\sq \id"] &
	\S A \ar[r, two heads] & A \ . 
\end{tikzcd}\]
Such differential graded monoid $\S\calC$-module structure are called \emph{quasi-bifree}.

\begin{proposition}\label{prop:EndofRosetta}
For  coaugmented dg coproperads $\calC$ and dg vector spaces $A$, there is a natural bijection between bidifferentials on bifree module $\S\calC$-comodule 
$\S \calC \sq \S A$ and twisting morphisms from $\calC$ to $\End_A$:
\[
\Bidiff(\S \calC \sq \S A) \cong \Tw(\calC , \End_A)\ .
\]
\end{proposition}

\begin{proof}
The natural isomorphism of dg Lie algebras of \cref{prop:isodgLie} induces a natural isomorphism between the associated set of solutions to the respective Maurer--Cartan equations. Notice that the 
Maurer--Cartan equation for biderivations is equal to the defining relation for bidifferentials: $\left(d_{\S \calC\sq \S A} +d\right)^2=[d_{\S\calC\sq \S A}, d]+{\textstyle \frac12}[d,d]=0$\ . 
Under this correspondance, the vanishing condition on $\I$ for twisting morphisms 
is equivalent to the vanishing condition on $A$ for bidifferentials. 
\end{proof}

\begin{theorem}[Complete Rosetta stone]\label{thm:CompleteRosetta}
There exist natural bijections 
	\[
\begin{tikzcd}[column sep=small, row sep=tiny]
	\Hom_{\dgprop}(\Cobar\calC,\End_A)  \ar[r,phantom,"\cong" description] & 
	\Tw(\calC,\End_A) 	\ar[r,phantom,"\cong" description] & 
	\Hom_{\dgcoprop}(\calC,\Bar\End_A)  \ar[r,phantom,"\cong" description]& \Bidiff(\S \calC \sq \S A)\\
	F_\alpha \ar[r,leftrightarrow]& \alpha\ar[r,leftrightarrow]& G_\alpha \ar[r,leftrightarrow] & d_\alpha\ , 
\end{tikzcd}
\]
for conilpotent dg coproperads $\calC$. 
\end{theorem}

\begin{proof}
This is a direct corollary of the partial Rosetta stone given in \cref{Rosetta} and \cref{prop:EndofRosetta}. 
\end{proof}

\subsection{Infinity-morphisms}\label{subsec:InftyMorph} Let $\calC$ be a coaugmented dg coproperad. 

\begin{definition}[$\infty$-morphism]\label{def:InftyMor}
An \emph{$\infty$-morphism} $A\rightsquigarrow B$ between  two $\Cobar \calC$-gebra structures $\alpha\in  \Tw(\calC,\End_A)$ and $\beta\in  \Tw(\calC,\End_B)$ is a morphism 
\[\left(\S\calC \square \S A, d_{\S\calC \sq \S A}+d_\alpha\right) \to \left(\S\calC \square \S B, d_{\S\calC \sq \S B}+d_\beta\right)\]
 of dg monoid $\S \calC$-comodules.
\end{definition}

By definition, $\Cobar \calC$-gebras together with $\infty$-morphisms form a category, which is isomorphic to the sub-category of quasi-bifree monoid $\S \calC$-comodules on dg vector spaces. We denote it by  $\infty\textsf{-}\Cobar\calC\textsf{-}\mathsf{gebras}$. 

\medskip 

In order to give a shorter description of $\infty$-morphisms, we need to introduce the following notions. 

\begin{definition}[Left and right infinitesimal composition products]
The \emph{left} and \emph{right infinitesimal composition products}
\[\M \libt  \N \qquad \text{and} \qquad \M \ribt  \N\ , \] 
are the sub-$\Sy$-bimodules of $(\I \oplus \M)\boxtimes \N$ and $\M\boxtimes (\I \oplus \N)$ made up of the linear parts in $\M$ and $\N$ respectively.
\end{definition}

\begin{definition}[Left and right infinitesimal decomposition map]
The \emph{left} and \emph{right infinitesimal decomposition maps} of a coproperad are defined respectively by 
\[\begin{tikzcd}[column sep=large, row sep=tiny]
\Cop{}{(*)} \ : \ \oC \arrow[r,"\Delta"] & 
 \calC \boxtimes \calC \arrow[r,"(\eps; \id)\boxtimes \id"] & \oC \libt  \calC\ , \\
\Cop{(*)}{} \ : \ \oC \arrow[r,"\Delta"] & 
 \calC \boxtimes \calC \arrow[r,"\id \boxtimes (\eps; \id)"] & \calC \ribt  \oC\ , 
\end{tikzcd}
\]
and in each case by $\I\cong \I \boxtimes \I$ on $\I$.
\end{definition}

\begin{remark}
Notice that when the coproperad is conilpotent, that is coaugmented and given by a comonadic coproperad, these maps are simply given by the composites with the projections onto the module of 2-level graphs with one vertex on the bottom level or the top level respectively. 
\end{remark}

These notions give rise to the following operations. 

\begin{definition}[Left and right actions]
Given $f\in \Hom_{\Sy}\big(\calC, \End^A_B\big)$, $\alpha \in \Hom_{\Sy}(\calC, \End_A)$, and $\beta\in \Hom_{\Sy}(\calC, \End_B)$, the \emph{left action}  of $\beta$ on $f$ and the \emph{right action} of $\alpha$ on $f$ are defined respectively by 
	\[
	\begin{tikzcd}[column sep=normal, row sep=tiny]
	\beta \lhd f  \ar[r,phantom,":" description] &
	\oC \arrow[r,"\Cop{}{(*)}"] &  \oC \libt \calC 
	\arrow[r,"\beta\libt f"] & \End_B \libt \End^A_B \arrow[r] & \End^A_B\ ,
\\
& \I \arrow[r,"\cong"] &  \I \boxtimes \I 
	\arrow[r,"\beta\boxtimes f"] & \End_B \boxtimes \End^A_B \arrow[r] & \End^A_B\ ,
\\
	f \rhd \alpha  \ar[r,phantom,":" description] &
	\oC \arrow[r,"\Cop{(*)}{}"] &  \calC \ribt \oC 
	\arrow[r,"f\ribt \alpha"] & \End^A_B \ribt \End_A \arrow[r] & \End^A_B \ ,\\
	& \I \arrow[r,"\cong"] &  \I \boxtimes \I 
	\arrow[r,"f\boxtimes \alpha"] & \End^A_B \boxtimes \End_A \arrow[r] & \End^A_B\ ,
	\end{tikzcd}
	\]
where the rightmost arrows are given by the usual composition of functions.	
\end{definition}

\begin{proposition}\label{prop:infmor}
The data of an $\infty$-morphism $F : (\S\calC \square \S A, d_\alpha) \to (\S\calC \square \S B, d_\beta)$ 
is equivalent to the data of a morphism 
$f : \calC \to  \End^A_B$ 
of $\Sy$-bimodules satisfying 
\begin{equation}\label{eq:Morph}
\partial (f)=
f  \rhd \alpha - \beta \lhd f\ .
\end{equation}
\end{proposition}

\begin{proof}
Under \cref{lem:bifree}, the data of a  morphism $F : \S\calC \square \S A \to \S\calC \square \S B$ is equivalent to the data of a map of $\Sy$-bimodules $f : \calC \to \End^A_B$, explicitly given by 
\[\begin{tikzcd}[column sep=large]
	f\ : \ \calC \boxtimes A  \subset \S (\calC \boxtimes A)\cong \S \calC \sq \S A \ar[r,"F"] &
	\S \calC \sq \S B \ar[r,"\S(\eps) \sq \id"] &
	\S\I \sq \S B\cong \S B\ .
\end{tikzcd}\]
In the other way round, one recovers $F$ from $f$ by
\[\begin{tikzcd}[column sep=large]
F\ : \ 	\S \calC \sq \S A  \ar[r,"\S(\Delta)\sq \id"] &
	\S \calC \sq \S \calC \sq \S A \cong \S\calC \sq \S(\calC \boxtimes A)  \ar[r,"\S \calC \sq \S(f)"] &
	\S\calC \sq \S(\S B) \ar[r,"\id\sq \widetilde{\nu}"] & \S\calC \sq \S B \ .
\end{tikzcd}\]
It remains to show that the relation $\left(d_{\S\calC\sq \S B}+d_\beta\right) \circ F- F \circ \left(d_{\S\calC \sq \S A}+d_\alpha\right)=0$
is equivalent to the relation $\partial (f)-
f  \rhd \alpha + \beta \lhd f=0$ 
under these isomorphisms. Let us denote by $i : \calC \boxtimes A \hookrightarrow \S \calC \sq \S A$ the canonical inclusion.  Since $F$ is a morphism of monoid $\S\calC$-comodules and since $d_\alpha$ and $d_\beta$ are biderivations, the first relation holds if and only if the following composite vanishes 
\[(\S(\eps)\sq \id) \circ\left(\left(d_{\S\calC\sq \S B}+d_\beta\right) \circ F- F \circ \left(d_{\S\calC \sq \S A}+d_\alpha\right)\right)\circ i =0\ .\]
Under the isomorphism $\Hom_{\Sy}(\calC\boxtimes A, \S B)\cong \Hom_{\Sy}\big(\calC , \End^A_B\big)$, we claim  the following correspondances between the various terms 
\begin{eqnarray*}
-(\S(\eps)\sq \id) \circ F \circ d_{\S\calC\sq \S A}  \circ i
&\longleftrightarrow&  -f \circ d_\calC - \partial_A \circ f   \\
(\S(\eps)\sq \id)  \circ d_{\S\calC\sq \S B} \circ F  \circ i
&\longleftrightarrow& \partial_B \circ f\\
-(\S(\eps)\sq \id) \circ F \circ d_\alpha  \circ i
&\longleftrightarrow&    -f\rhd \alpha \\
(\S(\eps)\sq \id)  \circ d_\beta \circ F  \circ i
&\longleftrightarrow& \beta \lhd f\ ,
\end{eqnarray*}
where $\partial_A$ and $\partial_B$ stand respectively for the part of the differential of $\End_B^A$ made up of $d_A$ and $d_B$. The second correspondance relies on $\eps \circ d_\calC=0$. 
The sum of the first two terms on the right-hand side is equal to $\partial(f)$\ .
The first two correspondences only deal with the internal differentials of $\calC$, $A$ and $B$, in contrast with the other two which involve the algebraic structures $\alpha$ and $\beta$. The first correspondence is clear and the second one relies on $\eps \circ d_\calC=0$.

To get the third correspondance, one can apply the explicit formula \eqref{eq:dvarphi} using $\alpha$ for the biderivation $d_\alpha$ and the above explicit formula using $f$ for $F$. 
The composite on the left-hand side
 amounts to first applying the coproduct $\Delta$, then applying $f$ to every vertex of the bottom level and applying $\eps$ to every vertex except one which is mapped under $\alpha$ at the top level; this is nothing but the right action $f \rhd \alpha$. 
One can proceed similarly for the fourth correspondance. 
\end{proof}

\begin{remark}
As in the operadic case the operator $\lhd$ defines a left $\L_\infty$-module structure of the dg Lie algebra 
$\Hom_{\Sy}\big(\calC, \End_B\big)$ on 
$\Hom_{\Sy}\big(\calC, \End^A_B\big)$
and  the operator $\rhd$ defines a right $\L_\infty$-module structure of the dg Lie algebra 
$\Hom_{\Sy}\big(\calC, \End_A\big)$ on 
$\Hom_{\Sy}\big(\calC, \End^A_B\big)$. Equivalently, they endow 
\[ \Hom_{\Sy}\big(\calC, \End_B\big) \oplus 
\Hom_{\Sy}\big(\calC, \End^A_B\big)
\oplus \Hom_{\Sy}\big(\calC, \End_A\big)
\]
with an $\L_\infty$-algebra structure which extends the dg Lie algebra structures on 
$\Hom_{\Sy}\big(\calC, \End_A\big)$ and $\Hom_{\Sy}\big(\calC, \End_B\big)$. 
Given a Maurer--Cartan element $\alpha+\beta$, one can twist the above $\L_\infty$-algebra. The solutions concentrated in $f\in \Hom_{\Sy}\big(\calC, \End^A_B\big)$ to its Maurer--Cartan equation  
is equal to \cref{eq:Morph}. The advantage of such an interpretation is that it allows one to apply to $\infty$-morphisms of homotopy gebras all results and methods of the general deformation theory of $\L_\infty$-algebras, like the obstruction theory developed in \cref{sec:Obstruction}. 
With this interpretation, we will use the integration theory of $\L_\infty$-algebras to enrich simplicially the category 
of $\Cobar\calC$-gebras  with $\infty$-morphisms  
in a forthcoming paper. 
\end{remark}

\begin{proposition}[Naturality of $\infty$-morphisms]\label{prop:NatInfMorph}
Let $G : \calC \to \calD$ be a morphism of conilpotent dg coproperads, let 
$\alpha : \calD \to \End_A$ and $\beta : \calD \to \End_B$ be two $\Cobar \calD$-gebra structures, and let $f : \calD \to \End^A_B$ be an $\infty$-morphism from $\alpha$ to $\beta$. The composite 
\[
\begin{tikzcd}
Gf\ :  \ \calC \arrow[r,"G"] & 
\calD  \arrow[r,"f"]  & \End^A_B
\end{tikzcd}
\]
defines an $\infty$-morphism from the $\Cobar \calC$-gebra $\alpha G$ on $A$ to the $\Cobar \calC$-gebra $\beta G$ on $B$\ .
\end{proposition}

\begin{proof}
The natural bijections in the Rosetta Stone \cref{thm:CompleteRosetta} show that $\alpha G$ and $\beta G$ are $\Cobar \calC$-gebra structures. The fact that $Gf$ is an $\infty$-morphism can be proved directly from the definition and the natural bijection given in \cref{prop:EndofRosetta}. One can also prove it using the characterisation given in \cref{prop:infmor}: 
\begin{align*}
\partial(fG)-(fG)\rhd (\alpha G) + (\beta G)\lhd (fG)=
\big(\partial(f)-f\rhd \alpha + \beta \lhd f\big)\, G=0\ , 
\end{align*}
since $G$ is a morphism of dg coproperads.
\end{proof}

\begin{proposition}
Under the above isomorphism, the composite $G\circ F$ of two $\infty$-morphisms is given by 
\[\begin{tikzcd}
g\circledcirc f \ : \ \calC \arrow[r,"\Delta"] & 
\calC \boxtimes \calC \arrow[r,"g\boxtimes f"]  & 
\End^B_C \boxtimes \End_B^A \arrow[r]  & 
\End_C^A \  .
\end{tikzcd}\]
\end{proposition}

\begin{proof}
This is a direct corollary of the formulas given at the beginning of the proof of \cref{prop:infmor}. Using the same notations, the composite 
$(\S(\eps)\sq \id) \circ  G \circ F  \circ i$ is equal to 
$g\circledcirc f$
under the isomorphism $\Hom_{\Sy}(\calC\boxtimes A, \S B)\cong \Hom_{\Sy}\big(\calC , \End^A_B\big)$. 
\end{proof}

\begin{remark}
In a forthcoming paper, we will show that the 2-colored dg properad which encodes the data of  two $\Cobar \calC$-gebra structures related by a (strict) morphism admits a quasi-free (cofibrant) resolution, which is actually the 
2-colored dg properad which encodes the data of  two $\Cobar \calC$-gebra structures related by a $\infty$-morphism. Such a result gives a first homotopical justification for the present definition of $\infty$-morphisms. 
\end{remark}

\subsection{Infinity-isomorphisms}
Since we require the dg coproperad $\calC$ to be coaugmented, one can single out a first component 
\[\begin{tikzcd}
f_{(0)} \ : \ \I  \arrow[r,hook] &  \I \oplus \oC \cong \calC  \arrow[r,"f"] & \End^A_B 
\end{tikzcd}\]
of every $\infty$-morphism.

\begin{definition}[$\infty$-isomorphism]
An \emph{$\infty$-isomorphism} is an $\infty$-morphism $f : A \rightsquigarrow B$ whose first component $f_{(0)} : A \xrightarrow{\cong} B$ is an isomorphism. 
\end{definition}

\begin{theorem}\label{prop:Inverse}
When $\calC$ is a conilpotent coproperad, the class of $\infty$-isomorphisms is the class of isomorphisms of the category $\infty\textsf{-}\Cobar\calC\textsf{-}\mathsf{gebras}$. 
\end{theorem}

\begin{proof}
Let $f$ be an $\infty$-isomorphism. We consider the map $f^{-1} : \calC \to \End^B_A$ defined by  $\left(f^{-1}\right)_{(0)}\coloneqq f_{(0)}^{-1}$ and by  
\begin{equation}
\begin{tikzcd}
\oC  \arrow[r,"\Delta_{\oC}"] & \G^c\Big(\oC\Big)   \arrow[r,"\widetilde{\G}^c(f)"] & \End^B_A\ ,
\end{tikzcd}
\end{equation}
where the image of an element of $\gg\Big(\oC\Big)$, for $\gg\in \bGs$, under the map $\widetilde{\G}^c(f)$ amounts to applying $f$ to all the vertices, to labeling all the edges including the leaves by $f_{(0)}^{-1}$ and by multiplying the result with the coefficient $(-1)^{|\gg|}$, where $|\gg|$ stands for the number of vertices. We claim that $f^{-1}$ is an $\infty$-morphism inverse to $f$. 

Let us first prove that $f^{-1}$ is right-inverse to $f$. The restriction of the composite 
$f^{-1}\circledcirc f$ to $\I$ amounts to sending $\id$ to $f^{-1}_{(0)}\circ f_{(0)}=\id_A$.  By definition, the restriction of the composite $f^{-1}\circledcirc f$ to $\oC$ amounts to considering the composite $\left(\Delta_{\oC}\boxtimes \id\right)\circ \Delta$ and then labeling the top vertices by $f$, the bottom vertices by $f^{-1}$ and add a sign corresponding to the total number of bottom vertices. By the definition of comonadic coproperads, the composite $\left(\Delta_{\oC}\boxtimes \id\right)\circ \Delta$ is equal to the bottom-left part of the diagram \eqref{eq:CoalComonad} composed with the projection onto the module made up of partitioned graphs that can be organised with a top level of vertices. The commutativity of this diagram shows that this composite is equal to $\Delta_{\oC}$ followed by the sum of all the possible ways to partition the vertices of a graph such that there is one non-empty top level made up of vertices. 
\begin{figure}[h!]
\begin{tikzpicture}
	\draw[thick] (0.2,-0.5) -- (0.2,5.5);
	\draw[thick] (0.5,0) -- (0.5,2);
	\draw[thick] (1,-0.5) -- (1,0) -- (1.5,1) --  (1,2) -- (1,5.5);
	\draw[thick] (1.5,4.5) -- (1.5,5.5);
	\draw[thick] (1.5,-0.5) -- (1.5,0) -- (2,1);
	\draw[thick] (6.2,-0.5) -- (6.2,1.5) -- (5.7,2);
	\draw[thick] (6.5,-0.5) -- (6.5,1.5) -- (6,2) ;
	\draw[thick] (6.3,2) -- (6.3,5.5);
	\draw[thick] (7.5,-0.5) -- (7.5,5.5);
	\draw[thick] (8.5,2) -- (8.5,5.5);
	\draw[thick]  (1.7,2.5) to[out=90,in=270] (5.8,4) -- (5.8,5.5);
	\draw[draw=white,double=black,double distance=2*\pgflinewidth,thick]  (5.5,2.5) to[out=90,in=270] (2.5,4) -- (2.5,5.5);
	\draw[dotted] (1.2,1.2) circle (1.9cm);
	\draw[dotted] (6.9,1.4) circle (2.1cm);
	\draw[dotted] (1.9,4.7) ellipse (1.4cm and 0.7cm);
	\draw[dotted] (6.7,4.7) ellipse (1.4cm and 0.7cm);
	\draw[fill=white] (0,0) rectangle (2,0.5);
	\draw[fill=white] (1,1) rectangle (3,1.5);
	\draw[fill=white] (0,2) rectangle (2,2.5);
	\draw[fill=white] (6,0) rectangle (8,0.5);
	\draw[fill=white] (6,1) rectangle (8,1.5);
	\draw[fill=white] (5.2,2) rectangle (6.5,2.5);
	\draw[fill=white] (7.2,2) rectangle (8.7,2.5);
	\draw[fill=white] (0.7,4.5) rectangle (3,5);
	\draw[fill=white] (5.5,4.5) rectangle (7.8,5);
	\draw[<->,red,thick] (1,2.2) -- (2,4.2);
	\draw[red] (1.6,3.4) node [left] {$\pm$};
\end{tikzpicture}
\end{figure} 

For any reduced graph $\gg\in \bGs$, let us consider one top vertex. It can appear twice in the image of these composites: either on the top level or on a sub-graph located on the bottom level. The resulting final operation in $\End_A$ is the same up to a minus sign.  This proves that the restriction of $f^{-1}\circledcirc f$ to $\oC$ is trivial.
The proof that $f^{-1}$ is left-inverse to $f$ is symmetric. 

Let us now prove in a similar way that $f^{-1}$ is an $\infty$-morphism. Using the notations and the result of \cref{prop:infmor}, we have to show that $\partial(f^{-1})=f^{-1} \rhd\beta - \alpha \lhd f^{-1}$\ . On $\I$, the left-hand side amounts to $d_A \circ f_{(0)}^{-1}-f_{(0)}^{-1}\circ d_B$, which vanishes since $f_{(0)}$ is a chain map. The right-hand side also vanishes since $\alpha|_{\I}=0$ and $\beta|_{\I}=0$ by the definition of twisting morphisms. On $\oC$ now, the left-hand side is equal to 
\[\partial(f^{-1})=\partial^B_A\circ f^{-1} - f^{-1}\circ d_{\oC}=
\widetilde{\G}^c\left(f; \partial^B_A\circ f\right)\circ \Delta_{\oC}-
\widetilde{\G}^c\left(f; f\circ d_{\oC}\right)\circ \Delta_{\oC}=
\widetilde{\G}^c\left(f; \partial(f)\right)\circ \Delta_{\oC}
\ , \]
where $\partial^B_A$ stands for the differential of $\End_A^B$ and where $\widetilde{\G}^c(f;g)$ stands for the same map as $\widetilde{\G}^c(f)$ but where one applies $f$ to all vertices except one to which one applies $g$. 
We use \cref{eq:Morph}: $\partial (f)=
f  \rhd \alpha - \beta \lhd f$. It remains to show 
\begin{align*}
&\widetilde{\G}^c\left(f; f  \rhd \alpha\right)\circ \Delta_{\oC}=-\alpha \lhd \left(
\widetilde{\G}^c\left(f\right)\circ \Delta_{\oC}
\right)=-\alpha \lhd f^{-1} \quad  \text{and}\\
&\widetilde{\G}^c\left(f; \beta \lhd f\right)\circ \Delta_{\oC}=-\left(
\widetilde{\G}^c\left(f\right)\circ \Delta_{\oC}
\right)\rhd \beta=-f^{-1} \rhd \beta \ .
\end{align*}
We use the same argument as above. The left-hand side of the first relation amounts to applying first the composite 
$\G^c\left(\id;\Delta_{(*,1)}\right) \circ \Delta_{\oC}$, which is equal to the bottom-left part of the diagram \eqref{eq:CoalComonad} composed with the projection onto the module made up partitioned graphs where all the vertices form one block except for one block which contains a 2-level sub-graph with only one vertex on its top level. 
Then, one applies $\alpha$ to this vertex and $f$ to all the other vertices; all the edges, including the leaves, are labeled by $f_{(0)}^{-1}$ except for the edges in the block which are labeled by the identity. 
Using the commutativity of the diagram \eqref{eq:CoalComonad}, the composite $\G^c\left(\id;\Delta_{(*,1)}\right) \circ \Delta_{\oC}$ is equal to $\Delta_{\oC}$ followed by the sum of all the possible ways to partition the vertices of a graph under the above shape. 
\begin{figure}[h!]
	\begin{tikzpicture}
	\draw[thick] (2,-0.8) -- (2,.2) to[out=90,in=270] (2.5,7) -- (2.5,8);
	\draw[thick] (2.8,0.2) -- (4.3,3) --  (4.3,3.5) -- (3.4,5);
	\draw[thick] (4.6,3.5) --  (3.7,5) -- (3.7,5.5) -- (4.2,6) -- (4.6,6.5) -- (4.6,8) ;
	\draw[thick] (3.3,5.5) to[out=90,in=270] (3.1,7);
	\draw[thick] (4,6.5) -- (3.4, 7) -- (3.4,8);
	\draw[thick] (2.3,-.8) -- (2.3,0.2) -- (3.3,2);
	\draw[thick] (4.7,-.8) -- (4.7,.2) -- (4.2,2);
	\draw[thick] (5.3,-.8) -- (5.3,.2) -- (5.9,2) -- (5.9,2.5) -- (5.7,3);
	\draw[thick] (7.1,-.8) -- (7.1,.2) -- (6.7,2);
	\draw[thick] (7.7,-.8) -- (7.7,.2);
	\draw[thick] (5.7,3.5) -- (5.3,5) -- (5.3,8);
	\draw[thick] (6.2,5.5) -- (6.2,8);
	\draw[thick] (5.9,5.5) -- (5.9,6.5);
	\draw[thick] (5.9,7.5) -- (5.9,8);
	\draw[thick] (7,6.5) -- (7,8);
	\draw[fill=white](1.7,-0.3) rectangle (3.2,0.2);
	\draw[fill=white](4.3,-.3) rectangle (5.7,0.2);
	\draw[fill=white](6.8,-.3) rectangle (8,0.2);
	\draw[fill=white](3,2) rectangle (4.5,2.5);
	\draw[fill=white](5.5,2) rectangle (7,2.5);
	\draw[fill=white](4,3) rectangle (6,3.5);
	\draw (3.75,2.2) node {$f$};
	\draw (6.25,2.2) node {$f$};
	\draw[<->,red,thick] (6.5,0.6) -- (6.5,2.2);
	\draw[red] (6.5,1.1) node [right] {$\pm$};
	\draw (5,3.2) node {$\varphi$};
	\draw[dotted] (5,2.5) ellipse (2.2cm and 1.8cm);	
	\draw[fill=white](3,5) rectangle (4.1,5.5);
	\draw[fill=white](5,5) rectangle (6.5,5.5);
	\draw[fill=white](3.5,6) rectangle (5,6.5);
	\draw[fill=white](5.6,6) rectangle (7.3,6.5);
	\draw[fill=white](2.2,7) rectangle (3.7,7.5);
	\draw[fill=white](5,7) rectangle (6.5,7.5);
	\end{tikzpicture}
\end{figure} 
We claim that only the summands where $\alpha$ is applied to a bottom vertex of a graph survive. 
Since any graph can appear above this vertex, the sign implies that we get exactly $-\alpha\lhd f^{-1}$\ . 
When there is a non-trivial vertex below the one, called $v$, where $\alpha$ applies, there are two cases: either this vertex appears in the same block as $v$ or not. In the end, these two elements produce the same map in $\End_A^B$ but with a different sign, so they cancel. 
\end{proof}

An important class of $\infty$-isomorphisms is given by the following notion. 

\begin{definition}[$\infty$-isotopy]
An \emph{$\infty$-isotopy} is an $\infty$-morphism $f : (A,\alpha) \rightsquigarrow (A,\beta)$ whose first component $f_{(0)}=\id : A \to A$ is the identity. 
\end{definition}

This notion admits the following homotopical interpretation: the dg properad which encodes 
the data of  two $\Cobar \calC$-gebra structures on the same underlying dg module related by an $\infty$-isotopy
is a cylinder for the cobar construction $\Cobar \calC$. This implies that the equivalence relation defined by being $\infty$-isotopic is equivalent to the left homotopy equivalence of morphisms $\Cobar \calC \to \End_A$ in the model category of dg properads \cite[Appendix~A]{MerkulovVallette09I}, see \cite[Section~5.2.3]{Fresse09ter}. 
The methods of \cite[Section~8]{Fresse10ter} show that it is also equivalent to the usual homotopy equivalence: the existence of a zig-zag of (strict) quasi-isomorphisms. We refer the reader to a forthcoming paper for a detailed exposition. 

\medskip

As we will see in the sequel, the following weaker notion carries suitable homotopy properties. For instance, they admit a ``homology'' inverse, see \cref{thm:InvInfQI}. 

\begin{definition}[$\infty$-quasi-isomorphism]
An \emph{$\infty$-quasi-isomorphism} is an $\infty$-morphism $f : A \rightsquigarrow B$ whose first component $f_{(0)} : A \xrightarrow{\cong} B$ is a quasi-isomorphism. 
\end{definition}

\section{The homotopy transfer theorem}\label{sec:HTT}

\subsection{Properadic Van der Laan morphism}

\begin{definition}[Contraction]
A \emph{contraction} of a chain complex $(A, d_A)$ is another chain complex $(H,d_H)$ equipped with chain maps $i$ and $p$ and a 
 homotopy $h$ of degree $1$
  \[
\begin{tikzcd}
(A,d_A)\arrow[r, shift left, "p"] \arrow[loop left, distance=1.5em,, "h"] & (H,d_H) \arrow[l, shift left, "i"]
\end{tikzcd} \ ,
\]
satisfying 
\begin{align*}
pi=\id_H\ , \quad ip -\id_A = d_Ah+hd_A\ , \quad 
hi=0\ , \quad ph=0\ , \quad \text{and} \quad h^2=0 \ .
\end{align*}
\end{definition}

We denote the projection $\pi\coloneqq ip$ and the identity of $A$ simply by $\id\coloneqq \id_A$. 
For any positive integer $n$, we consider the following symmetric homotopies 
$$ h_n:=\frac{1}{n!}\sum_{\sigma \in \Sy_n}  \sum_{k=1}^n \left(
\id^{\otimes (k-1)} \otimes h \otimes \pi^{\otimes (n-k)}
\right)^\sigma$$
from $\pi^{\otimes n}$ to $\id^{\otimes n}$. We denote generically by capital letters collections of map indexed by integers, like $\H$ for the collection of homotopies $\{h_n\}_{n\geqslant 1}$ and respectively  by $\II$, $\PP$, $\Pi$, and $\Id$ for the collections of maps $\{i^{\otimes n }\}_{n\geqslant 1}$, $\{p^{\otimes n }\}_{n\geqslant 1}$, $\{\pi^{\otimes n }\}_{n\geqslant 1}$, and $\{\id^{\otimes n }\}_{n\geqslant 1}$ \ .

\medskip

Let us recall from \cref{rem:BarEnd} that the underlying $\Sy$-bimodule of the bar construction of the endomorphism properad $\Bar\End_A$ is given  by 
$\I\oplus \G^c\big(s\End_A\big)$. 
We consider the set $\GsLev$ of non-trivial directed connected graphs with levels, where each level contains only one vertex.
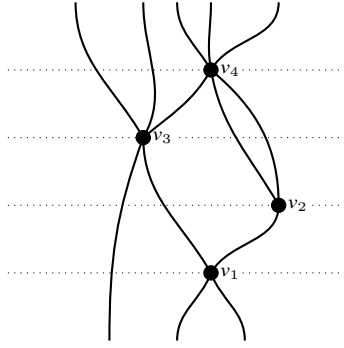
\[
\begin{tikzpicture}[scale=0.9,baseline=(n.base)]
\node (n) at (1,1.5) {}; %point base
\draw[dotted] (-3,0) -- (2,0);
\draw[dotted] (-3,1) -- (2,1);
\draw[dotted] (-3,2) -- (2,2);
\draw[dotted] (-3,3) -- (2,3);
\draw[white,fill=white] (0.2,-0.2) rectangle (0.4,0.2);
\draw[fill=black] (0,0) circle (3pt) node[right] {$\scriptstyle{v}_1$};
\draw[white,fill=white] (1.11,0.8) rectangle (1.45,1.2);
\draw[fill=black] (1,1) circle (3pt) node[right] {$\scriptstyle{v}_2$};
\draw[white,fill=white] (-0.9,1.9) rectangle (-0.5,2.1);
\draw[fill=black] (-1,2) circle (3pt) node[right] {$\scriptstyle{v}_3$};
\draw[white,fill=white] (0.2,2.8) rectangle (0.4,3.2);
\draw[fill=black] (0,3) circle (3pt) node[right] {$\scriptstyle{v}_4$};
\draw[thick] (0,0) to[out=60,in=270] (1,1) to[out=120,in=280] (0,3) to[out=320,in=90] (1,1) ;
\draw[thick] (0,0) to[out=120,in=270] (-1,2) to[out=40,in=240] (0,3);
\draw[thick] (-1,2) to[out=60,in=270] (-1,4) node {};
\draw[thick] (-1,2) to[out=120,in=270] (-2,4) node {};
\draw[thick] (-1,2) to[out=250,in=90] (-1.5,-1) node {};
\draw[thick] (0,3) to[out=100,in=270] (0,4) node {};
\draw[thick] (0,3) to[out=60,in=270] (1,4) node {};
\draw[thick] (0,3) to[out=120,in=270] (-0.5,4) node {};
\draw[thick] (0,0) to[out=250,in=90]  (-0.5,-1) node {};
\draw[thick] (0,0) to[out=290,in=90]  (0.5,-1) node {};
\end{tikzpicture}
\]

\begin{definition}[Leveled graph module]
The \emph{leveled graph  module} is defined by 
\[
\GLev(\M) \coloneqq \bigoplus_{\gg \in \GsLev} \gg(\M) \ , 
\]
where $g(\M)\subset (\I \oplus \M)^{\boxtimes k}$, where $k=|\gg|$ the number of vertices of $\gg$, is the module generated by leveled graphs whose vertices are labeled by elements of $M$.
\end{definition}

\begin{definition}[Levelization morphism]
The \emph{levelization morphism} 
\[\lev \ : \G^c(\M)  \longrightarrow \GLev(\M) \]
sends a graph to the sum of all the ways to put all the vertices on levels. 
\end{definition}

\begin{remark}
The levelization process amounts to considering all the ways to refine the partial order on the set of vertices given by the underlying directed graph into a total order. 
\end{remark}

\begin{definition}[Map $\PHI$] 
The degree $-1$ linear map 
 \[\PHI: \GLev\left(s\End_A\right) \to \End_H\]
is defined by removing all the suspensions $s$, labeling the top level of leaves by $\II$, labeling any intermediate level with $\H$, and labeling the bottom level of leaves by $\PP$: 
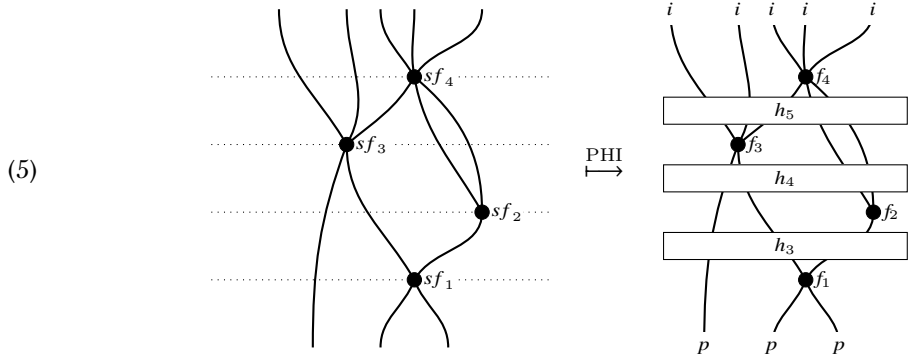
\begin{equation}\label{eq:dessin_PHI}
		\begin{tikzpicture}[scale=0.9,baseline=(n.base)]
			\node (n) at (1,1.5) {}; %point base
			\draw[dotted] (-3,0) -- (2,0);
			\draw[dotted] (-3,1) -- (2,1);
			\draw[dotted] (-3,2) -- (2,2);
			\draw[dotted] (-3,3) -- (2,3);
			\draw[white, fill=white] (.1,-.05) rectangle (.6,.05);
			\draw[fill=black] (0,0) circle (3pt) node[right] {$\scriptstyle{s f}_1$};
			\draw[white, fill=white] (1.1,.95) rectangle (1.6,1.05);
			\draw[fill=black] (1,1) circle (3pt) node[right] {$\scriptstyle{sf}_2$};
			\draw[white, fill=white] (-0.9,1.95) rectangle (-.35,2.05);
			\draw[fill=black] (-1,2) circle (3pt) node[right] {$\scriptstyle{sf}_3$};
			\draw[white, fill=white] (.1,2.95) rectangle (.6,3.05);
			\draw[fill=black] (0,3) circle (3pt) node[right] {$\scriptstyle{sf}_4$};
			\draw[thick] (0,0) to[out=60,in=270] (1,1) to[out=120,in=280] (0,3) to[out=320,in=90] (1,1) ;
			\draw[thick] (0,0) to[out=120,in=270] (-1,2) to[out=40,in=240] (0,3);
			\draw[thick] (-1,2) to[out=60,in=270] (-1,4) node {};
			\draw[thick] (-1,2) to[out=120,in=270] (-2,4) node {};
			\draw[thick] (-1,2) to[out=250,in=90] (-1.5,-1) node {};
			\draw[thick] (0,3) to[out=100,in=270] (0,4) node {};
			\draw[thick] (0,3) to[out=60,in=270] (1,4) node {};
			\draw[thick] (0,3) to[out=120,in=270] (-0.5,4) node {};
			\draw[thick] (0,0) to[out=250,in=90]  (-0.5,-1) node {};
			\draw[thick] (0,0) to[out=290,in=90]  (0.5,-1) node {};
		\end{tikzpicture}
		\quad \overset{\PHI}{\longmapsto} \quad
		\begin{tikzpicture}[scale=0.9,baseline=(n.base)]
			\node (n) at (1,1.5) {}; %point base
			\draw[fill=black] (0,0) circle (3pt) node[right] {$\scriptstyle{f}_1$};
			\draw[fill=black] (1,1) circle (3pt) node[right] {$\scriptstyle{f}_2$};
			\draw[fill=black] (-1,2) circle (3pt) node[right] {$\scriptstyle{f}_3$};
			\draw[fill=black] (0,3) circle (3pt) node[right] {$\scriptstyle{f}_4$};
			\draw[thick] (0,0) to[out=60,in=270] (1,1) to[out=120,in=280] (0,3) to[out=320,in=90] (1,1) ;
			\draw[thick] (0,0) to[out=120,in=270] (-1,2) to[out=40,in=240] (0,3);
			\draw[thick] (-1,2) to[out=60,in=270] (-1,4) node[fill=white] {$\scriptstyle{i}$};
			\draw[thick] (-1,2) to[out=120,in=270] (-2,4) node[fill=white] {$\scriptstyle{i}$};
			\draw[thick] (-1,2) to[out=250,in=90] (-1.5,-1) node[fill=white] {$\scriptstyle{p}$};
			\draw[thick] (0,3) to[out=100,in=270] (0,4) node[fill=white] {$\scriptstyle{i}$};
			\draw[thick] (0,3) to[out=60,in=270] (1,4) node[fill=white] {$\scriptstyle{i}$};
			\draw[thick] (0,3) to[out=120,in=270] (-0.5,4) node[fill=white] {$\scriptstyle{i}$};
			\draw[thick] (0,0) to[out=250,in=90]  (-0.5,-1) node[fill=white] {$\scriptstyle{p}$};
			\draw[thick] (0,0) to[out=290,in=90]  (0.5,-1) node[fill=white] {$\scriptstyle{p}$};
			\draw[fill=white] (-2.1,0.3) rectangle (1.5,0.7);
			\draw (-0.3,0.5) node {$\scriptstyle{{h}}_3$};
			\draw[fill=white] (-2.1,1.3) rectangle (1.5,1.7);
			\draw (-0.3,1.5) node {$\scriptstyle{{h}}_4$};
			\draw[fill=white] (-2.1,2.3) rectangle (1.5,2.7);
			\draw (-0.3,2.5) node {$\scriptstyle{{h}}_5$};
		\end{tikzpicture}\ .
\end{equation}
\end{definition}

\begin{remark}
Notice that the map $\PHI$ produces signs due to the application of the Koszul sign rule when permuting elements. 
 \end{remark}

\begin{theorem}\label{prop:VanDerLaan}
The map $\varphi$ in $\Hom_{\bbS}(\B\End_A,\End_H)$ defined on non-trivial elements by the composition 
\[\begin{tikzcd}
\G^c\big(s\End_A\big) \arrow[r,"\lev"] & \GLev\big(s\End_A\big)  \arrow[r,"\PHI"] & \End_H
\end{tikzcd}\]
and by $\varphi|_\I\coloneqq 0$ 
is a twisting morphism.
\end{theorem}

\begin{definition}[Van der Laan map] 
We call $\varphi$ the \emph{Van der Laan twisting morphism} and we call the induced morphisms 
\[F_\varphi : \Cobar\Bar\End_A \to \End_H\qquad \text{and} \qquad G_\varphi : \Bar\End_A \to \Bar\End_H\]
of dg properads and coaugmented dg coproperads the \emph{Van der Laan morphisms}. 
\end{definition}

\begin{remark}
Given a contraction of $A$ onto $H$, there is no chance to produce a morphism of dg properads from $\End_A$ to $\End_H$ by simply 
 pulling back by $i$ and pushing forward by $p$. 
 On the other hand, the Van der Laan map provides us with a canonical \emph{$\infty$-quasi-isomophism} of (homotopy) properads from $\End_A$ to $\End_H$.
\end{remark}

\begin{proof}[{Proof of \cref{prop:VanDerLaan}}]

For any non-trivial graph $\gg\in \bGs$, an element of the graph module $\gg(s\End_A\big)$ 
can be written $\gg(sf_1, \ldots, s f_k)$, where $k\coloneqq |\gg|$ is the number of vertices of $\gg$ and where $f_1, \ldots, f_k\in \End_A$. Implicitly, we chose a planar representation of the graph $\gg$ where the vertices labeled from $1$ to $k$ can be read from left to right and from bottom to top. 
We denote by $\left(\widetilde{\gg}, \sigma\right)$ any levelization of the graph $\gg$ together with the induced permutation $\sigma$ of the way the vertices are read. Under such notations, the levelization is written 
\[\lev\big(\gg(sf_1, \ldots, s f_k)\big)=\sum_{\left(\widetilde{\gg}, \sigma\right)}
\varepsilon_\sigma\, \widetilde{\gg}\left(sf_{\sigma(1)}, \ldots, sf_{\sigma(k)}\right)\ ,
\]
where $\widetilde{\gg}\left(sf_{\sigma(1)}, \ldots, sf_{\sigma(k)}\right)$ is the left-hand side of Figure~\ref{eq:dessin_PHI} and 
where the sign $\varepsilon_\sigma$ comes from the permutation of terms $sf_i$ according to $\sigma$. 
Its composite with the map $\PHI$ gives 
\[\varphi\big(\gg(sf_1, \ldots, s f_k)\big)=\sum_{\left(\widetilde{\gg}, \sigma\right)}
\varepsilon_\sigma\, \widetilde{\gg}\left(\PP, f_{\sigma(1)}, \H, \ldots, \H, f_{\sigma(k)}, \II\right)
\ ,\]
using the slight but comprehensible extension of notation  $\widetilde{\gg}\left(\PP, f_{\sigma(1)}, \H, \ldots, \H, f_{\sigma(k)}, \II\right)$ for the right-hand side of Figure~\ref{eq:dessin_PHI}.

By definition, we need to prove the equation $\partial(\varphi) + \varphi \star \varphi =0$. To do so, we evaluate both terms on an element $\gg(sf_1, \ldots, s f_k)$.
The first term is equal to 
\[\partial(\varphi)\big(\gg(sf_1, \ldots, s f_k)\big)=
\partial_H\big(\varphi\big(\gg(sf_1, \ldots, s f_k)\big)\big)+
\varphi \big((d_1+d_2)\big(\gg(sf_1, \ldots, s f_k)\big)\big)\ , 
\]
where $\partial_H$ stands for the differential of $\End_H$. 
The sum of the two terms involving the underlying differentials, that is $\partial_H$ and $d_1$, is equal to the sum of the leveled graph composition 
$\varphi\big(\gg(sf_1, \ldots, s f_k)\big)$ with one differential $d_A$ labeling every input or output edge of every level labeled by an $h_n$: 
\begin{align*}
&\partial_H\big(\varphi\big(\gg(sf_1, \ldots, s f_k)\big)\big)+
\varphi \big(d_1\big(\gg(sf_1, \ldots, s f_k)\big)\big)=\\
&\sum_{i=1}^{k-1} \eps_i \sum_{\left(\widetilde{\gg}, \sigma\right)}
\eps_\sigma\, \widetilde{\gg}\big(\PP, f_{\sigma(1)}, \H, \ldots, f_{\sigma(i)}, \partial_A(\H), f_{\sigma(i+1)}, \ldots, \H, f_{\sigma(k)}, \II\big)=\\
&\sum_{i=1}^{k-1} \eps_i \sum_{\left(\widetilde{\gg}, \sigma\right)}
\eps_\sigma\, \widetilde{\gg}\left(\PP, f_{\sigma(1)}, \H, \ldots, f_{\sigma(i)}, \Pi-\Id, f_{\sigma(i+1)}, \ldots, \H, f_{\sigma(k)}, \II\right)\ ,
\end{align*}
where $\eps_i=(-1)^{|f_{\sigma(1)}|+\cdots +|f_{\sigma(i)}|+i-1}$. 
We claim that 
\begin{equation}\label{eqn:Rel1}\tag{i} 
\varphi \big(d_2\big(\gg(sf_1, \ldots, s f_k)\big)\big)=
\sum_{i=1}^{k-1} \eps_i
\sum_{\left(\widetilde{\gg}, \sigma\right)}
\eps_\sigma\, \widetilde{\gg}\left(\PP, f_{\sigma(1)}, \H, \ldots, f_{\sigma(i)}, \Id, f_{\sigma(i+1)}, \ldots, \H, f_{\sigma(k)}, \II\right)
\end{equation}
and that 
\begin{equation}\label{eqn:Rel2}\tag{ii} 
(\varphi\star \varphi)\big(\gg(sf_1, \ldots, s f_k)\big)=
-\sum_{i=1}^{k-1} \eps_i
\sum_{\left(\widetilde{\gg}, \sigma\right)}
\eps_\sigma\, \widetilde{\gg}\left(\PP, f_{\sigma(1)}, \H, \ldots, f_{\sigma(i)}, \Pi, f_{\sigma(i+1)}, \ldots, \H, f_{\sigma(k)}, \II\right)\ ,
\end{equation}
which would conclude the proof.

To prove Relation~\eqref{eqn:Rel1}, we use the fact, proved in \cite[Lemma~6.6]{Vallette07}, that the levelization map commutes with the part $d_2$ of the differential of the bar construction and the differential of the simplicial bar construction, which amounts to composing all pairs of levels of operations. Under the present notation, this means 
\[
\lev\big(d_2\big(\gg(sf_1, \ldots, s f_k)\big)\big)=
\sum_{i=1}^{k-1} \eps_i
\sum_{\left(\widetilde{\gg}, \sigma\right)}
\eps_\sigma\, \widetilde{\gg}\left( sf_{\sigma(1)}, \ldots, s\gamma\big(f_{\sigma(i)}, f_{\sigma(i+1)}\big), \ldots, sf_{\sigma(k)}\right)\ ,
\]
where $\gamma\big(f_{\sigma(i)}, f_{\sigma(i+1)}\big)$ means the composite of the two levels (possibly disconnected) with $f_{\sigma(i)}$ at the bottom and $f_{\sigma(i+1)}$ at the top. Composing with the map $\PHI$, this gives Relation~\eqref{eqn:Rel1}:
\begin{align*}
\varphi \big(d_2\big(\gg(sf_1, \ldots, s f_k)\big)\big)&=
\sum_{i=1}^{k-1} \eps_i
\sum_{\left(\widetilde{\gg}, \sigma\right)}
\eps_\sigma\, \widetilde{\gg}\left(\PP, f_{\sigma(1)}, \H, \ldots,\gamma\big(f_{\sigma(i)}, f_{\sigma(i+1)}\big), \ldots, \H, f_{\sigma(k)}, \II\right)\\
&=
\sum_{i=1}^{k-1} \eps_i
\sum_{\left(\widetilde{\gg}, \sigma\right)}
\eps_\sigma\, \widetilde{\gg}\left(\PP, f_{\sigma(1)}, \H, \ldots, f_{\sigma(i)},\Id,f_{\sigma(i+1)}, \ldots, \H, f_{\sigma(k)}, \II\right)\ .
\end{align*}
The left-hand side of Relation~\eqref{eqn:Rel2} amounts to cut in all possible ways the graph $g$ into 2 connected sub-graphs and then to apply $\varphi=\PHI \circ \lev$ to each of them. This produces the same result as first applying the levelisation map, considering only the levels (called admissible)  which cut the graph into 2 connected leveled-sub-graphs, labeling this level by $\Pi$, the other internal levels by $\H$ and the bottom and top levels by $\PP$ and $\II$ respectively. Under the present notations, this gives 
\begin{align*}
(\varphi\star \varphi)\big(\gg(sf_1, \ldots, s f_k)\big)=
-\sum_{i : \text{admissible}} \eps_i
\sum_{\left(\widetilde{\gg}, \sigma\right)}
\eps_\sigma\, \widetilde{\gg}\left(\PP, f_{\sigma(1)}, \H, \ldots, f_{\sigma(i)}, \Pi, f_{\sigma(i+1)}, \ldots, \H, f_{\sigma(k)}, \II\right)\ .
\end{align*}
The similar terms but coming from non-admissible cuts produce at least two sub-graphs below or above the cut. 
The image of the concatenation of reduced graphs  under the composite of the levelization morphism and the map $\PHI$ is trivial by \cref{TechLem1} below.
 In the end, we do get Relation~\eqref{eqn:Rel2}: 
\begin{align*}
(\varphi\star \varphi)\big(\gg(sf_1, \ldots, s f_k)\big)=
-\sum_{i=1}^{k-1} \eps_i
\sum_{\left(\widetilde{\gg}, \sigma\right)}
\eps_\sigma\, \widetilde{\gg}\left(\PP, f_{\sigma(1)}, \H, \ldots, f_{\sigma(i)}, \Pi, f_{\sigma(i+1)}, \ldots, \H, f_{\sigma(k)}, \II\right)\ .
\end{align*}
\end{proof}

\begin{lemma}\label{TechLem1}
The following composite is trivial 
\[\begin{tikzcd}
\G^c\big(s\End_A\big)^{\otimes 2}
 \arrow[r,"\mu"] 
 \ar[rrr, bend right=15, "0" description, ]
 & 
\ncG\big(s\End_A\big)
 \arrow[r,"\lev"] 
 & \ncGLev\big(s\End_A\big)  \arrow[r,"\PHI"] & \End_H\ ,
\end{tikzcd}\]
where $\ncG$ (resp. $\ncGLev$) stands for endofunctor made up of non-necessarily connected directed reduced (resp. leveled) graphs and 
where $\mu$ stands for the concatenation of graphs.
\end{lemma}

\begin{proof}
Let us recall, for instance from \cite[Lemma 7]{DSV16}, the following relations 
\begin{align}
& \left(\id^{\otimes k}\otimes p^{\otimes l}\right)h_{k+l} = h_{k}\otimes p^{\otimes l} \ , \tag{a} \label{Rel(a)}\\
& h_k\otimes h_l=\left(h_k\otimes \id^{\otimes l}\right)h_{k+l} + h_{k+l}\left(\id^{\otimes k}\otimes h_l\right) \ , \tag{b} \label{Rel(b)}\\
& \left(h_k\otimes \id^{\otimes l}\right)h_{k+l} =- h_{k+l}\left(h_k\otimes \id^{\otimes l}\right) \ , \tag{c} \label{Rel(c)}
\end{align}
for any $k,l\geqslant 1$. 
The image of an element of $\G^c\big(s\End_A\big)^{\otimes 2}$ under the composite $\PHI\circ\lev\circ\mu$  can be depicted as follows, where we only represent the non-trivial vertices. The top thick line represents the map $\II$, the bottom thick map represents the map $\PP$, the intermediate thin lines represent maps $\H$, and the dashed lines represent some sums of maps (made up of $h_n$'s).  
The above-mentioned relations are interpreted as rewriting rules in order pull from bottom to top some $h_n$. (We not make the various signs explicit.) 
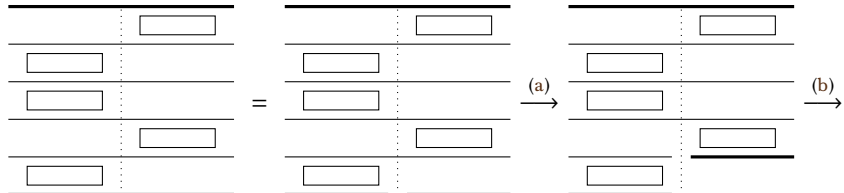
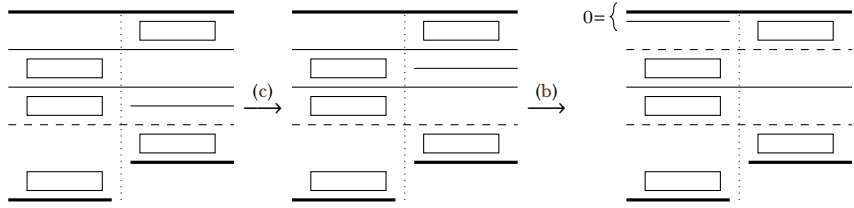
\begin{align*}
&	%%% DESSIN 1
	\begin{tikzpicture}[scale=0.5,baseline=(n.base)]
	\node (n) at (2.5,2) {};
	\node at (0, -1) {};
	\draw[dotted] (2.5,-0.25) -- (2.5,4.75); %ligne verticale
	\draw[very thick] (-.5,-0.25) --(5.5,-.25); %P
	\draw[fill=white] (0,0) rectangle (2,0.5); % Boite 1 
	\draw (-.5,0.75) --(5.5,0.75) ; %H1
	\draw[fill=white] (3,1) rectangle (5,1.5); %Boite 2 
	\draw (-.5,1.75) --(5.5,1.75) ; %H2
	\draw[fill=white] (0,2) rectangle (2,2.5);%Boite3
	\draw (-.5,2.75) --(5.5,2.75) ;%H3
	\draw[fill=white] (0,3) rectangle (2,3.5);%B4
	\draw (-.5,3.75) --(5.5,3.75) ;%H4
	\draw[fill=white] (3,4) rectangle (5,4.5);%B5
	\draw[very thick] (-.5,4.75) --(5.5,4.75); %I 
	\end{tikzpicture}
\ = \
	%%% DESSIN 2
	\begin{tikzpicture}[scale=0.5,baseline=(n.base)]
	\node (n) at (2.5,2) {};
	\draw[dotted] (2.5,-0.25) -- (2.5,4.75); %ligne verticale
	\draw[very thick] (-.5,-0.25) -- (2.25,-.25) ; %P1
	\draw[very thick] (2.75,-0.25) -- (5.5,-.25) ; %P2
	\draw[fill=white] (0,0) rectangle (2,0.5); % Boite 1 
	\draw (-.5,0.75) --(5.5,0.75) ; %H1
	\draw[fill=white] (3,1) rectangle (5,1.5); %Boite 2 
	\draw (-.5,1.75) --(5.5,1.75) ; %H2
	\draw[fill=white] (0,2) rectangle (2,2.5);%Boite3
	\draw (-.5,2.75) --(5.5,2.75) ;%H3
	\draw[fill=white] (0,3) rectangle (2,3.5);%B4
	\draw (-.5,3.75) --(5.5,3.75) ;%H4
	\draw[fill=white] (3,4) rectangle (5,4.5);%B5
	\draw[very thick] (-.5,4.75) --(5.5,4.75); %I 
	\end{tikzpicture}
\stackrel{\eqref{Rel(a)}}{\longrightarrow}
	%%% DESSIN 3
	\begin{tikzpicture}[scale=0.5,baseline=(n.base)]
	\node (n) at (2.5,2) {};
	\draw[dotted] (2.5,-0.25) -- (2.5,4.75); %ligne verticale
	\draw[very thick] (-.5,-0.25) -- (2.25,-.25) ; %P1
	\draw[fill=white] (0,0) rectangle (2,0.5); % Boite 1 
	\draw[very thick] (2.75,0.75) -- (5.5,0.75) ; %P2
	\draw (-.5,0.75) --(2.25,0.75) node[right] {}; %H1.5
	\draw[fill=white] (3,1) rectangle (5,1.5); %Boite 2 
	\draw (-.5,1.75) --(5.5,1.75) ; %H2
	\draw[fill=white] (0,2) rectangle (2,2.5);%Boite3
	\draw (-.5,2.75) --(5.5,2.75) ;%H3
	\draw[fill=white] (0,3) rectangle (2,3.5);%B4
	\draw (-.5,3.75) --(5.5,3.75) ;%H4
	\draw[fill=white] (3,4) rectangle (5,4.5);%B5
	\draw[very thick] (-.5,4.75) --(5.5,4.75); %I 
	\end{tikzpicture}
\stackrel{\eqref{Rel(b)}}{\longrightarrow}
\end{align*}
\begin{align*}
&	%%% DESSIN 4
	\begin{tikzpicture}[scale=0.5,baseline=(n.base)]
	\node (n) at (2.5,2) {};
	\draw[dotted] (2.5,-0.25) -- (2.5,4.75); %ligne verticale
	\draw[very thick] (-.5,-0.25) -- (2.25,-.25) ; %P1
	\draw[fill=white] (0,0) rectangle (2,0.5); % Boite 1 
	\draw[very thick] (2.75,0.75) -- (5.5,0.75) ; %P2
	\draw[fill=white] (3,1) rectangle (5,1.5); %Boite 2 
	\draw[dashed] (-.5,1.75) --(5.5,1.75) ; %H2BIS
	\draw (2.75,2.25) --(5.5,2.25) ; %H2.5
	\draw[fill=white] (0,2) rectangle (2,2.5);%Boite3
	\draw (-.5,2.75) --(5.5,2.75) ;%H3
	\draw[fill=white] (0,3) rectangle (2,3.5);%B4
	\draw (-.5,3.75) --(5.5,3.75) ;%H4
	\draw[fill=white] (3,4) rectangle (5,4.5);%B5
	\draw[very thick] (-.5,4.75) --(5.5,4.75); %I 
	\end{tikzpicture}\stackrel{\eqref{Rel(c)}}{\longrightarrow}
 	%%% DESSIN 5
	\begin{tikzpicture}[scale=0.5,baseline=(n.base)]
	\node (n) at (2.5,2) {};
	\draw[dotted] (2.5,-0.25) -- (2.5,4.75); %ligne verticale
	\draw[very thick] (-.5,-0.25) -- (2.25,-.25) ; %P1
	\draw[fill=white] (0,0) rectangle (2,0.5); % Boite 1 
	\draw[very thick] (2.75,0.75) -- (5.5,0.75) ; %P2
	\draw[fill=white] (3,1) rectangle (5,1.5); %Boite 2 
	\draw[dashed] (-.5,1.75) --(5.5,1.75) ; %H2BIS
	\draw[fill=white] (0,2) rectangle (2,2.5);%Boite3
	\draw (-.5,2.75) --(5.5,2.75) ;%H3
	\draw (2.75,3.25) --(5.5,3.25) ; %H3.5
	\draw[fill=white] (0,3) rectangle (2,3.5);%B4
	\draw (-.5,3.75) --(5.5,3.75) ;%H4
	\draw[fill=white] (3,4) rectangle (5,4.5);%B5
	\draw[very thick] (-.5,4.75) --(5.5,4.75); %I 
	\end{tikzpicture}	
\stackrel{\eqref{Rel(b)}}{\longrightarrow}
%%% DESSIN 6
	\begin{tikzpicture}[scale=0.5,baseline=(n.base)]
	\node (n) at (2.5,2) {};
	\draw[dotted] (2.5,-0.25) -- (2.5,4.75); %ligne verticale
	\draw[very thick] (-.5,-0.25) -- (2.25,-.25) ; %P1
	\draw[fill=white] (0,0) rectangle (2,0.5); % Boite 1 
	\draw[very thick] (2.75,0.75) -- (5.5,0.75) ; %P2
	\draw[fill=white] (3,1) rectangle (5,1.5); %Boite 2 
	\draw[dashed] (-.5,1.75) --(5.5,1.75) ; %H2BIS
	\draw[fill=white] (0,2) rectangle (2,2.5);%Boite3
	\draw (-.5,2.75) --(5.5,2.75) ;%H3
	\draw[fill=white] (0,3) rectangle (2,3.5);%B4
	\draw[dashed] (-.5,3.75) --(5.5,3.75) ;%H4
	\draw[fill=white] (3,4) rectangle (5,4.5);%B5
	\draw (-.5,4.5) --(2.25,4.5) node[right] {};  %H4.5
	\draw[very thick] (-.5,4.75) --(5.5,4.75); %I 
	\draw[decorate, decoration={brace}] (-0.75,4.25) -- (-0.75,5) ;
	\draw (-0.8,4.625) node[left] {$\scriptstyle{0=}$};
	\end{tikzpicture}	
\end{align*}
In the end, each component contains at the top a composite of the form $h_k\circ  i^{\otimes k}=0$, which concludes the proof. 
\end{proof}

\subsection{Universal $\infty$-morphisms}
Recall that the counit of the bar-cobar adjunction provides us with a morphism of dg properads $\Cobar\Bar \End_A \longrightarrow \End_A$, which  can be interpreted as a canonical $\Cobar\Bar \End_A$-gebra structure on $A$. The Van der Laan morphism $F_\varphi : \Cobar\Bar \End_A \longrightarrow \End_H$ endows $H$ with a $\Cobar\Bar \End_A$-gebra structure.

\begin{definition}[Maps $\HHI$ and $\PHH$] 
The degree $0$ linear maps 
 \[\HHI \ :\  \GLev\left(s\End_A\right) \to \End^H_A \quad \text{and} \quad \PHH \ : \ \GLev\left(s\End_A\right) \to \End^A_H\]
are defined  by removing all the suspensions $s$, labeling the top level of leaves by $\II$ (resp. $\H$), labeling any intermediate level with $\H$, and labeling the bottom level of leaves by $\H$ (resp. $\PP$).  
\end{definition}

\begin{theorem}\label{prop:UnivInfMorph}\leavevmode
\begin{itemize}

\item[$\diamond$] The map $i_\infty$ in $\Hom_{\bbS}\big(\B\End_A,\End^H_A\big)$ defined on non-trivial elements by the composition 
\[\begin{tikzcd}
\G^c\big(s\End_A\big) \arrow[r,"\lev"] & \GLev\big(s\End_A\big)  \arrow[r,"\HHI"] & \End^H_A
\end{tikzcd}\]
and on $\I$ by $i$ 
is an $\infty$-morphism of $\Cobar \B \End_A$-gebras. 

\item[$\diamond$] The map $p_\infty$ in $\Hom_{\bbS}\big(\B\End_A,\End^A_H\big)$ defined on non-trivial elements by the composition 
\[\begin{tikzcd}
\G^c\big(s\End_A\big) \arrow[r,"\lev"] & \GLev\big(s\End_A\big)  \arrow[r,"\PHH"] & \End^A_H
\end{tikzcd}\]
and on $\I$ by $p$ 
is an $\infty$-morphism of $\Cobar \B \End_A$-gebras. 
\end{itemize}
\end{theorem}

\begin{proof}
This proof carries similar computations than the proof of \cref{prop:VanDerLaan}, so we follow the same notations. Let us prove that $i_\infty$ is an $\infty$-morphism. The proof that $p_\infty$ is an $\infty$-morphism is similar and  obtained by a vertical symmetry. We use the classical notation $\epsilon : \Bar \End_A \to \End_A$ for  the universal twisting morphism corresponding to the counit of adjunction, i.e. to the $\Cobar \Bar \End_A$-gebra structure on $A$. We need to prove that $i_\infty$ satisfies Equation~\eqref{eq:Morph}: 
\[\partial (i_\infty)=i_\infty  \rhd \varphi - \epsilon \lhd i_\infty\ .\]
We evaluate it on a generic element $\gg(sf_1, \ldots, s f_k)$ of $\G^c(s\End_A)$\ . 
Notice first that 
\[i_\infty\big(\gg(sf_1, \ldots, s f_k)\big)=\sum_{\left(\widetilde{\gg}, \sigma\right)}
\varepsilon_\sigma\, \widetilde{\gg}\left(\H, f_{\sigma(1)}, \H, \ldots, \H, f_{\sigma(k)}, \II\right)
\ ,\]
By the same argument as in the proof of  \cref{prop:VanDerLaan}, we get 
\begin{align*}
&\partial^H_A\big(i_\infty\big(\gg(sf_1, \ldots, s f_k)\big)\big)-
i_\infty \big(d_1\big(\gg(sf_1, \ldots, s f_k)\big)\big)=\\
&-\sum_{i=0}^{k-1} \eps_i \sum_{\left(\widetilde{\gg}, \sigma\right)}
\eps_\sigma\, \widetilde{\gg}\big(\H, f_{\sigma(1)}, \H, \ldots, f_{\sigma(i)}, \partial_A(\H), f_{\sigma(i+1)}, \ldots, \H, f_{\sigma(k)}, \II\big)=\\
&-\sum_{i=0}^{k-1} \eps_i \sum_{\left(\widetilde{\gg}, \sigma\right)}
\eps_\sigma\, \widetilde{\gg}\left(\H, f_{\sigma(1)}, \H, \ldots, f_{\sigma(i)}, \Pi-\Id, f_{\sigma(i+1)}, \ldots, \H, f_{\sigma(k)}, \II\right)\ ,
\end{align*}
where $\partial^H_A$ stands for the differential of $\End_A^H$ and where $\varepsilon_0\coloneqq 1$\ .
The argument proving Relation~\eqref{eqn:Rel1} in the proof of  \cref{prop:VanDerLaan} shows as well that 
\begin{equation*}
i_\infty \big(d_2\big(\gg(sf_1, \ldots, s f_k)\big)\big)=
\sum_{i=1}^{k-1} \eps_i
\sum_{\left(\widetilde{\gg}, \sigma\right)}
\eps_\sigma\, \widetilde{\gg}\left(\H, f_{\sigma(1)}, \H, \ldots, f_{\sigma(i)}, \Id, f_{\sigma(i+1)}, \ldots, \H, f_{\sigma(k)}, \II\right)\ .
\end{equation*}
So the left-hand side of Equation~\eqref{eq:Morph} is equal to 
\begin{eqnarray*}
\partial(i_\infty)\big(\gg(sf_1, \ldots, s f_k)\big)&=&-\sum_{i=0}^{k-1} \eps_i \sum_{\left(\widetilde{\gg}, \sigma\right)}
\eps_\sigma\, \widetilde{\gg}\big(\H, f_{\sigma(1)}, \H, \ldots, f_{\sigma(i)}, \Pi, f_{\sigma(i+1)}, \ldots, \H, f_{\sigma(k)}, \II\big)\\
&&- \sum_{\left(\widetilde{\gg}, \sigma\right)}
\eps_\sigma\, \widetilde{\gg}\left(\Id,f_{\sigma(1)}, \H, \ldots,  \H, f_{\sigma(k)}, \II\right)\ .
\end{eqnarray*}
We claim that the first term on the right-hand side is equal to 
$i_\infty  \rhd \varphi$ and the second term on the right-hand side is equal to  $- \epsilon \lhd i_\infty$. 

The first term of the right-hand side is similar to the right-hand side of Relation~\eqref{eqn:Rel2} in the proof of  \cref{prop:VanDerLaan}, except that the bottom map is $\H$ or $\Pi$ here instead of $\PP$. The same arguments apply as well. First, the part above the map $\Pi$ can  contain only one non-trivial connected graph by \cref{TechLem1}, thereby corresponding to a single application of $\varphi$. Then, the part below the map $\Pi$ is equal to the bottom-left composite of the commutative diagram of \cref{TechLem2}, i.e. one considers first the global levelization of possibly disconnected graphs and then one applies the map $\HHI$. By \cref{TechLem2}, this is equal to first applying the levelization map to all the connected components and then applying to each of them the map $\HHI$; such a composite amounts to applying the map $i_\infty$ to all the connected components below the map $\Pi$. 
This proves the first claim. 

The second term on the right-hand side is equal to a first sum indexed by a choice of a bottom vertex and then a second sum which amounts to levelize the remaining above vertices. The label of the bottom vertex corresponds to applying once the map $\epsilon$.  By \cref{TechLem2}, applying the map $\HHI$ after the global levelization of the above possibly disconnected graph is equal to applying the levelization map to each of its connected components first and then applying to each of them the map $\HHI$; this corresponds to applying the map $i_\infty$ to each connected component. This proves the second claim and concludes the proof. 
\end{proof}

\begin{lemma}\label{TechLem2}
For any $k\geqslant 1$, the following diagram is commutative
	\[
	\begin{tikzcd}
		\G^c(s\End_A)^{\otimes k} \ar[r,"\lev^{\otimes k}"] \ar[d,"\mu^{k-1}"'] 
		& \GLev (s\End_A)^{\otimes k} \ar[r,"\HHI^{\otimes k}"]
		& \left(\End^H_A\right) ^{\otimes k} \ar[d,"\mu^{k-1}"]  \\
		\ncG(s\End_A) \ar[r,"\lev"] 
		& \ncGLev (s\End_A) \ar[r,"\HHI"]
		& \End^H_A\ , 
	\end{tikzcd} 
	\]
where we use the same notations as in \cref{TechLem1}.
\end{lemma}

\begin{proof}
	For $k=1$, it is trivial. Let us prove the statement for any $k\geqslant 2$ by induction on the number $N\geqslant 2$ of vertices of the  graphs of $\G^c(s\End_A)^{\otimes k}$. 
	
	For $N=2$, we have two horizontally concatenated vertices, and so $k=2$. For $f_1\in \End_A(m,p)$ and $f_2\in \End_A(n,q)$, a direct computation gives 
	\begin{align*}
	&\HHI  \circ   \lev \circ \mu\left(sf_1 \otimes sf_2\right) =
	\HHI\left( 
	(sf_1\otimes \id^{\otimes n})(\id^{\otimes p}\otimes sf_2) +(-1)^{(|f_1|+1)(|f_2|+1)} 
	(\id^{\otimes m}\otimes sf_2)(sf_1\otimes \id^{\otimes q})
	\right)\\
	&= h_{m+n}(f_1\otimes \id^{\otimes n})h_{p+n}(\id^{\otimes k}\otimes f_2) i^{\otimes (p+q)}
	+(-1)^{(|f_1|+1)(|f_2|+1)} 
	h_{m+n}(\id^{\otimes m}\otimes f_2)h_{m+q}(f_1\otimes \id^{\otimes q})
	i^{\otimes (p+q)}\\
	& = h_{m+n}(f_1\otimes \id^{\otimes n})(i^{\otimes p}\otimes h_n)(\id^{\otimes p}\otimes f_2i^{\otimes q})
	+(-1)^{(|f_1|+1)(|f_2|+1)} 
	h_{m+n}(\id^{\otimes m}\otimes f_2)(h_m\otimes i^{\otimes q})(f_1i^{\otimes p}\otimes \id^{\otimes q})\\
	& = (-1)^{|f_1|}h_{m+n}(\id^{\otimes m}\otimes h_n) (f_1 i^{\otimes p}\otimes f_2i^{\otimes q})
	-(-1)^{|f_1|} 
	h_{m+n}(h_m\otimes \id^{\otimes n}) (f_1 i^{\otimes p}\otimes f_2i^{\otimes q})\\
	& =(-1)^{|f_1|}(h_m\otimes h_n)(f_1 i^{\otimes p}\otimes f_2i^{\otimes q})
	=h_mf_1 i^{\otimes p}\otimes h_nf_2i^{\otimes q}
	= \mu \circ \HHI^{\otimes 2}  \circ   \lev^{\otimes 2} \left(sf_1 \otimes sf_2\right)\ ,
	\end{align*}
	where the third equality comes from Relation~\eqref{Rel(a)} of the proof of \cref{TechLem1}, applied to $i$ instead of $p$, and where the fifth equality comes from Relation~\eqref{Rel(b)}+\eqref{Rel(c)}. 
	
	Suppose now that the induction hypothesis holds up to $N\geqslant 2$, that is
	\[
	\mu^{k-1}\circ \HHI^{\otimes k}\circ \lev^{\otimes k} 
	 (\gamma_1 \otimes \cdots \otimes \gamma_k)
	= \HHI\circ\lev\circ \mu^{k-1}(\gamma_1 \otimes \cdots \otimes \gamma_k)\ ,
	\]
	for  
	$\gamma_1, \cdots, \gamma_k\in \G^c(s\End_A)$ with a total number of vertices equal to $N$. Let us prove that the result still holds for $N+1$, by induction on $k\geqslant 2$. For $k=2$, 
	let $\gamma_1\coloneq \gg_1(sf^1_1, \ldots, sf^1_m)\in \G^c(s\End_A)$ and $\gamma_2\coloneq \gg_2(sf^2_1, \ldots, sf^2_n) \in \G^c(s\End_A)$ such that $m+n=N+1$. Using the same conventions as above,
	 we have 
	\begin{align*}
	&\mu\circ\big(\HHI\circ \lev(\gamma_1) \otimes\HHI\circ\lev(\gamma_2)\big)= \\
	&\sum_{\left(\widetilde{\gg_1}, \sigma_1\right)\atop \left(\widetilde{\gg_2}, \sigma_2\right)}
	\varepsilon_{\sigma_1}\varepsilon_{\sigma_2}\, 
	\widetilde{\gg_1}\left(\H, f^1_{\sigma_1(1)}, \H, \ldots, \H, f^1_{\sigma_1(m)}, \II\right)\otimes
	\widetilde{\gg_2}\left(\H, f^2_{\sigma_2(1)}, \H, \ldots, \H, f^2_{\sigma_2(n)}, \II\right)=\\
	&\sum_{\left(\widetilde{\gg_1}, \sigma_1\right)\atop \left(\widetilde{\gg_2}, \sigma_2\right)}
	\varepsilon_{\sigma_1}\varepsilon_{\sigma_2}(-1)^{|f^1|+m-1}\, (h_p\otimes h_q)\left(
	\widetilde{\gg_1}\left(f^1_{\sigma_1(1)}, \H, \ldots, \H, f^1_{\sigma_1(m)}, \II\right)\otimes
	\widetilde{\gg_2}\left(f^2_{\sigma_2(1)}, \H, \ldots, \H, f^2_{\sigma_2(n)}, \II\right)\right)\ , 
	\end{align*}
	where $|f^1|=|f^1_1|+\cdots+|f^1_m|$, where $p$ and $q$ are respectively the numbers of outputs of the graphs $\gg_1$ and $\gg_2$. Using Relation~\eqref{Rel(b)}+\eqref{Rel(c)} of the proof of \cref{TechLem1}, we get  
	\begin{align*}
	&\mu\circ\big(\HHI\circ \lev(\gamma_1) \otimes\HHI\circ\lev(\gamma_2)\big)= \\
	&\sum_{\left(\widetilde{g_1}, \sigma_1\right)\atop \left(\widetilde{g_2}, \sigma_2\right)}
	\varepsilon_1\, 
	h_{p+q}\left(f^1_{\sigma_1(1)}\otimes \id^{\otimes q}\right)
	\left(
	\widetilde{g_1}'\left(\H, f^1_{\sigma_1(2)}, \H, \ldots, \H, f^1_{\sigma_1(m)}, \II\right)\otimes
	\widetilde{g_2}\left(\H, f^2_{\sigma_2(1)}, \H, \ldots, \H, f^2_{\sigma_2(n)}, \II\right)\right)+ \\
	&\sum_{\left(\widetilde{g_1}, \sigma_1\right)\atop \left(\widetilde{g_2}, \sigma_2\right)}
	\varepsilon_2\, 
	h_{p+q}\left(\id^{\otimes p}\otimes f^2_{\sigma_2(1)}\right)
	\left(
	\widetilde{g_1}\left(\H, f^1_{\sigma_1(1)}, \H, \ldots, \H, f^1_{\sigma_1(m)}, \II\right)\otimes
	\widetilde{g_2}'\left(\H, f^2_{\sigma_2(2)}, \H, \ldots, \H, f^2_{\sigma_2(n)}, \II\right)\right)\ ,
	\end{align*}
	where $\varepsilon_1=\varepsilon_{\sigma_1}\varepsilon_{\sigma_2}$, where 
	$\varepsilon_2=\varepsilon_{\sigma_1}\varepsilon_{\sigma_2}(-1)^{(|f^2_{\sigma_2(1)}|+1)(|f^1|+m)}$, and 
	where $\widetilde{g_1}'$ and 
	$\widetilde{g_2}'$ stand respectively for the possibly disconnected leveled graphs obtained by removing the bottom vertex labeled by $f^1_{\sigma_1(1)}$ and by $f^2_{\sigma_2(1)}$. Using the induction hypothesis to the corresponding labeled subgraphs $\gamma_1'\otimes \gamma_2$ and $\gamma_1\otimes \gamma_2'$ which have $N$ vertices, we get
	\begin{align*}
	\mu\circ\big(\HHI\circ \lev(\gamma_1) \otimes\HHI\circ\lev(\gamma_2)\big)&
	= \sum_{\left(\widetilde{\gg}, \sigma\right)\atop \mathrm{bot}(\widetilde{\gg})=1}\varepsilon_\sigma\, 
	\widetilde{\gg}\left(\H, f^1_{\sigma(1)}, \H, \ldots,  \II\right) +
	\sum_{\left(\widetilde{\gg}, \sigma\right)\atop \mathrm{bot}(\widetilde{\gg})=2}\varepsilon_\sigma\, 
	\widetilde{\gg}\left(\H, f^2_{\sigma(1)}, \H, \ldots,  \II\right)\\
	&=\HHI\circ\lev\circ \mu(\gamma_1\otimes \gamma_2)\ ,
	\end{align*} 
	where the graph $\gg\coloneq \gg_1\otimes \gg_2$ is the concatenation of $\gg_1$ and $\gg_2$ and where 
	$\mathrm{bot}(\widetilde{\gg})=1$ or $\mathrm{bot}(\widetilde{\gg})=2$ mean that the bottom vertex of the leveled graph $\widetilde{\gg}$ lies on the left-hand side (1) or on the right-hand side (2). 
	We now suppose that the result holds true for $k\geqslant 2$. If we have $k+1$ horizontally concatenated connected graphs, then we apply the above argument to the pair of graphs made up of the first $k$ concatenated graphs and the last one. We conclude with the induction hypothesis. 
\end{proof}

\subsection{The homotopy transferred structure}
We can pullback the above mentioned universal structures (Van der Laan twisting morphism and universal $\infty$-morphisms)
by any $\Cobar \calC$-gebra structure on $A$
 to get the following homotopy transfer theorem. 

\begin{theorem}[Homotopy Transfer Theorem]\label{thm:HTT}
Given a $\Cobar \calC$-gebra structure $\alpha \in \Tw(\calC, \End_A)$ on $A$ and a contraction from $A$ onto $H$, the composite 
\[\begin{tikzcd}[column sep=normal]
	\oC\arrow[r,"G_\alpha"] &  \Bar\End_A \arrow[r,"\varphi"] & \End_H 
\end{tikzcd}\]
defines a $\Cobar \calC$-gebra structure on $H$ and the composites 
\[
\begin{tikzcd}[column sep=normal, row sep=tiny]
	& \oC\arrow[r,"G_\alpha"] &  \Bar\End_A \arrow[r,"i_\infty"] & \End^H_A & \text{and} & 
\oC\arrow[r,"G_\alpha"] &  \Bar\End_A \arrow[r,"p_\infty"] & \End^A_H
\end{tikzcd}\]
are  two $\infty$-quasi-isomorphisms from $H$ to $A$ extending $i$ and respectively from $A$ to $H$ extending $p$.
\end{theorem}

\begin{proof}
This is a direct corollary of \cref{prop:NatInfMorph} applied to \cref{prop:VanDerLaan} and \cref{prop:UnivInfMorph}. 
\end{proof}

\begin{definition}[Transferred structure]
The new $\Cobar\calC$-gebra structure obtained on $H$ under the above theorem is called 
the \emph{transferred $\Cobar\calC$-gebra structure}. Since it is related to the original $\Cobar\calC$-gebra structure  on $A$ by $\infty$-quasi-isomorphisms, it is \emph{weakly equivalent} to it. 
 \end{definition}
 
Notice that the transferred $\Cobar\calC$-gebra structure on $H$ is equal to 
	\[
	\begin{tikzcd}[column sep=normal]
	\oC\arrow[r,"\Delta_{\oC}"] & \G^c\big(\oC\big) \arrow[r,"\G^c(s\alpha)"] & \Bar\End_A \arrow[r,"\varphi"] & \End_H 
	\end{tikzcd}
	\]
	in terms of a twisting morphism in $\Tw(\calC,\End_H)$.
This comes from the fact that the morphism $G_\alpha : \calC \to \Bar \End_A$ of conilpotent dg coproperads associated to the twisting morphism $\alpha$ by \cref{thm:CompleteRosetta} is equal on non-trivial elements to 
the composite $\G^c(s\alpha)\circ \Delta_{\oC}$. Since $s\alpha$ is of degree $0$ and since the map $\PHI$ in 
the composite $\varphi=\PHI\circ \lev$ 
does not produce any extra sign, the only sign comes the comonadic decomposition map  $\Delta_{\oC}$ followed by the levelisation map. 

\begin{remark}
	This theorem generalizes the operadic one of \cite[Appendix~B]{GCTV12} and  \cite[Theorem 10.3.1]{LodayVallette12}). 
\end{remark}

\begin{remark} Notice that labeling all the internal edges of a graphs with the contracting homotopy $h$, as done in the particular case of homotopy unimodular Lie bialgebras in \cite{Merkulov10}, cannot produce the general homotopy transfer formula. This does not produce homotopies with the correct homological degree: 
for instance, the homotopy for the involutivity relation of an involutive Lie bialgebra, see \cref{subsec:IBL}, would then become of degree $2$, instead of $1$. 
\end{remark}

\subsection{Applications}

\begin{theorem}[Homological invertibility of $\infty$-quasi-isomorphisms]\label{thm:InvInfQI}
	Let $A$ and $B$ be two $\Omega\calC$-gebras and let 
	$f:A \overset{\sim}{\rightsquigarrow} B$ be an $\infty$-quasi-isomorphism. There exists an $\infty$-quasi-isomorphism $g:B \overset{\sim}{\rightsquigarrow} A$ 
whose first component induces 	the homology inverse of the first component of $f$. 
\end{theorem}	

\begin{proof}
Since we are working over a field, we can choose once and for all a contraction of the chain complex $A$ (respectively $B$) into its homology $\H(A)$ (respectively $\H(B)$). Let us denote by $\alpha$ and $\beta$ respectively, the $\Cobar \calC$-gebra structures on $A$ and $B$. 
	By the homotopy transfer theorem \ref{thm:HTT}, the following composite of $\infty$-quasi-isomorphisms
	\[
	\begin{tikzcd}
	\mathrm{H}(A) \ar[r,squiggly,"i_\infty G_\alpha"] 
	& A \ar[r,squiggly,"f"]
	& B \ar[r,squiggly,"p_\infty G_\beta"]
	& \mathrm{H}(B)
	\end{tikzcd}
	\]
	is an $\infty$-isomorphism. By \cref{prop:Inverse}, it admits an inverse $\infty$-isomorphism $\widetilde{g} :\mathrm{H}(B) \overset{\cong}\rightsquigarrow \mathrm{H}(A)$. In the end, the following composite 
	\[
	\begin{tikzcd}
	g \ : \ 
	B \ar[r,squiggly,"p_\infty G_\beta"] 
	& \mathrm{H}(B) \ar[r,squiggly,"\widetilde{g}"]
	& \mathrm{H}(A) \ar[r,squiggly,"i_\infty G_\alpha"]
	& A
	\end{tikzcd} 
	\]
	produces the required $\infty$-quasi-isomorphism. 
\end{proof}

\begin{corollary}
	Let $A$ and $B$ be two $\Omega\calC$-gebras. If there exists a zig-zag of quasi-isomorphisms of 
	$\Omega\calC$-gebras 
	\[
	\begin{tikzcd}
	A \ar[r,"\sim"] 
	& \bullet 
	& \bullet \ar[l,"\sim"'] \ar[r,dotted, no head]
	&\bullet
	& \bullet \ar[l,"\sim"'] \ar[r,"\sim"] 
	& B
	\end{tikzcd} 
	\]
	then there exists an $\infty$-quasi-isomorphism $A \overset{\sim}{\rightsquigarrow} B$ (and  
	a $\infty$-quasi-isomorphism $B \overset{\sim}{\rightsquigarrow} A$) whose first component induces a homology 	isomorphism. 
\end{corollary}	

\begin{proof}
This is a direct corollary of \cref{thm:InvInfQI}. 
\end{proof}

\begin{remark}
In the operadic case, one can use the rectification \cite[Section~11.4]{LodayVallette12} to prove the converse property. 
Since the rectification of homotopy algebras over an operad relies on the free algebra construction, it cannot hold as such for homotopy gebras over a properad. So we do not know whether the fondamental equivalence ``zig-zig of quasi-isomorphisms''-``$\infty$-quasi-isomorphism'' still holds on the properadic level. 
\end{remark}

\section{Obstruction theory}\label{sec:Obstruction}

Let us suppose that the conilpotent dg coproperad $\calC$ is weight graded, i.e. $\calC=\bigoplus_{n\geqslant 0} \calC_{(n)}$, with $\calC_{(0)}=\I$, with a differential which lowers this weight by $-1$: $d_\calC :  \calC_{(n)} \to \calC_{(n-1)}$\ . This is the case when $\calC$ is the Koszul dual $\P^{\ac}$ of an inhomogenous quadratic properad \cite[Appendix~A]{GCTV12} or when it is the bar construction $\Bar \P$ of an augmented properad. 
Under this assumption, we can apply the standard methods \emph{mutatis mutandis} to develop the obstruction theory for $\infty$-morphisms of $\Cobar \calC$-gebras, as it was done in \cite[Appendix~A]{Vallette14} on the operadic level. 

\subsection{Main result}
Recall from \Cref{prop:infmor} that an $\infty$-morphism $f: (A, \alpha) \rightsquigarrow (B, \beta)$ between two $\Omega\calC$-gebras  is a map $f:\calC\rightarrow \End^A_B$ satisfying Equation~\eqref{eq:Morph}: 
\begin{equation*}
	\partial(f) = f \rhd \alpha - \beta \lhd f\ .
\end{equation*}
For any $n\geqslant 0$, we denote by 
\[\alpha_{(n)}\coloneqq \alpha|_{\calC_{(n)}}, \quad \beta_{(n)}\coloneqq \beta|_{\calC_{(n)}} \quad \text{and}\quad f_{(n)}\coloneqq f|_{\calC_{(n)}}\ ,\] the respective restrictions. Recall that $\alpha_{(0)}=0$ and $\beta_{(0)}=0$. 
Using these notations, \Cref{eq:Morph}, once evaluated on each $\calC_{(n)}$, 
 splits with respect to the weight grading and becomes equivalent to the following system of equations
\begin{equation}\label{eq::eqaution4Vallette}
	\partial^A_B f_{(n)} -f_{(n-1)}  d_\calC = \sum_{k=1}^n \left( f_{(\leqslant n-k)} \rhd \alpha_{(k)} - \beta_{(k)} \lhd f_{(\leqslant n-k)} \right)\ ,
\end{equation}
 indexed by $n\geqslant 0$. 

\begin{theorem}\label{thm:Obstr}
Let $\calC$ be a weight graded conilpotent dg coproperad and let $\alpha\in \Tw(\calC,\End_A)$ and $\beta\in \Tw(\calC,\End_B)$ be two $\Omega\calC$-gebra structures. Suppose that we are given $f_{(0)},\ldots,f_{(n-1)}$ satisfying \Cref{eq::eqaution4Vallette} up to $n-1$. The element
	\[
		\widetilde{f}_{(n)}\coloneqq \sum_{k=1}^n \left( f_{(\leqslant n-k)} \rhd \alpha_{(k)} - \beta_{(k)} \lhd f_{(\leqslant n-k)} \right) +f_{(n-1)} d_\calC  
	\] 
	is a cycle in the chain complex $\left( \Hom_{\bbS}\big(\calC,\End^A_B\big), \big(\partial^A_B\big)_* \right)$. Therefore, there exists an element $f_{(n)}$ satisfying \Cref{eq::eqaution4Vallette} at weight $n$ if and only if the cycle $\widetilde{f}_{(n)}$ is a boundary element.
\end{theorem}
\begin{proof}
	Let us prove that $\partial^A_B  \widetilde{f}_{(n)}=0$; the second statement is straightforward. Distributing the differentials everywhere, we get 
\begin{align*}
	\partial^A_B \widetilde{f}_{(n)}  = & 
	\sum_{k=1}^n \left( \big(f;\partial^A_B  f\big)_{(\leqslant n-k)} \rhd \alpha_{(k)} + \beta_{(k)} \lhd \big(f;\partial^A_B  f\big) _{(\leqslant n-k)} 
	+ f_{(\leqslant n-k)} \rhd (\partial_A \alpha_{(k)}) - (\partial_B\beta_{(k)}) \lhd f_{(\leqslant n-k)}
	 \right) \\
	&   + \partial^A_B  f_{(n-1)}d_\calC \ ,
\end{align*} 
	where $\big(f;\partial_B^A f\big)$ is our standard notation for $f$ applied everywhere except at one place where 
	$\partial_B^A f$ is applied. Similarly to \eqref{eq:Morph}, the  Maurer--Cartan equation satisfied by $\alpha$ and $\beta$ also splits with respect to the weight grading:
	\[
		\partial_A\alpha_{(n)} = - \sum_{k+l=n\atop k,l \geqslant 1} \alpha_{(k)}\star  \alpha_{(l)} - \alpha_{(n-1)}d_\calC 
		\quad \mbox{ and } \quad 
		\partial_B\beta_{(n)} = - \sum_{k+l=n\atop k,l \geqslant 1} \beta_{(k)}\star  \beta_{(l)} - \beta_{(n-1)}d_\calC \ . 
	\]
Using them and the induction hypothesis for $f_{(0)},\ldots,f_{(n-1)}$, we get by substitution
	\begin{align*}
		\partial^A_B\widetilde{f}_{(n)}  = &
\Big( (f;f \rhd \alpha) \rhd \alpha - (f;\beta \lhd f ) \rhd \alpha +(f;f d_\calC)\rhd \alpha \\
&+\beta \lhd (f;f \rhd \alpha)-\beta \lhd (f;\beta \lhd f)  
		+\beta \lhd (f;f d_\calC )
		 \\
		&    -f\rhd (\alpha\star  \alpha) - f \rhd (\alpha d_\calC) + (\beta\star  \beta) \lhd f + (\beta d_\calC) \lhd f 
		 \\
		& +  (f \rhd \alpha) d_\calC + (\beta  \lhd f)d_\calC + f(d_\calC )^2\Big)|_{\calC_{(n)}}\ , 
	\end{align*}
dropping the weight indices for simplicity; we will not need them for the rest of the proof. 
Since $\calC$ is a dg coproperad coming from a  comonadic dg coproperad, we have 
\[(f \rhd \alpha) d_\calC = f \rhd (\alpha d_\calC)  - \big(f;  f d_\calC\big)\rhd \alpha
\quad \text{and} \quad 
(\beta \lhd f) d_\calC = (\beta d_\calC) \lhd f  -   \beta \lhd  \big(f;  f d_\calC\big)\ .
\]
Only the following terms remain
	\begin{align*}
		\partial^A_B\widetilde{f}_{(n)}  = &
\Big( (f;f \rhd \alpha) \rhd \alpha - (f;\beta \lhd f ) \rhd \alpha  +\beta \lhd (f;f \rhd \alpha)  
		-\beta \lhd (f;\beta \lhd f)  \\ & 
   -f\rhd (\alpha\star  \alpha)  + (\beta\star  \beta) \lhd f \Big)|_{\calC_{(n)}}
		 \ .
	\end{align*}
We claim that the following identities hold
\begin{equation}\label{eqn:Id1}
\tag{a} (f;f \rhd \alpha) \rhd \alpha =   f \rhd ( \alpha \star  \alpha)\ ; 
\end{equation}
\begin{equation}\label{eqn:Id2}
\tag{b} \beta \lhd (f;\beta \lhd f)=(\beta\star  \beta) \lhd f\ ;
\end{equation}
\begin{equation}\label{eqn:Id3}
\tag{c} (f;\beta \lhd f ) \rhd \alpha = \beta \lhd (f;f \rhd \alpha)\ ,
\end{equation}

Equation~\eqref{eqn:Id1} follows from the same argument as in the proof of \cref{prop:ComCopropCoprop}. One is applying here twice the comonadic decomposition map in order to produce all the graphs with two top vertices labeled by $\alpha$ and a bottom level of vertices labeled by $f$. When the two top vertices do not sit one above the other and can be vertically switched, the corresponding terms cancel due to the sign convention since the degree of $\alpha$ is equal to $-1$. Only remain the terms where the two vertices do sit one above the other, which is obtained by $ f \rhd ( \alpha \star  \alpha)$. Equation~\eqref{eqn:Id2} is the vertical symmetry of Equation~\eqref{eqn:Id1}; so it holds as well. Equation~\eqref{eqn:Id3} is once again proved using a similar argument than the one of  \cref{prop:ComCopropCoprop}. Composing the comonadic decomposition map one way or another produces the same kind of graphs: a middle level of vertices labeled by $f$, a top vertex labeled by $\alpha$, and a bottom vertex labeled by $\beta$. 
\end{proof}

\subsection{Application}

\begin{proposition}
Let $\calC$ be a weight graded conilpotent dg coproperad. 
Let $(A, d_A, \alpha)$ be a $\Cobar \calC$-gebra structure and let $(B, d_B, 0)$ be an an acyclic chain complex equipped with trivial $\Cobar \calC$-gebra structure. Any map of chain complexes $A\to B$ extends to an $\infty$-morphism $(A, d_A, \alpha)\rightsquigarrow (B, d_B, 0)$. 
\end{proposition}

\begin{proof}
Since $(B, d_B)$ is acyclic, so is the chain complex $\left( \Hom_{\bbS}\big(\calC,\End^A_B\big), \big(\partial^A_B\big)_* \right)$\ . We can then can apply iteratively \cref{thm:Obstr} starting from $f_{(0)}$ equal to the original chain map. 
\end{proof}

\begin{remark}
We will use \cref{thm:Obstr} in a following paper dealing with a model category structure on dg monoid $\S\calC$-comodules to the one on dg $\calC$-coalgebras given in \cite{Vallette14}.
\end{remark}

\section{Examples}

\subsection{Algebras and coalgebras over an operad}
Recall that a biderivation $\S \calC \sq \S A\to \S \calC \sq \S A$ of a bifree monoid $\S \calC$-comodule generated by a graded vector space $A$ is equivalent to a map $\calC \boxtimes A\to  \S A$, by \cref{lem:BiDerHom}. 
When $\calC$ is a (coaugmented) cooperad, that is a coproperad concentrated in arities $(1,n)$ for $n\in \NN$, such a data amounts to a map $\calC(A)\cong \calC \boxtimes A \to  A$, that is to a coderivation of the cofree $\calC$-coalgebra $\calC(A)$. In this operadic case, the notion of bidifferential of the monoid $\S \calC$-comodule $\S \calC \sq \S A$ 
coincides with the notion of codifferential of the cofree $\calC$-coalgebra $\calC(A)$:
\[\Bidiff(\S \calC \sq \S A)\cong \Codiff(\calC(A))\ .\]

This way the properadic Rosetta stone \cref{thm:CompleteRosetta} applied to cooperads recovers exactly the operadic Rosetta stone \cite[Theorem~10.1.13]{LodayVallette12}. As a direct consequence, the present notion of $\infty$-morphism of $\Omega \calC$-(al)gebras given in \cref{subsec:InftyMorph} agrees with the operadic one, see  \cite[Section~10.2]{LodayVallette12}. 

\medskip

The operadic part, i.e. the part of arities $(1,n)$ for $n\in \NN$, of the Homotopy Transfer Theorem as established in \cref{sec:HTT} also recovers the operadic formulas, see \cite[Section~10.3]{LodayVallette12} and references therein. 
Applying the maps $\PHI$ or $\HHI$ to a tree with vertices labeled by elements of $s\End_A$ actually produces the same underlying tree with internal edges labeled by the contracting homotopy $h$, see \cite[Proof of Lemma~6]{DSV16}: nearly all the terms coming from the labelling of levelled trees by levels of homotopies cancel. So the present properadic approach (\cref{thm:HTT}) applied to a cooperad $\calC$ produces the exact same formulas for the homotopy transferred structure and the extension $i_\infty$ into an $\infty$-morphism given in \cite[Theorem~10.3.1]{LodayVallette12}. 
Applying the map $\PHH$ produces levelled trees and thus the formula for the extension $p_\infty$ into an $\infty$-morphism is precisely the one given in \cite[Proposition~10.3.9]{LodayVallette12}. 

\medskip

Encoding categories of coalgebras by a ``reversed'' operad, that is a properad concentrated in arities $(n, 1)$ for $n\in \NN$, one automatically gets the exact same homotopical properties for homotopy coalgebras, see for instance \cite{LGL18, CPRNW19} for seminal examples of applications. 

\subsection{The genus 0 case}
Let $\calC$ be a (coaugmented) codioperad, that is a genus 0 coproperad: the image of the structure map 
$\Delta_{\oC} : \oC \to \G^c\big(\oC\big)$ lands in the summand made up of genus 0 graphs only. In this case, the same collapsing phenomenon as above in the operadic case holds true since the argument of \cite[Proof of Lemma~6]{DSV16} applies as well: the homotopy transferred structure 
of \cref{thm:HTT} is actually given by the genus 0 graphs produced by $\Delta_{\oC}$ with internal edges labeled by the contracting homotopy $h$. Notice however that the many cancelations that appear here do not take place anymore when one applies the maps $\HHI$ or $\PHH$; thus the 
 extensions $i_\infty$ and  $p_\infty$ into a $\infty$-morphisms are still made up of genus 0 levelled graphs. 
 
 \medskip

This case covers the example of homotopy Lie bialgebras which are the gebras encodes by the properad 
$\Omega \big(\calS^c \Frob_\diamond^*\big)$, where $\calS^c$ stands for  the suspension coproperad  and where $\Frob_{\diamond}$ stands for the properad encoding Frobenius bialgebras satisfying the condition $\mu \circ \Delta=0$. We refer the reader to \cite{Vallette07} and to the next section for more details. 

\subsection{The Koszul case}\label{subsec:KD}
Let $\calP(E,R)=\G(E)/(R)$ be a quadratic properad, see \cite[Section~2.9]{Vallette07}. Its Koszul dual cooperad $\calP^{\ac}\coloneqq \calC(sE, s^{2}R)$ is a conilpotent coproperad which provides us with a canonical morphism of dg properads $\Omega \calP^{\ac} \to \calP$. When this latter one is a quasi-isomorphism, the properad $\calP$ is called \emph{Koszul}, see \cite[Section~7]{Vallette07}. In this case, the cobar construction $\calP_\infty\coloneqq\Omega \calP^{\ac}$ is the minimal model of $\calP$ and the homotopy theory of $\calP_{\infty}$-gebras is equivalent to the homotopy theory of $\calP$-gebras. 

\medskip

The present theory of $\infty$-morphisms and the homotopy transfer theorem applies to all $\calP_\infty$-gebras (actually  to all $\Omega \calP^{\ac}$-gebras, the properad $\calP$ being Koszul or not). The most difficult part lies now in the description of the Koszul dual coproperad $\calP^{\ac}$. 

\begin{definition}[Suspension properad]
The \emph{suspension properad} $\calS$ is the quadratic properad defined by 
two skew-symmetric binary operations of degree $1$
\[
	\vcenter{\hbox{\begin{tikzpicture}[baseline=1.8ex,scale=0.15]
	\draw (0,4) -- (2,2);
	\draw (4,4) -- (2,2);
	\draw (2,2) -- (2,0);
	\draw (0,4) node[above] {$\scriptscriptstyle{1}$};
	\draw (4,4) node[above] {$\scriptscriptstyle{2}$};
	\draw (2,0) node[below] {$\scriptscriptstyle{1}$};
	\draw[fill=white] (2,2) circle (8pt);
	\end{tikzpicture}}}
	= -
	\vcenter{\hbox{\begin{tikzpicture}[baseline=1.8ex,scale=0.15]
	\draw (0,4) -- (2,2);
	\draw (4,4) -- (2,2);
	\draw (2,2) -- (2,0);
	\draw (0,4) node[above] {$\scriptscriptstyle{2}$};
	\draw (4,4) node[above] {$\scriptscriptstyle{1}$};
	\draw (2,0) node[below] {$\scriptscriptstyle{1}$};
	\draw[fill=white] (2,2) circle (8pt);
	\end{tikzpicture}}}
	\qquad \mbox{and} \qquad
	\vcenter{\hbox{\begin{tikzpicture}[baseline=1.8ex,scale=0.15]
	\draw (2,4) -- (2,2);
	\draw (4,0) -- (2,2);
	\draw (2,2) -- (0,0);
	\draw[fill=white] (2,2) circle (8pt);
	\draw (2,4) node[above] {$\scriptscriptstyle{1}$};
	\draw (4,0) node[below] {$\scriptscriptstyle{2}$};
	\draw (0,0) node[below] {$\scriptscriptstyle{1}$};
	\end{tikzpicture}}}
	= - 
	\vcenter{\hbox{\begin{tikzpicture}[baseline=1.8ex,scale=0.15]
	\draw (2,4) -- (2,2);
	\draw (4,0) -- (2,2);
	\draw (2,2) -- (0,0);
	\draw (2,4) node[above] {$\scriptscriptstyle{1}$};
	\draw (4,0) node[below] {$\scriptscriptstyle{1}$};
	\draw (0,0) node[below] {$\scriptscriptstyle{2}$};
	\draw[fill=white] (2,2) circle (8pt);
	\end{tikzpicture}}}
	\]
which satisfy the following relations
		\begin{equation*}
	\vcenter{\hbox{\begin{tikzpicture}[baseline=0.5ex,scale=0.15]
	\draw (0,4) --(4,0);
	\draw (4,4) -- (2,2);
	\draw (4,0) -- (4,-2);
	\draw (8,4) -- (4,0);
	\draw (0,4) node[above] {$\scriptscriptstyle{1}$};
	\draw (4,4) node[above] {$\scriptscriptstyle{2}$};
	\draw (8,4) node[above] {$\scriptscriptstyle{3}$};
	\draw (4,-2) node[below] {$\scriptscriptstyle{1}$};
	\draw[fill=white] (2,2) circle (8pt);
	\draw[fill=white] (4,0) circle (8pt);
	\end{tikzpicture}}}
	  = - 
	\vcenter{\hbox{\begin{tikzpicture}[baseline=0.5ex,scale=0.15]
	\draw (0,4) --(4,0);
	\draw (4,4) -- (6,2);
	\draw (4,0) -- (4,-2);
	\draw (8,4) -- (4,0);
	\draw (0,4) node[above] {$\scriptscriptstyle{1}$};
	\draw (4,4) node[above] {$\scriptscriptstyle{2}$};
	\draw (8,4) node[above] {$\scriptscriptstyle{3}$};
	\draw (4,-2) node[below] {$\scriptscriptstyle{1}$};
	\draw[fill=white] (6,2) circle (8pt);
	\draw[fill=white] (4,0) circle (8pt);
	\end{tikzpicture}}} 
	\ , \qquad
	\vcenter{\hbox{\begin{tikzpicture}[baseline=0.5ex,scale=0.15]
	\draw (0,0) --(4,4);
	\draw (4,0) -- (2,2);
	\draw (4,4) -- (4,6);
	\draw (8,0) -- (4,4);
	\draw (0,0) node[below] {$\scriptscriptstyle{1}$};
	\draw (4,0) node[below] {$\scriptscriptstyle{2}$};
	\draw (8,0) node[below] {$\scriptscriptstyle{3}$};
	\draw (4,6) node[above] {$\scriptscriptstyle{1}$};
	\draw[fill=white] (2,2) circle (8pt);
	\draw[fill=white] (4,4) circle (8pt);
	\end{tikzpicture}}}
	 = - 
	\vcenter{\hbox{\begin{tikzpicture}[baseline=0.5ex,scale=0.15]
	\draw (0,0) --(4,4);
	\draw (4,0) -- (6,2);
	\draw (4,4) -- (4,6);
	\draw (8,0) -- (4,4);
	\draw (0,0) node[below] {$\scriptscriptstyle{1}$};
	\draw (4,0) node[below] {$\scriptscriptstyle{2}$};
	\draw (8,0) node[below] {$\scriptscriptstyle{3}$};
	\draw (4,6) node[above] {$\scriptscriptstyle{1}$};
	\draw[fill=white] (6,2) circle (8pt);
	\draw[fill=white] (4,4) circle (8pt);
	\end{tikzpicture}}}
	\ , \quad \text{and} \quad
	\vcenter{\hbox{\begin{tikzpicture}[baseline=2ex,scale=0.15]
	\draw (0,0) --(2,2);
	\draw (4,0) -- (2,2);
	\draw (2,2) -- (2,4);
	\draw (4,6) -- (2,4);
	\draw (0,6) -- (2,4);
	\draw (0,0) node[below] {$\scriptscriptstyle{1}$};
	\draw (4,0) node[below] {$\scriptscriptstyle{2}$};
	\draw (0,6) node[above] {$\scriptscriptstyle{1}$};
	\draw (4,6) node[above] {$\scriptscriptstyle{2}$};
	\draw[fill=white] (2,2) circle (8pt);
	\draw[fill=white] (2,4) circle (8pt);
	\end{tikzpicture}}}
	 = - 
	\vcenter{\hbox{\begin{tikzpicture}[baseline=0.5ex,scale=0.15]
	\draw (0,2) -- (2,0);
	\draw (4,2) -- (2,0);
	\draw (4,2) -- (6,0);
	\draw (4,4) -- (4,2);
	\draw (2,0) -- (2,-2);
	\draw[fill=white] (2,0) circle (8pt);
	\draw[fill=white] (4,2) circle (8pt);
	\draw (0,2) node[above] {$\scriptscriptstyle{1}$};
	\draw (4,4) node[above] {$\scriptscriptstyle{2}$};
	\draw (2,-2) node[below] {$\scriptscriptstyle{1}$};
	\draw (6,0) node[below] {$\scriptscriptstyle{2}$};
	\end{tikzpicture}}}\ .
	\end{equation*}
\end{definition}

\begin{lemma}\label{lem:SusProp}
The suspension properad is spanned by skew-symmetric corollas with genus, i.e. its component of arity $(m,n)$, for $m,n \geqslant 1$,  is isomorphic to 
$$\calS(m,n)\cong \bigoplus_{g\geqslant 0} \sgn_{\Sy_m}\otimes  \k s^{n+m+2g-2}  \otimes \sgn_{\Sy_n^{\mathrm{op}}}\ ,$$
where $\sgn$ stands for the signature representation. 
Its properadic composition maps 
amounts to the usual isomorphisms between tensors of suspensions.
\end{lemma}

\begin{proof}
The proof is straightforward: under the above mentioned relations, any binary graph of genus $g$ is equal to a right comb composed with a sequence of $g$ simple ``diamonds'' and then a left reversed comb:
\[
\vcenter{\hbox{\begin{tikzpicture}[scale=0.3]
	\draw (3,1) node[above] {$\scriptscriptstyle{1}$} --(4,0);
	\draw (4,0) -- (6,2);
	\draw (5,1) -- (4,2);
	\draw (4,2) node[above] {$\scriptscriptstyle{2}$};
	\draw[dotted] (5,1) -- (7,3);
	\draw (7,3) -- (9,5) node[above]{$\scriptscriptstyle{n}$} ;
	\draw (8,4) -- (7,5);
	\draw (4,0) -- (4,-1);
	\draw  (4,-1) -- (3,-2) -- (4,-3) -- (5,-2) -- (4,-1);
	\draw  (4,-3) -- (4, -4);
	\draw[dotted] (4,-4) -- (4,-5);
	\draw  (4,-5) -- (4, -6);
	\draw  (4,-6) -- (3,-7) -- (4,-8) -- (5,-7) -- (4,-6);
	\draw (2,-11)-- (4,-9) -- (5,-10);
	\draw (3,-10)--(4,-11);
	\draw (5,-10) node[below]{$\scriptscriptstyle{m}$}; 
	\draw[dotted] (2,-11) -- (1,-12);
	\draw[decorate,decoration={brace,raise=0.1cm}] (6,-1) -- (6,-8) node[above=1.05cm,right=0.2cm] {$\scriptscriptstyle{g}$};
	\draw (4,-8) -- (4,-9);
	\draw (1,-12) -- (-1,-14);
	\draw (0,-13) -- (1,-14);
	\draw (-1,-14) node[below]{$\scriptscriptstyle{1}$}; 
	\draw (1,-14) node[below]{$\scriptscriptstyle{2}$}; 	
	\draw[fill=white] (4,-8) circle (5pt) ;
	\draw[fill=white] (3,-10) circle (5pt) ;
	\draw[fill=white](4,-9)  circle (5pt);
	\draw[fill=white] (0,-13) circle (5pt);
	\draw[fill=white] (4,-6) circle (5pt);
	\draw[fill=white] (4,0) circle (5pt);
	\draw[fill=white] (8,4) circle (5pt);
	\draw[fill=white] (5,1) circle (5pt)  ;
	\draw[fill=white] (4,-1) circle (5pt);
	\draw[fill=white] (4,-3) circle (5pt);
	\end{tikzpicture}}}
\]
\end{proof}

We denote by $\oH$ the arity-wise tensor product of $\Sy$-bimodules, also known as the Hadamard tensor product. 
Recall that the free properad $\G(E)$ admits three gradings: the number $n$ of inputs, the number $m$ of outputs,  and the genus $g$ of the underlying graph. Notice that the space $R$ of relations of a binary quadratic properad $\calP(E, R)$ is homogenous with respect to these three gradings. This induces two arity gradings and a genus grading on the binary quadratic properad $\calP(E, R)$. We consider  the arity-wise  and genus wise tensor product of  binary quadratic properads, that we denote by $\oG$.

\begin{definition}[Koszul dual properad]
The \emph{Koszul dual properad} of a binary quadratic properad $\calP(E,R)$   
is the properad 
\[
\calP^!\coloneqq  \calS \oG  \left(\calP^{\ac}\right)^* \ .
\]
\end{definition}

Notice that this is well a duality functor: $\left(\calP^!\right)^!\cong \calP$. 

\begin{proposition}\label{prop:KosDualProp}
The Koszul dual properad of a finitely generated binary quadratic properad $\calP(E,R)$
admits the following binary quadratic presentation:
\[  
\calP^!\cong \calP\big(s^{-1}\calS\oH E^*, R^\perp\big)
\ . \]
\end{proposition}

\begin{proof} The argument are \emph{mutatis mutandis} similar to the proof of 
\cite[Proposition~7.2.1]{LodayVallette12}. One first notices that the Koszul dual properad is actually by the Manin white product of the two properads $\calS$ and  $\left(\calP^{\ac}\right)^*$, see \cite{Vallette08}. 
Then, from the definition of Manin white product, one gets the required quadratic presentation: 
\[\calP^!=  \calS \oG  \left(\calP^{\ac}\right)^*\cong \calS \bigcirc  \left(\calP^{\ac}\right)^*
\cong \calP\big(s^{-1}\calS\oH E^*, R^\perp\big) \ . \]
\end{proof}

The canonical object of the Koszul duality theory is the Koszul dual coproperad $\calP^{\ac}$; the Koszul dual properad $\calP^!$ is not essential. It is just introduced as a practical way to compute the Koszul dual coproperad: one first computes the Koszul dual properad using the presentation given in \cref{prop:KosDualProp} and then one uses 
\[\calP^{\ac}\cong \calS^c\oG \big(\calP^!\big)^*\ ,\] 
where $\calS^c$ stands for the suspension coproperad defined on the same underlying $\Sy$-bimodule as the properad $\calS$ with decomposition maps dual to the composition maps. For any finitely generated binary quadratic properad $\calP$, we use the following simple notation the coproperadic suspension
\[  
\calS^c \calP^*\coloneqq \calS^c\oG  \calP^*
\ . \]

\subsection{Example: Involutive Lie bialgebras up to homotopy}\label{subsec:IBL}

In \cite{CFL11}, Cieliebak--Fukaya--Lat\-sch\-ev developed a notion of involutive Lie bialgebra up to homotopy together with $\infty$-morphisms between them, and related homotopy properties. 
We also refer the reader to \cite{Hajek18} for more details. 
We recover the notions and results of \cite{CFL11} as a particular case of our general theory.
However the present explicit formula for the homotopy transfer theorem is new. 
Such results plays a seminal role in symplectic field theory, 
string topology, cyclic homology, and Lagrangian Floer theory 
 \cite{CFL11} and are deeply connected with ribbons graphs and moduli spaces of curves \cite{MW15}.

\medskip

Recall that an \emph{involutive Lie bialgebra} $(A, [\, , ], \delta)$ is a Lie bialgebra satisfying the extra relation 
$[\, , ]\circ \delta=0$.

\begin{definition}[Properad $\IBL$]
	The properad $\IBL$ of involutive Lie bialgebras is generated by two skew-symmetric operations of degree $0$
	\[
	\vcenter{\hbox{\begin{tikzpicture}[baseline=1.8ex,scale=0.15]
	\draw (0,4) -- (2,2);
	\draw (4,4) -- (2,2);
	\draw[fill=black] (2,2) circle (6pt);
	\draw (2,2) -- (2,0);
	\draw (0,4) node[above] {$\scriptscriptstyle{1}$};
	\draw (4,4) node[above] {$\scriptscriptstyle{2}$};
	\draw (2,0) node[below] {$\scriptscriptstyle{1}$};
	\end{tikzpicture}}}
	= -
	\vcenter{\hbox{\begin{tikzpicture}[baseline=1.8ex,scale=0.15]
	\draw (0,4) -- (2,2);
	\draw (4,4) -- (2,2);
	\draw[fill=black] (2,2) circle (6pt);
	\draw (2,2) -- (2,0);
	\draw (0,4) node[above] {$\scriptscriptstyle{2}$};
	\draw (4,4) node[above] {$\scriptscriptstyle{1}$};
	\draw (2,0) node[below] {$\scriptscriptstyle{1}$};
	\end{tikzpicture}}}
	\qquad \mbox{and} \qquad
	\vcenter{\hbox{\begin{tikzpicture}[baseline=1.8ex,scale=0.15]
	\draw (2,4) -- (2,2);
	\draw (4,0) -- (2,2);
	\draw (2,2) -- (0,0);
	\draw[fill=black] (2,2) circle (6pt);
	\draw (2,4) node[above] {$\scriptscriptstyle{1}$};
	\draw (4,0) node[below] {$\scriptscriptstyle{2}$};
	\draw (0,0) node[below] {$\scriptscriptstyle{1}$};
	\end{tikzpicture}}}
	= - 
	\vcenter{\hbox{\begin{tikzpicture}[baseline=1.8ex,scale=0.15]
	\draw (2,4) -- (2,2);
	\draw (4,0) -- (2,2);
	\draw (2,2) -- (0,0);
	\draw[fill=black] (2,2) circle (6pt);
	\draw (2,4) node[above] {$\scriptscriptstyle{1}$};
	\draw (4,0) node[below] {$\scriptscriptstyle{1}$};
	\draw (0,0) node[below] {$\scriptscriptstyle{2}$};
	\end{tikzpicture}}}
	\]
which satisfy the following relations
	\begin{align*}
&	\vcenter{\hbox{\begin{tikzpicture}[baseline=0.5ex,scale=0.15]
	\draw (0,4) --(4,0);
	\draw (4,4) -- (2,2);
	\draw (4,0) -- (4,-2);
	\draw (8,4) -- (4,0);
	\draw (0,4) node[above] {$\scriptscriptstyle{1}$};
	\draw (4,4) node[above] {$\scriptscriptstyle{2}$};
	\draw (8,4) node[above] {$\scriptscriptstyle{3}$};
	\draw (4,-2) node[below] {$\scriptscriptstyle{1}$};
	\draw[fill=black] (2,2) circle (6pt);
	\draw[fill=black] (4,0) circle (6pt);
	\end{tikzpicture}}}
	  +  
	\vcenter{\hbox{\begin{tikzpicture}[baseline=0.5ex,scale=0.15]
	\draw (0,4) --(4,0);
	\draw (4,4) -- (2,2);
	\draw (4,0) -- (4,-2);
	\draw (8,4) -- (4,0);
	\draw (0,4) node[above] {$\scriptscriptstyle{2}$};
	\draw (4,4) node[above] {$\scriptscriptstyle{3}$};
	\draw (8,4) node[above] {$\scriptscriptstyle{1}$};
	\draw (4,-2) node[below] {$\scriptscriptstyle{1}$};
	\draw[fill=black] (2,2) circle (6pt);
	\draw[fill=black] (4,0) circle (6pt);
	\end{tikzpicture}}}
	  +  
	\vcenter{\hbox{\begin{tikzpicture}[baseline=0.5ex,scale=0.15]
	\draw (0,4) --(4,0);
	\draw (4,4) -- (2,2);
	\draw (4,0) -- (4,-2);
	\draw (8,4) -- (4,0);
	\draw (0,4) node[above] {$\scriptscriptstyle{3}$};
	\draw (4,4) node[above] {$\scriptscriptstyle{1}$};
	\draw (8,4) node[above] {$\scriptscriptstyle{2}$};
	\draw (4,-2) node[below] {$\scriptscriptstyle{1}$};
	\draw[fill=black] (2,2) circle (6pt);
	\draw[fill=black] (4,0) circle (6pt);
	\end{tikzpicture}}}
	  =   0 
	 \ , \qquad
	\vcenter{\hbox{\begin{tikzpicture}[baseline=0.5ex,scale=0.15]
	\draw (0,0) --(4,4);
	\draw (4,0) -- (2,2);
	\draw (4,4) -- (4,6);
	\draw (8,0) -- (4,4);
	\draw (0,0) node[below] {$\scriptscriptstyle{1}$};
	\draw (4,0) node[below] {$\scriptscriptstyle{2}$};
	\draw (8,0) node[below] {$\scriptscriptstyle{3}$};
	\draw (4,6) node[above] {$\scriptscriptstyle{1}$};
	\draw[fill=black] (2,2) circle (6pt);
	\draw[fill=black] (4,4) circle (6pt);
	\end{tikzpicture}}}
	  +   
	\vcenter{\hbox{\begin{tikzpicture}[baseline=0.5ex,scale=0.15]
	\draw (0,0) --(4,4);
	\draw (4,0) -- (2,2);
	\draw (4,4) -- (4,6);
	\draw (8,0) -- (4,4);
	\draw (0,0) node[below] {$\scriptscriptstyle{2}$};
	\draw (4,0) node[below] {$\scriptscriptstyle{3}$};
	\draw (8,0) node[below] {$\scriptscriptstyle{1}$};
	\draw (4,6) node[above] {$\scriptscriptstyle{1}$};
	\draw[fill=black] (2,2) circle (6pt);
	\draw[fill=black] (4,4) circle (6pt);
	\end{tikzpicture}}}
	  +   
	\vcenter{\hbox{\begin{tikzpicture}[baseline=0.5ex,scale=0.15]
	\draw (0,0) --(4,4);
	\draw (4,0) -- (2,2);
	\draw (4,4) -- (4,6);
	\draw (8,0) -- (4,4);
	\draw (0,0) node[below] {$\scriptscriptstyle{3}$};
	\draw (4,0) node[below] {$\scriptscriptstyle{1}$};
	\draw (8,0) node[below] {$\scriptscriptstyle{2}$};
	\draw (4,6) node[above] {$\scriptscriptstyle{1}$};
	\draw[fill=black] (2,2) circle (6pt);
	\draw[fill=black] (4,4) circle (6pt);
	\end{tikzpicture}}}
	  =   0  \ ,
\\	
&	\vcenter{\hbox{\begin{tikzpicture}[baseline=2ex,scale=0.15]
	\draw (0,0) --(2,2);
	\draw (4,0) -- (2,2);
	\draw (2,2) -- (2,4);
	\draw (4,6) -- (2,4);
	\draw (0,6) -- (2,4);
	\draw (0,0) node[below] {$\scriptscriptstyle{1}$};
	\draw (4,0) node[below] {$\scriptscriptstyle{2}$};
	\draw (0,6) node[above] {$\scriptscriptstyle{1}$};
	\draw (4,6) node[above] {$\scriptscriptstyle{2}$};
	\draw[fill=black] (2,2) circle (6pt);
	\draw[fill=black] (2,4) circle (6pt);
	\end{tikzpicture}}}
	  -  
	\vcenter{\hbox{\begin{tikzpicture}[baseline=0.5ex,scale=0.15]
	\draw (0,2) -- (2,0);
	\draw (4,2) -- (2,0);
	\draw (0,2) -- (-2,0);
	\draw (0,4) -- (0,2);
	\draw (2,0) -- (2,-2);
	\draw[fill=black] (2,0) circle (6pt);
	\draw[fill=black] (0,2) circle (6pt);
	\draw (0,4) node[above] {$\scriptscriptstyle{1}$};
	\draw (4,2) node[above] {$\scriptscriptstyle{2}$};
	\draw (-2,0) node[below] {$\scriptscriptstyle{1}$};
	\draw (2,-2) node[below] {$\scriptscriptstyle{2}$};
	\end{tikzpicture}}}
	  +  
	\vcenter{\hbox{\begin{tikzpicture}[baseline=0.5ex,scale=0.15]
	\draw (0,2) -- (2,0);
	\draw (4,2) -- (2,0);
	\draw (0,2) -- (-2,0);
	\draw (0,4) -- (0,2);
	\draw (2,0) -- (2,-2);
	\draw[fill=black] (2,0) circle (6pt);
	\draw[fill=black] (0,2) circle (6pt);
	\draw (0,4) node[above] {$\scriptscriptstyle{2}$};
	\draw (4,2) node[above] {$\scriptscriptstyle{1}$};
	\draw (-2,0) node[below] {$\scriptscriptstyle{1}$};
	\draw (2,-2) node[below] {$\scriptscriptstyle{2}$};
	\end{tikzpicture}}}
	  -  
	\vcenter{\hbox{\begin{tikzpicture}[baseline=0.5ex,scale=0.15]
	\draw (0,2) -- (2,0);
	\draw (4,2) -- (2,0);
	\draw (4,2) -- (6,0);
	\draw (4,4) -- (4,2);
	\draw (2,0) -- (2,-2);
	\draw[fill=black] (2,0) circle (6pt);
	\draw[fill=black] (4,2) circle (6pt);
	\draw (0,2) node[above] {$\scriptscriptstyle{1}$};
	\draw (4,4) node[above] {$\scriptscriptstyle{2}$};
	\draw (2,-2) node[below] {$\scriptscriptstyle{1}$};
	\draw (6,0) node[below] {$\scriptscriptstyle{2}$};
	\end{tikzpicture}}}
	  +  
	\vcenter{\hbox{\begin{tikzpicture}[baseline=0.5ex,scale=0.15]
	\draw (0,2) -- (2,0);
	\draw (4,2) -- (2,0);
	\draw (4,2) -- (6,0);
	\draw (4,4) -- (4,2);
	\draw (2,0) -- (2,-2);
	\draw[fill=black] (2,0) circle (6pt);
	\draw[fill=black] (4,2) circle (6pt);
	\draw (0,2) node[above] {$\scriptscriptstyle{2}$};
	\draw (4,4) node[above] {$\scriptscriptstyle{1}$};
	\draw (2,-2) node[below] {$\scriptscriptstyle{1}$};
	\draw (6,0) node[below] {$\scriptscriptstyle{2}$};
	\end{tikzpicture}}}
	  =   0 \ , \quad \text{and} \quad 
	\vcenter{\hbox{\begin{tikzpicture}[baseline=0.9ex,scale=0.15]
	\draw (2,6) -- (2,4) -- (0,2) -- (2,0) -- (2,-2);
	\draw (2,4) -- (4,2) -- (2,0);
	\draw[fill=black] (2,4) circle (6pt);
	\draw[fill=black] (2,0) circle (6pt);
	\draw (2,6) node[above] {$\scriptscriptstyle{1}$};
	\draw (2,-2) node[below] {$\scriptscriptstyle{1}$};
	\end{tikzpicture}}}
	= 0 \ .
\end{align*}
	
Recall that a \emph{Frobenius bialgebra} $(A, \mu, \Delta)$ is a commutative algebra equipped with a cocommutative coproduct, which a morphism of modules. 
	
\end{definition}
\begin{definition}[Properad Frob]
	The 	properad $\Frob$ encoding Frobenius bialgebras is generated by two symmetric operations of degree $0$
	\[
	\vcenter{\hbox{\begin{tikzpicture}[baseline=1.8ex,scale=0.15]
	\draw (0,4) -- (2,2);
	\draw (4,4) -- (2,2);
	\draw (2,2) -- (2,0);
	\draw (0,4) node[above] {$\scriptscriptstyle{1}$};
	\draw (4,4) node[above] {$\scriptscriptstyle{2}$};
	\draw (2,0) node[below] {$\scriptscriptstyle{1}$};
	\draw[fill=white] (2,2) circle (8pt);
	\end{tikzpicture}}}
	= 
	\vcenter{\hbox{\begin{tikzpicture}[baseline=1.8ex,scale=0.15]
	\draw (0,4) -- (2,2);
	\draw (4,4) -- (2,2);
	\draw (2,2) -- (2,0);
	\draw (0,4) node[above] {$\scriptscriptstyle{2}$};
	\draw (4,4) node[above] {$\scriptscriptstyle{1}$};
	\draw (2,0) node[below] {$\scriptscriptstyle{1}$};
	\draw[fill=white] (2,2) circle (8pt);
	\end{tikzpicture}}}
	\qquad \mbox{and} \qquad
	\vcenter{\hbox{\begin{tikzpicture}[baseline=1.8ex,scale=0.15]
	\draw (2,4) -- (2,2);
	\draw (4,0) -- (2,2);
	\draw (2,2) -- (0,0);
	\draw[fill=white] (2,2) circle (8pt);
	\draw (2,4) node[above] {$\scriptscriptstyle{1}$};
	\draw (4,0) node[below] {$\scriptscriptstyle{2}$};
	\draw (0,0) node[below] {$\scriptscriptstyle{1}$};
	\end{tikzpicture}}}
	= 
	\vcenter{\hbox{\begin{tikzpicture}[baseline=1.8ex,scale=0.15]
	\draw (2,4) -- (2,2);
	\draw (4,0) -- (2,2);
	\draw (2,2) -- (0,0);
	\draw[fill=white] (2,2) circle (8pt);
	\draw (2,4) node[above] {$\scriptscriptstyle{1}$};
	\draw (4,0) node[below] {$\scriptscriptstyle{1}$};
	\draw (0,0) node[below] {$\scriptscriptstyle{2}$};
	\end{tikzpicture}}}
	\]
which satisfy the following relations
	\begin{equation*}
	\vcenter{\hbox{\begin{tikzpicture}[baseline=0.5ex,scale=0.15]
	\draw (0,4) --(4,0);
	\draw (4,4) -- (2,2);
	\draw (4,0) -- (4,-2);
	\draw (8,4) -- (4,0);
	\draw (0,4) node[above] {$\scriptscriptstyle{1}$};
	\draw (4,4) node[above] {$\scriptscriptstyle{2}$};
	\draw (8,4) node[above] {$\scriptscriptstyle{3}$};
	\draw (4,-2) node[below] {$\scriptscriptstyle{1}$};
	\draw[fill=white] (2,2) circle (8pt);
	\draw[fill=white] (4,0) circle (8pt);
	\end{tikzpicture}}}
	\quad  = \quad
	\vcenter{\hbox{\begin{tikzpicture}[baseline=0.5ex,scale=0.15]
	\draw (0,4) --(4,0);
	\draw (4,4) -- (6,2);
	\draw (4,0) -- (4,-2);
	\draw (8,4) -- (4,0);
	\draw (0,4) node[above] {$\scriptscriptstyle{1}$};
	\draw (4,4) node[above] {$\scriptscriptstyle{2}$};
	\draw (8,4) node[above] {$\scriptscriptstyle{3}$};
	\draw (4,-2) node[below] {$\scriptscriptstyle{1}$};
	\draw[fill=white] (6,2) circle (8pt);
	\draw[fill=white] (4,0) circle (8pt);
	\end{tikzpicture}}} 
	\ , \qquad
	\vcenter{\hbox{\begin{tikzpicture}[baseline=0.5ex,scale=0.15]
	\draw (0,0) --(4,4);
	\draw (4,0) -- (2,2);
	\draw (4,4) -- (4,6);
	\draw (8,0) -- (4,4);
	\draw (0,0) node[below] {$\scriptscriptstyle{1}$};
	\draw (4,0) node[below] {$\scriptscriptstyle{2}$};
	\draw (8,0) node[below] {$\scriptscriptstyle{3}$};
	\draw (4,6) node[above] {$\scriptscriptstyle{1}$};
	\draw[fill=white] (2,2) circle (8pt);
	\draw[fill=white] (4,4) circle (8pt);
	\end{tikzpicture}}}
	\quad = \quad
	\vcenter{\hbox{\begin{tikzpicture}[baseline=0.5ex,scale=0.15]
	\draw (0,0) --(4,4);
	\draw (4,0) -- (6,2);
	\draw (4,4) -- (4,6);
	\draw (8,0) -- (4,4);
	\draw (0,0) node[below] {$\scriptscriptstyle{1}$};
	\draw (4,0) node[below] {$\scriptscriptstyle{2}$};
	\draw (8,0) node[below] {$\scriptscriptstyle{3}$};
	\draw (4,6) node[above] {$\scriptscriptstyle{1}$};
	\draw[fill=white] (6,2) circle (8pt);
	\draw[fill=white] (4,4) circle (8pt);
	\end{tikzpicture}}}
	\ , \quad \text{and} \quad
	\vcenter{\hbox{\begin{tikzpicture}[baseline=2ex,scale=0.15]
	\draw (0,0) --(2,2);
	\draw (4,0) -- (2,2);
	\draw (2,2) -- (2,4);
	\draw (4,6) -- (2,4);
	\draw (0,6) -- (2,4);
	\draw (0,0) node[below] {$\scriptscriptstyle{1}$};
	\draw (4,0) node[below] {$\scriptscriptstyle{2}$};
	\draw (0,6) node[above] {$\scriptscriptstyle{1}$};
	\draw (4,6) node[above] {$\scriptscriptstyle{2}$};
	\draw[fill=white] (2,2) circle (8pt);
	\draw[fill=white] (2,4) circle (8pt);
	\end{tikzpicture}}}
	\quad = \quad
	\vcenter{\hbox{\begin{tikzpicture}[baseline=0.5ex,scale=0.15]
	\draw (0,2) -- (2,0);
	\draw (4,2) -- (2,0);
	\draw (4,2) -- (6,0);
	\draw (4,4) -- (4,2);
	\draw (2,0) -- (2,-2);
	\draw[fill=white] (2,0) circle (8pt);
	\draw[fill=white] (4,2) circle (8pt);
	\draw (0,2) node[above] {$\scriptscriptstyle{1}$};
	\draw (4,4) node[above] {$\scriptscriptstyle{2}$};
	\draw (2,-2) node[below] {$\scriptscriptstyle{1}$};
	\draw (6,0) node[below] {$\scriptscriptstyle{2}$};
	\end{tikzpicture}}}\ .
	\end{equation*}
\end{definition}

\begin{proposition}\label{prop:IBLac}\leavevmode

	\begin{enumerate}
		\item The two properads $\IBL$ and $\Frob$ are Koszul dual to one another:
\[
\IBL^!\cong \Frob \quad \text{and} \quad \Frob^!\cong \IBL\ .
\]		
		\item The Koszul dual coproperad  is isomorphic to the suspension coproperad :
		\[\IBL^{\ac}\cong \calS^c \Frob^*\cong \calS^c\ .\]
		It admits basis elements of the following form
		\[
		\vcenter
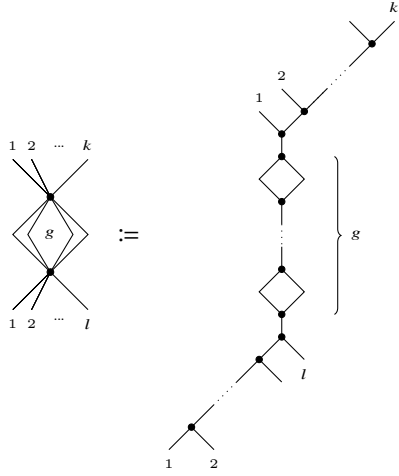
{\hbox{\begin{tikzpicture}[scale=0.25]
	\draw (2,6) node[above] {$\scriptscriptstyle{k}$};
	\draw (-2,6) node[above] {$\scriptscriptstyle{1}$};
	\draw (-1,6) node[above] {$\scriptscriptstyle{2}$};
	\draw (0.5,6) node[above] {$\scriptscriptstyle{\cdots}$};
	\draw (0,4) -- (-2,6) -- (0,4) -- (-1,6) -- (0,4) -- (2,6);
	\draw[fill=black] (0,4) circle (5pt);
	\draw (0,4) -- (2,2) -- (0,0) -- (-2,2) -- (0,4);
	\draw (0,4) -- (-1.2,2) -- (0,0);
	\draw (0,4) -- (1.2,2) -- (0,0);
	\draw (0,2) node {$\scriptscriptstyle{g}$};
	\draw (0,0) -- (-2,-2) -- (0,0) -- (-1,-2) -- (0,0) -- (2,-2);
	\draw[fill=black] (0,0) circle (5pt);
	\draw (2,-2) node[below] {$\scriptscriptstyle{l}$};
	\draw (-2,-2) node[below] {$\scriptscriptstyle{1}$};
	\draw (-1,-2) node[below] {$\scriptscriptstyle{2}$};
	\draw (0.5,-2) node[below] {$\scriptscriptstyle{\cdots}$};
\end{tikzpicture}}}
\ \coloneqq  \
	\vcenter{\hbox{\begin{tikzpicture}[scale=0.3]
	\draw (3,1) node[above] {$\scriptscriptstyle{1}$} --(4,0);
	\draw (4,0) -- (6,2);
	\draw (5,1) -- (4,2);
	\draw (4,2) node[above] {$\scriptscriptstyle{2}$};
	\draw[dotted] (5,1) -- (7,3);
	\draw (7,3) -- (9,5) node[above]{$\scriptscriptstyle{k}$} ;
	\draw[fill=black] (4,0) circle (4pt);
	\draw[fill=black] (8,4) circle (4pt);
	\draw[fill=black] (5,1) circle (4pt)  ;
	\draw (8,4) -- (7,5);
	\draw (4,0) -- (4,-1);
	\draw[fill=black] (4,-1) circle (4pt);
	\draw  (4,-1) -- (3,-2) -- (4,-3) -- (5,-2) -- (4,-1);
	\draw[fill=black] (4,-3) circle (4pt);
	\draw  (4,-3) -- (4, -4);
	\draw[dotted] (4,-4) -- (4,-5);
	\draw  (4,-5) -- (4, -6);
	\draw[fill=black] (4,-6) circle (4pt);
	\draw  (4,-6) -- (3,-7) -- (4,-8) -- (5,-7) -- (4,-6);
	\draw[fill=black] (4,-8) circle (4pt) ;
	\draw[fill=black] (3,-10) circle (4pt) ;
	\draw[fill=black](4,-9)  circle (4pt);
	\draw[fill=black] (0,-13) circle (4pt);
		\draw (2,-11)-- (4,-9) -- (5,-10);
	\draw (3,-10)--(4,-11);
	\draw (5,-10) node[below]{$\scriptscriptstyle{l}$}; 
	\draw[dotted] (2,-11) -- (1,-12);
	\draw[decorate,decoration={brace,raise=0.1cm}] (6,-1) -- (6,-8) node[above=1.05cm,right=0.2cm] {$\scriptscriptstyle{g}$};
	\draw (4,-8) -- (4,-9);
	\draw (1,-12) -- (-1,-14);
	\draw (0,-13) -- (1,-14);
	\draw (-1,-14) node[below]{$\scriptscriptstyle{1}$}; 
	\draw (1,-14) node[below]{$\scriptscriptstyle{2}$}; 	
	\end{tikzpicture}}}
		\]
		with $k, l\geqslant 1$ and $g\geqslant 0$. They are denoted by  $c_{k,l,g}$, where  $c_{1,1,0}$ is the identity, and their degree is equal to $|c_{k,l,g}|=k+l+2g-2$. 
		
		\item The infinitesimal coproduct $\Delta_{(1,1)}$ of $\IBL^{\ac}\cong \calS^c$ is given on basis elements, for $k+l+g\geqslant 3$, by:
		\[
			\Delta_{(1,1)}(c_{k,l,g}) = 
			\sum_{{\substack{
						r\geqslant 1 \\
						g'+g''+r-1=g \\ 
						k'+k''-r=k \\
						l'+l''-r=l\\
						k'+l'+g'\geqslant 3\\
						k''+l''+g''\geqslant 3
					}}} 
		\sum_{{\substack{\sigma\in Sh^{-1}_{k'-r,k''}\\
						\tau \in Sh_{l',l''-r}}}}
		\varepsilon\, \sgn(\sigma)\sgn(\tau)
		\, {}^\tau(c_{k',l',g'} \underset{r}{\circ} c_{k'',l'',g''})^\sigma\ ,
		\]
		with $\varepsilon=(-1)^{\frac{(r-1)(r-2)}{2}+(k'-r)(k''-r)+(l'-r)(l''-r)+(k'-r)(l''-r)}$, 
		where $\underset{r}{\circ}$ denotes 
		the composite of the $r$ last outputs of $c_{k'',l'',g''}$ along the $r$ first inputs of $c_{k',l',g'}$ and where $Sh_{a,b}$ denotes the set of $(a,b)$-shuffles.
		
		\item The image an element $c_{k,l,g}$ under the coproperadic decomposition map $\Delta$ is equal to the sum (including a sign) of 2-level connected graphs with vertices labeled with elements $c_{k',l',g'}$ such that the total number of inputs is equal to $k$, the total number of outputs is equal to $l$ and the total genus, including the one of the vertices ($g'$) is equal to $g$. 
		
	\end{enumerate}
\end{proposition}

\begin{proof}\leavevmode
\begin{enumerate}
\item This point is straightforward application of \cref{prop:KosDualProp}. 

\item One can see that the properad $\Frob$ admits a basis similar to that of the suspension properad $\calS$, by the same argument as in the proof of \cref{lem:SusProp}. 

\item Let us denote by $m_{k,l,g}$ the basis element of the suspension properad $\calS$ made up of  a right comb of arity $k$ composed with a sequence of $g$ simple ``diamonds'' and then a left reversed comb of arity $l$. 
Using the defining quadratic relations of $\calS$, 
one can see that the composite $\underset{r}{\circ}$ of the $r$ last outputs of $m_{k'',l'',g''}$ along the $r$ first inputs of 
$m_{k',l',g'}$ is equal to 
\[ 
m_{k',l',g'} \underset{r}{\circ} m_{k'',l'',g''} = (-1)^{\frac{(r-1)(r-2)}{2}+(k'-r)(k''-r)+(l'-r)(l''-r)+(k'-r)(l''-r)} \, m_{k, l, g}\ ,
\]
where $k=k'+k''-r$, $l=l'+l''-r$, and $g=g'+g''+r-1$. The infinitesimal coproduct of the coproperad $\calS^c$ is the categorical dual of this composite of the properad $\calS$; so it carries the same form and signs. 
\item This point is a straightforward consequence of the previous results. 
\end{enumerate}
\end{proof}

\begin{remark}
The first point can be interpreted as saying that symplectic field theory \cite{EGH00} is Koszul dual to (2-dimensional) topological quantum field theory \cite{Atiyah88}. 
\end{remark}

\begin{theorem}[{\cite{CamposMerkulovWillwacher14}}]
	The properads $\IBL$ and $\Frob$ are Koszul. 
\end{theorem}

This implies that
	\[
		\IBL_\infty\coloneqq \Omega\, \IBL^{\ac}\cong \Omega\,  \calS^c  \overset{\sim}{\twoheadrightarrow} \IBL 
	\]
	is the minimal model and a cofibrant resolution of the properad $\IBL$. 
	
\begin{proposition}[Homotopy involutive Lie bialgebra]
An $\IBL_\infty$-gebra structure  on a differential graded vector space $A$ is a collection of skew-symmetric operations 
\[\mu_{k,l,g} : A^{\wedge k} \to A^{\wedge l}\ ,\]
of degree $|\mu_{k,l,g}|=k+l+2g-3$, 
for  $k, l\geqslant 1$, $g\geqslant 0$, and $k+l+g\geqslant 3$, satisfying 
\[
			\partial(\mu_{k,l,g}) = 
			\sum_{{\substack{
						r\geqslant 1 \\
						g'+g''+r-1=g \\ 
						k'+k''-r=k \\
						l'+l''-r=l
					}}} 
		\sum_{{\substack{\sigma\in Sh^{-1}_{k'-r,k''}\\
						\tau \in Sh_{l',l''-r}}}}
\varepsilon\, \sgn(\sigma)\sgn(\tau)
		\, {}^\tau(c_{k',l',g'} \underset{r}{\circ} c_{k'',l'',g''})^\sigma\ , 
		\]
where $\varepsilon=(-1)^{\frac{(r-1)(r-2)}{2}+(k'-r)(k''-r)+(l'-r)(l''-r)+(k'-r)(l''-r)+k'+l'+2g'+1}$. 		
\end{proposition}

\begin{proof}
Recall that an $\IBL_\infty$-gebra structure on a differential graded vector space $A$ amounts to a twisting morphism 
$\alpha \in \Hom_{\Sy}(\IBL^{\ac}, \End_A)$. This result is thus a direct corollary of Points (2) and (3) of \cref{prop:IBLac}. 
\end{proof}

\begin{remark}\leavevmode
\begin{itemize}
\item[$\diamond$] Restricting to the genus 0 part, that is $g=0$, one gets the notion of a homotopy Lie bialgebra. 

\item[$\diamond$] Restricting to the operadic part, that is $l=1$ and $g=0$, on gets the notion of a homotopy Lie algebra, under the exact same form including the sign as in \cite[Proposition~10.1.7]{LodayVallette12}.
\end{itemize}
\end{remark}

The notion a homotopy  involutive Lie bialgebra given here, which actually comes from 
\cite{Vallette07}, agrees with the one given in \cite{CFL11}.
We recover its properties established in \emph{loc. cit.} as the particular case of our general theory applied to the coproperad $\calC=\IBL^{\ac}\cong \calS^c$, as summarized in the table below.

Let $F:A \rightsquigarrow B$ be an $\infty$-morphism: the condition \eqref{eq:Morph} stated in \Cref{prop:infmor} on the associated morphism $f:\calS^c  \to \End_B^A$ amounts to
\[
f \rhd \alpha - \beta \lhd f  =\partial^A_B \circ f \ ,
\]
since $d_{\calS^c}=0$.
We claim that this condition is equivalent to the one appearing in \cite[Definition~2.9]{CFL11} where the authors define the notion of an $\IBL_\infty$-morphism. Let us first give a dictionary between the two papers: the first part of the table explains the correspondence between the notions of $\IBL_\infty$-structures and the second part the correspondence between the notions of $\infty$-morphisms. 
\medskip

\begin{center}
	\begin{tabular}{|c|c|}
		\hline
		\rule{0pt}{10pt}{\sc Notation of the present paper } &  {\sc Notations of \cite{CFL11}}\\
		\hline
		\rule{0pt}{10pt} $A$ (underlying object)&
		 $C$\\
		\hline
		\rule{0pt}{10pt}$\S \calS^c \sq \S  A$&
		$E C$ \\
		\hline

		\rule{0pt}{10pt}$\mu_{k,l,g}=\alpha(c_{k,l,g})$ and $d_A$&
		$\frakp_{k,l,g}$ and $\frakp_{1,1,0}$ \\
		\hline
		\rule{0pt}{10pt} bidifferential $d_{\alpha}$ &
		$\hat{\frakp}:EC \to EC$\\
		\hline
		\rule{0pt}{10pt} $(d_\alpha)^2=0$ (see \Cref{prop:EndofRosetta}) & 
		$\hat{\frakp} \circ \hat{\frakp}=0$ (see \cite[Equation~(2.2)]{CFL11}) \\
		\hline
		\rule{0pt}{10pt} $(A,\alpha : \calS^c \to \End_A)$ $\IBL_\infty$-algebra &
		$\left(C,\{\frakp_{k,l,g}\}_{k,l\in \NN^*, g\in\NN}\right)$\\
		\hline
		\hline 
		\rule{0pt}{10pt} $f: \calS^c \boxtimes A \to \S B$&
		$\{\frakf_{k.l.g}\}: EC^+ \to EC^-$\\
		\hline
		\rule{0pt}{10pt} $f_{(0)}: A \to B$&
		$\frakf_{1,1,0}: C^+ \to C^-$\\
		\hline
		\rule{0pt}{10pt}  $f \rhd \alpha + f\circ d_{\S A} $& 
		the connected part of $e^\frakf\hat{\frakp}^+$\\
		\hline
		\rule{0pt}{10pt} $ f \rhd \alpha - \beta \lhd f  =\partial^A_B \circ f $
		&
		the connected part of $e^\frakf\hat{\frakp}^+ - \hat{\frakp}^- e^\frakf =0$ \\
		\rule{0pt}{10pt} (see \Cref{prop:infmor})& (see \cite[Definition 2.9]{CFL11}) \\
		\hline
	\end{tabular}
\end{center}

\medskip

Under this correspondence, Theorem~1.2  of  \cite{CFL11}, dealing with the ```quasi''-invertibility of $\infty$-quasi-isomorphisms,  is a  particular case  of \Cref{thm:InvInfQI}. Notice that the proof for the associativity relation for the composite of $\infty$-morphisms in \cite[Remark~2.13]{CFL11} is cumbersome and thus not given there. With the present approach, such a relation is automatic. 

\begin{remark}
Regarding the last two lines of the above table, terms involving disconnected graphs appear in \cite{CFL11}. However the authors show in \cite[Lemma 2.10]{CFL11} that their condition can be actually be reduced to connected compositions. In the present paper, the condition  involves only connected graphs. Such a phenomenon can be explained conceptually as follows. The Koszul dual coproperad $\IBL^{\ac}\cong \calS^c$ admits a cocommutative coalgebra structure,  
with respect to the monoidal product $\ot$, given by 
\[
c_{k,l,g}\mapsto \sum_{{\substack{
						k'+k''=k \\
						l'+l''=l\\
 						g'+g''=g 
					}}} 
					(-1)^{(l''-1)(k'+2g'+1)}\, c_{k',l',g'} \ot c_{k'',l'',g''}\ .
\]
This cocommutative coproduct induces the commutative product on the convolution algebra $\Hom_{\Sy}(\allowbreak\IBL^{\ac} , \End_A)$ used in  \cite{CFL11}. Notice that, up to suspensions, this latter space is isomorphic to the quantized space of polynomial functions $\widehat{\S}(A\oplus A^*)[[\hbar]]$  on the symplectic manifold $A\oplus A^*$, see \cite{DCTT10} for more details. Under this identification, the above-mentioned commutative product is nothing but the product of functions, structure also known as the Weyl algebra. 
In the present language, this extra algebraic structure allows the authors of \cite{CFL11} to work directly with $\calS^c\boxtimes A$ instead of $\S(\calS^c\boxtimes A)$. 
Notice however that such a richer algebraic structure is specific to the case $\calP=\IBL$ and that it unfortunately does not hold in general for a quadratic properad $\calP$. 
\end{remark}

	In \cite[Section 3]{CFL11}, the authors develop an obstruction theory with many aspects. First, applied to the extension of a linear map between $\IBL_\infty$-gebras into an $\infty$-morphism, this corresponds to the present obstruction theory of  \cref{sec:Obstruction}. We use an induction on the weight $n$, which is equal to $k+l+2g-2$ in the notations of \emph{loc. cit.}. The authors use a similar induction on triples $(k,l,g)$ for an ordering mainly based on the natural ordering of the integers $k+l+2g$. Then, one part of the obstruction theory of \cite[Proposition~3.1]{CFL11} provides us with the existence of a transferred homotopy involutive Lie bialgebra structure on the underlying homology groups of a homotopy involutive Lie bialgebra. This form of the homotopy transfer theorem \cite[Theorme~1.3]{CFL11} is a  particular case of \Cref{thm:HTT}. Since its proof in \emph{loc. cit.} is based on obstruction theory, one can only get inductive formulas. 
On the other hand, the present algebraic approach gives us the following general explicit closed formula. 

\begin{theorem}[Homotopy transfer theorem for $\IBL_\infty$-gebras]\label{thm:HTTIBL}
Let $\{\mu_{k,l,g}\}$ be a homotopy involutive Lie bialgebra structure on a chain complex $A$. The transferred homotopy involutive Lie bialgebra structure $\{\nu_{k,l,g}\}$ on any contraction $H$ is given by
\[
\nu_{k,l,g}=\sum_{\gg\in \Gs^{\mathrm{lev}}_{k,l,g}} \varepsilon_\gg\,  
\gg\left(\PP, \mu_{k_1,l_1,g_1}, \H, \ldots, \H, \mu_{k_N,l_N,g_N}, \II  \right)\ ,
\]
where the sum runs over directed connected leveled graphs with $k$ inputs and $l$ outputs and with vertices indexed by non-negative integers $g_1, \ldots, g_N$ satisfying $g_1+\cdots+g_N+\mathrm{genus}(\gg)=g$ and 
where the sign $\varepsilon_\gg$ is equal to the sign obtained when composing the graph $\gg$ into the element $c_{k,l,g}$ in the suspension properad $\calS$. 

The same formula, replacing the label $p$, on the output leaves, by $h$, gives  
the $\infty$-quasi-isomorphism from $H$ to $A$ extending $i$. And the same formula,  with a sign change, 
replacing the label $i$, on the input leaves, by $h$, gives  
the $\infty$-quasi-isomorphism from $A$ to $H$ extending $p$.
\end{theorem}

\begin{figure}[h!]
\[\nu_{5,4,6}=\sum_{\gg\in \Gs^{\mathrm{lev}}_{5,4,6}}\ -\ 
		\begin{tikzpicture}[scale=0.9,baseline=(n.base)]
			\node (n) at (1,1.5) {}; %point base
			\draw[fill=black] (0,0) circle (3pt) node[right] {$\, \scriptstyle{\mu_{2,3,1}}$};
			\draw[fill=black] (1,1) circle (3pt) node[right] {$\, \scriptstyle{\mu_{2,1,0}}$};
			\draw[fill=black] (-1,2) circle (3pt) node[right] {$\, \scriptstyle{\mu_{3,2,2}}$};
			\draw[fill=black] (0,3) circle (3pt) node[right] {$\, \scriptstyle{\mu_{3,3,1}}$};
			\draw[thick] (0,0) to[out=60,in=270] (1,1) to[out=120,in=280] (0,3) to[out=320,in=90] (1,1) ;
			\draw[thick] (0,0) to[out=120,in=270] (-1,2) to[out=40,in=240] (0,3);
			\draw[thick] (-1,2) to[out=60,in=270] (-1,4) node[fill=white] {$\scriptstyle{i}$};
			\draw[thick] (-1,2) to[out=120,in=270] (-2,4) node[fill=white] {$\scriptstyle{i}$};
			\draw[thick] (-1,2) to[out=250,in=90] (-1.5,-1) node[fill=white] {$\scriptstyle{p}$};
			\draw[thick] (0,3) to[out=100,in=270] (0,4) node[fill=white] {$\scriptstyle{i}$};
			\draw[thick] (0,3) to[out=60,in=270] (1,4) node[fill=white] {$\scriptstyle{i}$};
			\draw[thick] (0,3) to[out=120,in=270] (-0.5,4) node[fill=white] {$\scriptstyle{i}$};
			\draw[thick] (0,0) to[out=250,in=90]  (-0.8,-1) node[fill=white] {$\scriptstyle{p}$};
			\draw[thick] (0,0) to[out=290,in=90]  (0.8,-1) node[fill=white] {$\scriptstyle{p}$};			
			\draw[thick] (0,0) to[out=260,in=90]  (0,-1) node[fill=white] {$\scriptstyle{p}$};			
			\draw[fill=white] (-2.1,0.3) rectangle (1.5,0.7);
			\draw (-0.3,0.5) node {$\scriptstyle{{h}}_3$};
			\draw[fill=white] (-2.1,1.3) rectangle (1.5,1.7);
			\draw (-0.3,1.5) node {$\scriptstyle{{h}}_4$};
			\draw[fill=white] (-2.1,2.3) rectangle (1.5,2.7);
			\draw (-0.3,2.5) node {$\scriptstyle{{h}}_5$};
		\end{tikzpicture}\]
	\caption{An element in the transferred $\IBL_\infty$-structure.}
		\label{fig:TransStr}
\end{figure}
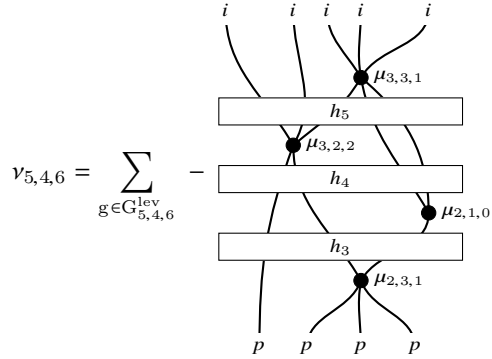 

\begin{proof}
This is a direct application of \cref{thm:HTT}. Recall that if we denote the original structure by a twisting morphism $\alpha$, then the transferred structure is given by the following composite
	\[
	\begin{tikzcd}[column sep=normal]
	\overline{\calS^c}\arrow[r,"\Delta_{\overline{\calS^c}}"] & \G^c\big(\overline{\calS^c}\big) \arrow[r,"\G^c(s\alpha)"] & \G^c\big(s\End_A\big) 
	 \arrow[r,"\lev"] & \GLev\big(s\End_A\big)  \arrow[r,"\PHI"] & \End_H\ .
	\end{tikzcd}
	\]
Applied to a basis element $c_{k,l,g}$, one gets the above mentioned formula by (iterating) Point~(3) of \cref{prop:IBLac}. 
Since $s\alpha$ is of degree $0$ and since the map $\PHI$ does not produce any extra sign, the only sign comes the comonadic decomposition map of the suspension coproperad $\calS^c$ followed by the levelisation map. 
\end{proof}

\cref{thm:HTTIBL} can be applied to all the examples given in \cite{CFL11}. For instance, it produces  explicit formulas 
for the homotopy involutive Lie bialgebra structure on the dual cyclic bar construction 
of a finite dimensional cyclic differential graded associative algebra \cite[Theorem~1.7]{CFL11} and 
of the de Rham cohomology of a closed oriented manifolds \cite[Theorem~1.11]{CFL11}. 
We hope that such explicit formulas will allow one to solve the conjectures mentioned in the introduction of \emph{loc. cit.} since one might be able to solve this way some divergence issues present there. As explained in the introduction of \cite{CFL11}, ``this should give explicit formulas for $\IBL_\infty$-gebra structure on $S^1$-equivariant symplectic cohomology, which  is essentially equivalent to symplectic field theory''.

\subsection{Other examples}
Without going into further details here, let us mention where the present properadic homotopical calculus can be fruitfully applied. The other known Koszul properads are the ones encoding respectively 

\begin{itemize}
\item[$\diamond$] (involutive) Frobenius bialgebras which play a key role in Poincar\'e duality \cite{Yalin18}, 

\item[$\diamond$] infinitesimal bialgebras \cite{Aguiar00} which is the structure carried by the classical bar construction of associative algebras \cite[Section~2.2.1]{LodayVallette12}, 

\item[$\diamond$] double Lie and double Poisson bialgebras \cite{Leray19protoII} which induce 
Poisson structures in non-commutative algebraic geometry \cite{VDB08} and which play a structural role in Lie theory (Goldman bracket and Turaev cobracket on surfaces of genus zero and the Kashiwara--Vergne problem) \cite{AKNN18}, 

\item[$\diamond$] quantum Airy structures \cite{KontsevichSoibelman18}, see also \cite{ABCO17}, which faithfully encodes the algebraic structures present in topological recursion according to Kontsevich--Soibelman. 
\end{itemize}

Notice that the last properad has not yet been proved to be Koszul, but this is the subject of a forthcoming paper.

\medskip

The classical notion of a (associative and coassociative) bialgebra is encoded by a properad $\mathrm{BiAss}$ which fails to be Koszul since its presentation is  not quadratic. As a consequence, its minimal model is so far out of reach, see \cite[Section~3.3]{MerkulovVallette09II}. Even if one could make this latter one explicit, its space of generators would form a homotopy coproperad, so one cannot expect from this a ``strict'' category but rather an 
 $\infty$-category of ``homotopy bialgebras''.
Instead, one can define a suitable  notion of \emph{a homotopy bialgebra} as a gebra over the bar-cobar resolution $\Omega \B\, \mathrm{BiAss}$. Applying the results of the present paper, we get automatically a notion of an $\infty$-morphism and thus a category of homotopy bialgebras together with their homotopy properties like the homotopy transfer theorem. With this approach, it would be interesting to study related seminal problems like the deformation quantization of Lie bialgebras as done in \cite{MerkulovWillwacher16}. 

\bibliographystyle{alpha}
\bibliography{bib}

\end{document}